\documentclass{amsart}

\usepackage{amsmath}
\usepackage{amsfonts}
\usepackage{amssymb}
\usepackage{amsthm}
\usepackage{comment}
\usepackage{epsfig}
\usepackage{psfrag}
\usepackage{mathrsfs}
\usepackage{amscd}
\usepackage[all]{xy}
\usepackage{rotating}
\usepackage{lscape}
\usepackage{amsbsy}
\usepackage{verbatim}
\usepackage{moreverb}
\usepackage{color}
\usepackage{bbm}
\usepackage{eucal}

\newtheorem{theorem}{Theorem}[subsection]
\newtheorem{theorem*}{Theorem}
\newtheorem{prop}[theorem]{Proposition}
\newtheorem{lemma}[theorem]{Lemma}
\newtheorem{cor}[theorem]{Corollary}

\theoremstyle{definition}
\newtheorem{definition}[theorem]{Definition}
\newtheorem{lem/def}[theorem]{Lemma/Definition}

\newtheorem{notation}[theorem]{Notation}

\newtheorem{convention}[theorem]{Convention}

\newtheorem{remark}[theorem]{Remark}
\newtheorem{example}[theorem]{Example}
\newtheorem{q}[theorem]{Question}

\newtheorem{observation}[theorem]{Observation}
\newtheorem{terminology}[theorem]{Terminology}

\theoremstyle{remark}

\definecolor{orange}{rgb}{1,0.5,0}
\definecolor{light-gray}{gray}{0.75}
\definecolor{brown}{cmyk}{0, 0.8, 1, 0.6}
\definecolor{plum}{rgb}{.5,0,1}

\DeclareMathOperator{\uno}{\mathbbm{1}}

\DeclareMathOperator*{\colim}{{\sf colim}}
\DeclareMathOperator*{\limit}{{\sf lim}}

\DeclareMathOperator{\Fun}{\sf Fun}

\DeclareMathOperator{\Map}{\sf Map}

\DeclareMathOperator{\Cat}{\sf Cat}

\DeclareMathOperator{\id}{\sf id}

\DeclareMathOperator{\Alg}{\mathsf{Alg}}
\DeclareMathOperator{\CAlg}{\mathsf{CAlg}}

\DeclareMathOperator{\op}{\mathsf{op}}

\DeclareMathOperator{\bsc}{\cB{\sf sc}}
\DeclareMathOperator{\Bsc}{\cB{\sf sc}}

\DeclareMathOperator{\Snglr}{\cS{\sf nglr}}

\DeclareMathOperator{\bdelta}{\boldsymbol{\Delta}}

\DeclareMathOperator{\Top}{\mathsf{Top}}

\DeclareMathOperator{\Emb}{\mathsf{Emb}}

\DeclareMathOperator{\Ar}{\mathsf{Ar}}

\DeclareMathOperator{\lag}{\langle}
\DeclareMathOperator{\rag}{\rangle}

\DeclareMathOperator{\Spaces}{\mathsf{Spaces}}
\DeclareMathOperator{\Spectra}{\mathsf{Spectra}}

\DeclareMathOperator{\Mfld}{\cM\mathsf{fld}}
\DeclareMathOperator{\mfld}{\cM\mathsf{fld}}
\DeclareMathOperator{\ZMfld}{\cZ\cM\mathsf{fld}}
\DeclareMathOperator{\zmfld}{\cZ\cM\mathsf{fld}}

\DeclareMathOperator{\zdisk}{\cZ\cD\mathsf{isk}}
\DeclareMathOperator{\ZDisk}{\cZ\cD\mathsf{isk}}

\DeclareMathOperator{\ZEmb}{\mathsf{ZEmb}}

\DeclareMathOperator{\fr}{\sf fr}

\DeclareMathOperator{\diskd}{\sf Disk}
\DeclareMathOperator{\snglrd}{\sf Snglr}

\DeclareMathOperator{\Sing}{\mathsf{Sing}}

\DeclareMathOperator{\BO}{\sf BO}

\def\ot{\otimes}

\DeclareMathOperator{\oo}{\infty}

\DeclareMathOperator{\disk}{\cD{\sf isk}}
\DeclareMathOperator{\Disk}{\cD{\sf isk}}
\DeclareMathOperator{\bBar}{\sf Bar}
\DeclareMathOperator{\cBar}{\sf cBar}
\DeclareMathOperator{\cAlg}{\sf cAlg}

\DeclareMathOperator{\Psh}{\mathsf{PShv}}

\DeclareMathOperator{\Mod}{\sf Mod}
\DeclareMathOperator{\Fin}{\sf Fin}

\DeclareMathOperator{\Ran}{{\sf Ran}}

\DeclareMathOperator{\ZTop}{\cZ\cT{\sf op}}

\newcommand{\ra}{\rightarrow}
\newcommand{\la}{\leftarrow}
\newcommand{\xra}{\xrightarrow}
\newcommand{\xla}{\xleftarrow}

\newcommand{\ov}{\overline}

\newcommand{\w}{\widetilde}

\newcommand{\tr}{\triangleright}
\newcommand{\tl}{\triangleleft}

\def\cA{\mathcal A}\def\cB{\mathcal B}\def\cC{\mathcal C}\def\cD{\mathcal D}
\def\cE{\mathcal E}\def\cF{\mathcal F}\def\cG{\mathcal G}
\def\cI{\mathcal I}\def\cJ{\mathcal J}\def\cK{\mathcal K}
\def\cM{\mathcal M}\def\cO{\mathcal O}
\def\cS{\mathcal S}\def\cT{\mathcal T}
\def\cV{\mathcal V}\def\cX{\mathcal X}
\def\cZ{\mathcal Z}

\def\DD{\mathbb D}
\def\HH{\mathbb H}

\def\RR{\mathbb R}\def\SS{\mathbb S}

\def\sB{\mathsf B}\def\sC{\mathsf C}\def\sD{\mathsf D}
\def\sH{\mathsf H}

\def\sN{\mathsf N}\def\sO{\mathsf O}

\def\bH{\mathbf H}

\def\bDelta{\mathbf\Delta}

\begin{document}

\title{Zero-Pointed Manifolds}
\author{David Ayala \& John Francis}
\date{}
\address{Department of Mathematics\\Montana State University\\Bozeman, MT 59717}
\email{david.ayala@montana.edu}
\address{Department of Mathematics\\Northwestern University\\Evanston, IL 60208-2370}
\email{jnkf@northwestern.edu}
\thanks{DA was partially supported by ERC adv.grant no.228082, and by the National Science Foundation under Award 0902639 and Award 1507704. JF was supported by the National Science Foundation under Award 1207758 and Award 1508040.}

\begin{abstract} 
We formulate a theory of pointed manifolds, accommodating both embeddings and Pontryagin--Thom collapse maps, so as to present a common generalization of Poincar\'e duality in topology and Koszul duality in $\cE_n$-algebra.

\end{abstract}



\maketitle

\tableofcontents

\section*{Introduction}

In this work, we introduce zero-pointed manifolds as a tool to solve two apparently separate problems. The first problem, from manifold topology, is to generalize Poincar\'e duality to factorization homology; the second problem, from algebra, is to show the Koszul self-duality of $n$-disk, or $\cE_n$, algebras. The category of zero-pointed manifolds can be thought of as a minimal home for manifolds generated by two kinds of maps, open embeddings and Pontryagin--Thom collapse maps of open embeddings. In this work, we show that this small formal modification of manifold topology gives rise to an inherent duality. Before describing zero-pointed manifolds, we recall these motivating problems in greater detail. 

\smallskip

Factorization homology theory, after Lurie \cite{HA}, is a comparatively new area, growing out of ideas about configuration spaces from both conformal field theory and algebraic topology. Most directly, it is a topological analogue of Beilinson \& Drinfeld's algebro-geometric factorization algebras of \cite{bd}. In algebraic topology, factorization homology generalizes both usual homology and the labeled configuration spaces of Salvatore \cite{salvatore} and Segal \cite{segallocal}. See \cite{fact} for a more extended introduction. The last few years has seen great activity in this subject, well beyond the basic foundations laid in \cite{HA}, \cite{fact}, and \cite{aft2}, including Gaitsgory \& Lurie's application of algebro-geometric factorization techniques to Tamagawa numbers in \cite{tamagawa}, and Costello \& Gwilliam's work on perturbative quantum field theory in \cite{kevinowen}, where Costello's renormalization machine is made to output a factorization homology theory, an algebraic model for the observables in a quantum field theory. 

\smallskip

One can ask if these generalized avatars of homology carry a form of Poincar\'e duality. An initial glitch in this question is that factorization homology is only covariantly natural with respect to open embeddings of manifolds, and one cannot formulate even usual Poincar\'e duality while only using pushforwards with respect to embeddings. One can then ask, en route to endowing factorization homology with a form of duality, as to the minimal home for manifold topology for which just usual Poincar\'e duality can be formulated. That is, a homology theory defines a covariant functor from $\Mfld_n$, $n$-manifolds with embeddings; a cohomology theory likewise defines a contravariant functor from $\Mfld_n$. For the formulation of duality results, what is the common geometric home for these two concepts?

\smallskip

As one answer to this question, in {\bf Section 1} we define zero-pointed manifolds. Our category $\ZMfld_{n}$ consists of pointed locally compact topological spaces $M_\ast$ for which the complement $M:=M_\ast \smallsetminus \ast$ is an $n$-manifold; every example of which is of the form $\overline M/\partial \overline M$, the quotient of an $n$-manifold with compact boundary by its boundary. The interesting feature of this category is the morphisms: a morphism between zero-pointed manifolds is a pointed map $f:M_\ast \ra N_\ast$ such that the restriction away from the zero-point, $f^{-1}N\ra N$, is an open embedding. This category $\ZMfld_n$ contains both $\mfld_n$ and $\mfld_n^{\op}$, the first by adding a disjoint basepoint and the second by 1-point compactifying. A functor from $\ZMfld_n$ thus has both pushforwards and extensions by zero, and both homology and cohomology can be thought of as covariant functor from $\ZMfld_n$. Lemma \ref{opposite} implies an isomorphism $\neg: \ZMfld_n\cong \ZMfld_n^{\op}$ between the category of zero-pointed $n$-manifolds and its own opposite, which presages further duality. 

\smallskip

In {\bf Section 3}, we extend the notion of factorization homology to zero-pointed manifolds. This gives a geometric construction of additional functorialities for factorization homology with coefficients in an augmented $n$-disk algebra. Namely, there exist extension-by-zero maps. In particular, the factorization homology \[\int_{(\RR^n)^+} A\] has the structure of an $n$-disk coalgebra via the pinch map, where $(\RR^n)^+$ is the 1-point compactification of $\RR^n$. By identifying the factorization homology of $(\RR^n)^+$ with the $n$-fold iterated bar construction, we arrive at an $n$-disk coalgebra structure on the $n$-fold iterated bar construction, or topological Andr\'e--Quillen homology, of an augmented $n$-disk algebra.

\smallskip

This construction is closely bound to the Koszul self-duality of the $\cE_n$ operad, first conceived by Getzler \& Jones in \cite{getzlerjones} contemporaneously with Ginzburg \& Kapranov's originating theory of \cite{gk}. Namely, it has long been believed that the operadic bar construction $\bBar\cE_n$ of the $\cE_n$ operad is equivalent to an $n$-fold shift of the $\cE_n$ co-operad. This is interesting because the bar construction extends to a functor $\Alg_{\cO}^{\sf aug} \longrightarrow \cAlg_{\bBar\cO}^{\sf aug}$ from augmented $\cO$-algebras to augmented coalgebras. If one stably identifies $\bBar\cE_n$ and a shift of $\cE_n$, then one can construct a functor
\[\Alg_{\cE_n}^{\sf aug}\longrightarrow \cAlg_{\cE_n}^{\sf aug}\]
from augmented $\cE_n$-algebras to augmented $\cE_n$-coalgebras. We construct exactly such a functor using this zero-pointed variant of factorization homology, which is given by taking the factorization homology of the pointed $n$-sphere $(\RR^n)^+$. In order to reduce from $n$-disk algebras to $\cE_n$-algebras, we use the framed variant of the theory which is a special case of theory of structured zero-pointed manifolds developed at the end of {\bf Section 1}.

\smallskip

A construction of such a functor has been previously accomplished by other means. Fresse performed the chain-level calculation of the Koszul self-duality of $\sC_\ast(\cE_n, R)$ in \cite{fressekoszul}. A direct calculation of self-duality of the bar construction of the operad $\cE_n$ in spectra has not yet been given. The construction of a functor as above was however accomplished in full generality by Lurie in \cite{dag10} using a formalism for duality given by twisted-arrow categories. We defer a comparison of our construction and theirs to another work; we will not need to make use of any comparison in this work or its sequel \cite{pkd}.
\smallskip

A virtue of our construction is that it is easy, geometric, and for our purposes accomplishes more via the connection to factorization homology. That is, in {\bf Section 4} we use this geometry to construct the Poincar\'e/Koszul duality map. Given a functor $\cF$ taking values on zero-pointed $n$-manifolds $M_\ast$, we obtain maps
\[\int_{M_\ast} \cF(\RR^n_+) \longrightarrow \cF(M_\ast) \longrightarrow \int^{M_\ast}\cF\bigl((\RR^n)^+\bigr)~.\]
The left hand map is a universal left approximation by a factorization homology theory; the right hand map is a universal right approximation by a factorization cohomology theory. The composite map is the Poincar\'e/Koszul duality map. While the operadic approach to constructing the functor from $\cE_n$-algebras to $\cE_n$-coalgebras requires one to work stably, such as in chain complexes or spectra, factorization homology applies unstably: in the case in which $\cF$ is a functor to spaces, the Poincar\'e/Koszul duality map generalizes the scanning maps of McDuff \cite{mcduff} and Segal \cite{segal}, as well as \cite{bodig}, \cite{kallel}, and \cite{salvatore}, which arose in the theory of configuration space models of mapping spaces. Finally, we prove in Theorem \ref{NAPD} that the Poincar\'e/Koszul duality map is an equivalence for a bicomplete Cartesian-sifted target. In particular, this gives a new proof of the non-abelian Poincar\'e duality of Lurie in \cite{HA}. Our result further specializes to a version of linear Poincar\'e duality, which assures that our duality map is an equivalence in the case of a stable $\oo$-category with direct sum; in this last case, our Poincar\'e/Koszul duality map becomes the Poincar\'e duality map of \cite{dww}.

\smallskip

We summarize this discussion by stating our main result in a simplified articulation (see Theorem~\ref{NAPD} for a precise articulation).
\begin{theorem*}[Main Theorem]\label{main.theorem}
Let $\cV$ be a symmetric monoidal $\infty$-category with suitable colimits and limits.  
Let $M$ be a compact framed $n$-manifold.
For each augmented $\cE_n$-algebra $A$ in $\cV$, there is a canonical \emph{Poincar\'e duality} morphism in $\cV$,
\[
{\sf PD}\colon 
\int_M A
\longrightarrow
\int^M \bBar^n(A)~,
\]
from the factorization homology over $M$ of $A$, to the factorization cohomology over $M$ of its $n$-fold Bar construction.  
Under favorable conditions on $\cV$, if $A$ is a member of a Koszul duality, then this Poincar\'e duality morphism is an equivalence.

\end{theorem*}

\begin{remark}

Theorem~\ref{main.theorem} can be generalized in two notable ways.
\begin{itemize}

\item
First, the tangential structure of a framing on $M$ can be replaced by any tangential structure $B$ on smooth $n$-manifolds.  In this replacement, the $\cE_n$-algebra $A$ is replaced by that of a $\Disk_n^B$-algebra, in the sense of~\cite{fact}.  In the case that $B=\sB G$ for $G\xra{\rho} {\sf GL}_n(\RR)$ a representation of a topological group $G$, a $\Disk_n^{\sB G}$-algebra is an $\cE_n$-algebra that is fixed with respect to the resulting action of $G$ on the $\infty$-category of $\cE_n$-algebras.  
The $n$-fold Bar construction on an augmented $\Disk_n^B$-algebra has the structure of a coaugmented $\Disk_n^B$-coalgebra.  
Through these modifications, the Poincar\'e duality morphism still exists canonically, as does the equivalence given a Koszul duality.

\item
Second, the $B$-manifold $M$ can be replaced by a compact $B$-manifold with boundary $\ov{M}$ together with a partition of the connected components of its boundary 
$\partial \ov{M} = \partial_-\ov{M} \amalg \partial_+ \ov{M}$.
Under these modifications the Poincar\'e duality morphism becomes
\[
{\sf PD}\colon 
\int^{\sf red}_{\ov{M} \smallsetminus \partial_+ \ov{M}} A
\longrightarrow
\int^{\ov{M}\smallsetminus \partial_-\ov{M}}_{\sf red}
\cBar^n(A)
~,
\]
involving \emph{reduced} factorization (co)homology, which factorization (co)homology (as developed in~\S\ref{sec.reduced}) with boundary conditions at $\uno$.  
As with many proofs of Poincar\'e duality, this non-compact phrasing is essential for the logic establishing the compact version.

\end{itemize}

\end{remark}

\begin{remark}
As indicated above, and detailed in~\S\ref{sec.duality}, 
in the case that $(\cV,\ot) = (\Spaces,\times)$, Theorem~1 generalizes non-abelian Poincar\'e duality of~\cite{HA}; in the case that $(\cV,\ot) = (\Spectra , \oplus)$, Theorem~1 generalizes Atiyah duality.  

\end{remark}

\begin{remark}
Theorem~1 does \emph{not} easily apply to the case $(\cV,\ot)=(\Mod_{\Bbbk},\ot)$ of chain complexes and tensor products over $\Bbbk$, which is a case of notable interest.
We devote a follow-up work,~\cite{pkd}, to this case and give an algebro-geometric interpretation of the result.  
	
\end{remark}

\smallskip

We note that the most appealing aspects of this work, such as the notions of zero-pointed manifolds and their basic properties, are not difficult. The comparatively technical stretch of this paper lies in {\bf Section 3}, where we show that factorization homology of zero-pointed manifolds is well-defined and well-behaved. The mix of $\oo$-category theory and point-set topology around the zero-point introduces bad behavior in $\Disk_{n,+/M_\ast}$, the $\oo$-overcategory appearing in the definition of factorization homology. Consequently, we make recourse to a hand-crafted auxiliary version of this disk category, $\Disk_+(M_\ast)$. This adaptation has two essential features. First, $\Disk_+(M_\ast)$ is sifted, a property which is necessary to show that factorization homology exists as a symmetric monoidal functor. Second, $\Disk_+(M_\ast)$ has a natural filtration by cardinality of embedded disks. This is an essential feature which $\Disk_{n,+/M}$ lacks. This cardinality filtration gives rise to highly nontrivial filtrations on factorization homology. This generalizes the Goodwillie--Weiss embedding calculus of \cite{weiss} to functors on zero-pointed manifolds or, alternatively, to those functors on manifolds with boundary which are reduced on the boundary. In the case $n=1$, in which case factorization homology of the circle is Hochschild homology, our cardinality filtration further specializes to the Hodge filtration on Hochschild homology developed by Burghelea--Vigu-Poirrier \cite{burvi}, Feigin--Tsygan \cite{feigintsygan1}, Gerstenhaber--Schack \cite{gs}, and Loday \cite{loday}; for general spaces, but still in the case of commutative algebras, our filtration specializes to the Hodge filtration of Pirashvili \cite{pirashvili} and Glasman \cite{glasman}.

\smallskip

These small technical modifications involved in the construction of the auxiliary $\Disk_+(M_\ast)$ play an essential role in the sequel \cite{pkd}. An essential step therein shows that the Poincar\'e/Koszul duality map interchanges the cardinality filtration and the Goodwillie tower. That is, Goodwillie calculus and Goodwillie--Weiss calculus are Koszul dual in this context, a feature we ultimately use to present one solution as to when the Poincar\'e/Koszul duality map is an equivalence.

\subsubsection*{\bf Implementation of $\infty$-categories}
In this work, we use Joyal's {\it quasi-category} model  of $\oo$-category theory \cite{joyal}. 
Boardman and Vogt first introduced these simplicial sets in \cite{bv}, as weak Kan complexes, and their and Joyal's theory has been developed in great depth by Lurie in~\cite{HTT} and~\cite{HA}, our primary references; see the first chapter of~\cite{HTT} for an introduction. We use this model, rather than model categories or simplicial categories, because of the great technical advantages for constructions involving categories of functors, which are ubiquitous in this work.

More specifically, we work inside of the quasi-category associated to this model category of Joyal's.  In particular, each map between quasi-categories is understood to be an iso- and inner-fibration; (co)limits among quasi-categories are equivalent to homotopy (co)limits with respect to Joyal's model structure.
As we work in this way, we refer the reader to these sources for $\infty$-categorical versions of numerous familiar results and constructions among ordinary categories.  

We will also make use of topological categories, by which we mean categories enriched in the Cartesian category of compactly generated weak Hausdorff topological spaces.
A key example of such is the topological category $\Mfld_n$, of $n$-manifolds and embeddings among them. 

By a functor $\cS \ra \cC$ from a topological category to an $\oo$-category $\cC$ we will always mean a functor $\sN\Sing \cS \ra \cC$ from the coherent nerve of the ${\sf Kan}$-enriched category obtained by applying the product preserving functor $\Sing$ to the morphism topological spaces. 

The reader unfamiliar with this language can substitute the words ``topological category" for ``$\oo$-category" wherever they occur in this paper to obtain the correct sense of the results, but they should then bear in mind the proviso that technical difficulties may then abound in making the statements literally true. The reader only concerned with algebras in chain complexes, rather than spectra, can likewise substitute ``pre-triangulated differential graded category" for ``stable $\oo$-category" wherever those words appear, with the same proviso.

\medskip

\noindent
{\bf Acknowledgements.} 
We thank the the referees, whose feedback improved this paper both in exposition, and by identifying mistakes and suggesting corrections thereto. In particular, the use of Siebenmann's open-track lemma, in the proof of Lemma~\ref{neg-maps}, was suggested by a referee.

\section{Zero-pointed spaces}

\noindent
For this section, we use the letters $X$, $Y$, and $Z$ for locally compact Hausdorff topological spaces.
For $W$ a topological space, we denote the coproduct $W_+:=W\amalg \{\ast\}$ in topological spaces.

\subsection{Pointed extensions and negation}
We define \emph{pointed extensions} and \emph{negations} thereof.

\begin{definition}\label{pointed-extension}
A pointed extension $X_\ast$ of $X$ is a locally compact Hausdorff topology on the underlying set of $\ast\amalg X$ extending the given topology on $X$.  
The category of \emph{pointed extensions (of $X$)} is the full subcategory
\[
\mathsf{Point}_X~ \subset~ \Top^{X_+/}
\]
of the undercategory consisting of the pointed extensions of $X$.  
\end{definition}

\begin{example}
The coproduct $X_+$ is a pointed extension of $X$, and it is initial in the category ${\sf Point}_X$.
The 1-point compactification $X^+$ is a pointed extension of $X$, and it is final in the category ${\sf Point}_X$.

\end{example}

\begin{example}\label{with.boundary}

Let $\ov{M}$ be a topological manifold with boundary $\partial \ov{M}$.  
Provided the boundary $\partial \ov{M}$ is compact, the pushout
\[
\ast\underset{\partial \ov{M}}\amalg \ov{M}
\]
is a pointed extension of the interior $M$ of $\ov{M}$.  

\end{example}

\begin{remark}\label{h-cobordism}
In the situation of Example~\ref{with.boundary}, the pointed extension $\ast\underset{\partial \ov{M}}\amalg \ov{M}$ is less information than the topological manifold with boundary $\ov{M}$.
For instance, consider two compact topological manifolds $\ov{M}$ and $\ov{M}'$ with boundary together with a homeomorphism $M\cong M'$ between their interiors.  
The pointed extensions $\ast\underset{\partial \ov{M}}\amalg \ov{M} \cong M^+ \cong \ast\underset{\partial \ov{M}'}\amalg \ov{M}'$ are both identified as the 1-point compactification of $M$.  
On the other hand, it need not be the case that the identification between interiors extends as an identification $\ov{M}\cong \ov{M}'$ of manifolds with boundary.

For instance, consider two topological manifolds $P$ and $Q$ that are h-cobordant but not homeomorphic (\cite{farrellhsiang} produces examples of such).
For $H$ such an h-cobordism, the Whitehead torsion of the product $S^1\times P$ vanishes.  
It follows from the topological s-cobordism theorem~(\cite{kirbysieb}, Essay 3) that the interiors of $[-1,1]\times P$ and of $[-1,1]\times Q$ are homeomorphic.

\end{remark}

\begin{observation}\label{local.topology}
A pointed extension of $X$ is determined by a basis for its topology about the base point.
More precisely, the identity map $X_\ast\to X_\ast'$ is continuous if and only if there are local bases 
\[
\cB~:=~\bigl\{\ast\in B\subset X_\ast\bigr\}
\qquad \text{ and }\qquad
\cB'~:=~\bigl\{ \ast \in B'\subset X_\ast' \bigr\}
\]
for which, for each member $B'\in \cB'$, there is a member $B\in \cB$ with $B\subset B'$.

\end{observation}

For the next result we make use of the Stone--\v{C}ech compactification $\widehat{X}$ of the locally compact Hausdorff topological space $X$.  
Specifically, we make reference to its boundary $\partial\widehat{X}:=\widehat{X}\setminus X$, which is a closed subspace of $\widehat{X}$.  
\begin{prop}\label{clopens}
The category ${\sf Point}_X$ is isomorphic to the poset ${\sf clo}(\partial \widehat{X})$ of clopen subsets of the boundary of the Stone--\v{C}ech compactification of $X$.  

\end{prop}

\begin{proof}
Each morphism $X_\ast\to X_\ast'$ in $\mathsf{Point}_X$ is necessarily a bijection under the underlying set of $X_+$.
It follows that ${\sf Point}_X$ is a poset.

Consider the assignment 
\begin{equation}\label{clo.point}
{\sf clo}(\partial \widehat{X}) \ni (S\subset \partial\widehat{X}) \mapsto \ast\underset{S}\amalg (S\cup X)\in {\sf Point}_X~;
\end{equation}
here, the union $S\cup X$ is understood as a subspace of $\widehat{X}$.
An inclusion of such clopens $S\subset S'$ determines a continuous map $\ast \underset{S}\amalg (S\cup X) \to \ast \underset{S'}\amalg (S'\cup X)$ under $X$.  
It follows that the assignment~(\ref{clo.point}) defines a map between posets.  

Consider the assignment
\begin{equation}\label{point.clo}
{\sf Point}_X\ni X_\ast \mapsto (!^{-1}(\ast) \subset \partial \widehat{X}) \in {\sf clo}(\partial \widehat{X})
\end{equation}
where $!\colon \widehat{X} \to (X_\ast)^+$ is the unique continuous map under $X$ obtained via the universal property of the Stone--\v{C}ech compactification.
Because the subspace $(X_\ast)^+\setminus X\subset (X_\ast)^+$ is a discrete subspace consisting of two points, then $!^{-1}(\ast)\subset \partial \widehat{X}$ is clopen, as required.  
Consider the diagram comprised of continuous maps under $X$:
\begin{equation}\label{!.pb}
\xymatrix{
!^{-1}(\ast)\cup X   \ar[r]  \ar[d]
&
\widehat{X}  \ar[d]^-{!}
\\
X_\ast \ar[r]
&
(X_\ast)^+.
}
\end{equation}
This is a pullback diagram. 
Now suppose the identity map $X_\ast \to X_\ast'$ between two pointed extensions of $X$ is continuous. 
There results a commutative diagram comprised of continuous maps under $X$:
\[
\xymatrix{
!^{-1}(\ast)\cup X  \ar[rr]  \ar[d]
&&
\widehat{X}  \ar[dd]^-{!}
\\
X_\ast  \ar[d]^{\id}
&&
\\
X_\ast'  \ar[rr]
&&
(X_\ast')^+.
}
\]
Because the diagram~(\ref{!.pb}), applied to $X_\ast'$, is a pullback, there results a canonical continuous map
\[
!^{-1}(\ast) \cup X \longrightarrow (!')^{-1}(\ast)\cup X
\]
under $X$ and over $\widehat{X}$; here, $!'\colon \widehat{X} \to (X_\ast')^+$ is the unique continuous map under $X$.  
In particular, there is an inclusion $!^{-1}(\ast)\subset (!')^{-1}(\ast)$.  
We conclude that the assignment~(\ref{point.clo}) is a map between posets, as desired.

We leave to the reader the verification that the assignments~(\ref{clo.point}) and~(\ref{point.clo}) are inverse to one another.

\end{proof}

The next result is a direct consequence of the fact that the poset ${\sf clo}(\partial \widehat{X})$ is a Boolean algebra. 
\begin{cor}\label{boolean}
The poset ${\sf Point}_X$ is a Boolean algebra.  
In particular, it has an initial object and a final object, it admits arbitrary colimits and limits, and products distribute over colimits.

\end{cor}

\begin{remark}

The poset ${\sf Point}_X$ is the poset, ordered by inclusion, of connected components of the topological space of \emph{ends} of $X$, introduced by Freudenthal in~\cite{freudenthal}.

\end{remark}

\begin{example}

The initial object in $\mathsf{Point}_X$ is $X_+$, the space $X$ with a disjoint based point.
The final object in $\mathsf{Point}_X$ is $X^+$, the 1-point compactification of $X$.  

\end{example}

The data of a pointed extension $X_\ast$ of $X$ determines which sequences in $X$ that leave compact subsets converge to the base point, which we think of as `infinity'.  
Consequently, one can informally contemplate the \emph{negation} $X_\ast^\neg$ of $X_\ast$ by making the complementary declaration: such a sequence belongs to $X_\ast^\neg$ if and only if it does not belong to $X_\ast$.  
In other words, $X_\ast^\neg$ is endowed with the complementary topology about $\ast$ from that of $X_\ast$. The following makes this heuristic precise.
\begin{cor}\label{neg}
There is a contravariant involution
\[
\neg\colon {\sf Point}_X^{\op} \xra{~\cong~} {\sf Point}_X\colon \neg~.  
\]
For $X_\ast$ a pointed extension of $X$, the value $X_\ast^\neg$ is the pointed extension of $X$ that is initial among pointed extensions $X_\ast'$ of $X$ for which the canonical map
\begin{equation}\label{neg.test}
X_+\longrightarrow X_\ast\underset{X^+}\times X_\ast'
\end{equation}
is a homeomorphism. 

\end{cor}

\begin{definition}\label{def.negation}
For $X_\ast$ a pointed extension of $X$, we denote the value $\neg(X_\ast)$ as $X_\ast^\neg$ and refer to this pointed extension of $X$ as the \emph{negation (of $X_\ast$)}.

\end{definition}

In other words, there is a relation $X_\ast'\ra X_\ast^\neg$ between pointed extensions of $X$ if and only if the canonical relation $X_+ \ra X_\ast\underset{X^+}\times  X_\ast'$ is an equality.

\begin{example}
The initial object and the final object are one another's negations: $X_+^\neg = X^+$ and $X_+ = (X^+)^\neg$.
\end{example}

\begin{example}\label{intervals}
The poset $\mathsf{Point}_{(-1,1)}$ can be displayed as
\[
\xymatrix{
(-1,1)_+  \ar[r]  \ar[d]
&
(-1,1]  \ar[d]
\\
[-1,1)  \ar[r]
&
(-1,1)^+.
}
\]
The contravariant involution $\neg\colon {\sf Point}_{(-1,1)}\cong {\sf Point}_{(-1,1)}^{\op}$ is the antipodal map on this square diagram.  
Note that the two intermediate terms are abstractly isomorphic, as pointed topological spaces.  

\end{example}

In the Boolean algebra ${\sf clo}(\partial \widehat{X})$, negation is given by taking complements of clopen subspaces.  
Unwinding the equivalence of Proposition~\ref{clopens} gives the next simple identification of negations in ${\sf Point}_X$.  
\begin{cor}\label{neg.one.point}
For $X_\ast$ a pointed extension of $X$, its negation is the based topological space
\[
X_\ast^\neg~:=~(X_\ast)^+\smallsetminus \ast~,
\]
which is the 1-point compactification without the base point of $X_\ast$.
In particular, the collection of subsets 
\[
\{\{\infty\}\cup (X\smallsetminus K)  \subset X_\ast^\neg \mid \ast \in K\underset{\rm compact} {X_\ast}\}
\]
is a local basis for the topology about the base point $\ast \in X_\ast^\neg$.

\end{cor}

\begin{observation}\label{mutual.boundary}
Let $\ov{X}$ be a compact Hausdorff topological space.  
Let $\partial_L , \partial_R \subset \ov{X}$ be a pair of disjoint closed subspaces.
Write $\partial \ov{X}:= \partial_L \cup \partial_R \subset \ov{X}$ and $X:= \ov{X} \smallsetminus \partial \ov{X}$.  
Consider the two pointed extensions of $X$
\[
X_\ast:=  \ast\underset{\partial_L}\amalg \bigl(\ov{X}\smallsetminus \partial_R\bigr)\qquad \text{ and }\qquad X_\ast^\neg:= \ast\underset{\partial_R}\amalg \bigl(\ov{X}\smallsetminus \partial_L\bigr)~.
\]
These two pointed extensions of $X$ are each other's negation, as the notation suggests.

\end{observation}

\begin{example}\label{top.cobordism}
Let $\ov{M}$ be a topological cobordism.
This is to say that $\ov{M}$ is a compact topological manifold with boundary $\partial \ov{M} = \partial_L \sqcup \partial_R$ which is partitioned as a coproduct.  
Observation~\ref{mutual.boundary} offers a pair of mutually negating pointed extensions $M_\ast$ and $M_\ast^\neg$ of the interior $M$.

\end{example}

\begin{prop}\label{B.obs}
Let $B$ be a compact Hausdorff topological space, and let $X_\ast$ be a pointed extension of $X$.  
The smash product $B_+\wedge X_\ast$ is a pointed extension of the product $B\times X$.
Furthermore, its negation $(B_+\wedge X_\ast)^\neg \cong B_+ \wedge X_\ast^\neg$.  

\end{prop}

\begin{proof}
The tube lemma gives that the smash product $B_+ \wedge X_\ast$ is locally compact about the base point.  
Furthermore, compactness of $B$ results in a canonical homeomorphism under $B\times X$ from the Stone--\v{C}ech compactification $\widehat{B\times X} \cong B\times \widehat{X}$.
The statement concerning negations follows.  

\end{proof}

Corollary~\ref{neg} has the following immediate consequence, which characterizes continuous based maps to negations.  
\begin{cor}\label{to.neg}
Let $X_\ast$ be a pointed extension of $X$.
For $Z_\ast$ a based compactly generated Hausdorff topological space, the subset $\Map^{\ast/}(Z_\ast,X_\ast^\neg)\subset \Map^{\ast/}(Z_\ast,X^+)$ consists of those based continuous maps $Z_\ast \to X^+$ for which the projection from the pullback factors:
\[
\xymatrix{
&
X_+  \ar[dr]
&
\\
X_\ast \underset{X^+}\times Z_\ast \ar[rr]  \ar@{-->}[ur]
&&
X^+~.
}
\]

\end{cor}

\subsection{Pointed embeddings}
We define \emph{zero-pointed embeddings} among pointed extensions.

\begin{definition}[Zero-pointed embeddings]\label{def:pointed-embeddings}
For $X_\ast$ and $Y_\ast$ pointed extensions of $X$ and $Y$, respectively, the space of \emph{zero-pointed topological embeddings} from $X_\ast$ to $Y_\ast$ is the compactly generated weak Hausdorff replacement of the subspace 
\[
\ZEmb^{\sf top}(X_\ast,Y_\ast) ~\subset~\Map^{\ast/}(X_\ast,Y_\ast)
\]
of the compact-open topology consisting of those based maps $f\colon X_\ast \to Y_\ast$ for which the restriction $f_|\colon f^{-1}Y \to Y$ is an open embedding.
\end{definition}

Corollary~5.6 of~\cite{lewis} states that composition defines a continuous map between compactly generated weak Hausdorff replacements of compact-open topologies on sets of continuous maps.  
This offers the following.  
\begin{observation}\label{zemb.compose}
For $X_\ast$, $Y_\ast$, and $Z_\ast$ locally compact Hausdorff pointed topological spaces.
Composition of zero-pointed topological embeddings defines a continuous map:
\[
\circ\colon \ZEmb^{\sf top}(X_\ast, Y_\ast) \times \ZEmb^{\sf top}(Y_\ast , Z_\ast) 
\longrightarrow
\ZEmb^{\sf top}(X_\ast , Z_\ast)
~,\qquad
(f , g)\mapsto g\circ f~.
\]

\end{observation}

\begin{example}\label{half-open}
Consider the poset of Example~\ref{intervals}.  
Consider the map
\[
f\colon [0,1]\longrightarrow  \ZEmb^{\sf top}\bigl((-1,1],(-1,1]\bigr)
\]
that evaluates as $f_t(x) = x+t$ if $x+t\leq 1$ and as $f_t(x)=1$ if $x+t\geq 1$.  
This map is a continuous path from the identity map to the constant map at the base point.  

\end{example}

\begin{example}
The only zero-pointed topological embedding $(\RR^n)^+ \to \RR^n_+$ is the constant map at the base point.  
On the other hand, evaluation at $0\in \RR^n_+$ defines a continuous retraction $\ZEmb^{\sf top}\bigl(\RR^n_+,(\RR^n)^+\bigr) \to (\RR^n)^+$; a section is given by the continuous based map $(\RR^n)^+ \to\ZEmb^{\sf top}\bigl(\RR^n_+,(\RR^n)^+\bigl)$ given as $v\mapsto (x\mapsto \frac{x}{1+\lVert x\rVert}+v)$. 

\end{example}

\begin{example}\label{pre.aug}
Let $X$ and $Y$ be locally compact Hausdorff topological spaces.
Denote their sets of connected components as $[X]$ and $[Y]$, respectively.
There is a canonical homeomorphism
\begin{equation}\label{pre.aug.hom}
\ZEmb^{\sf top}\Bigl(X_+,Y_+\Bigr)\xra{~\cong~}\coprod_{[X]_+\xra{f}[Y]_+} \prod_{Y_j\in [Y]} \Emb\bigl(f^{-1}(Y_j) , Y_j\bigr)
\end{equation}
to a coproduct indexed by the set of pointed maps between sets of connected components.  
Here, for $A$ and $B$ locally compact Hausdorff topological spaces, $\Emb(A,B)$ is the set of open embeddings from $A$ to $B$, equipped with the compact-open topology. 

\end{example}

\begin{lemma}\label{neg-maps}
Let $X_\ast$ and $Y_\ast$ be pointed extensions of $X$ and $Y$.  
Assume the topological spaces $(X_\ast)^+$ and $(Y_\ast)^+$ are locally connected. 
Negation implements a homeomorphism
\begin{equation}\label{eq.neg.map}
\neg \colon \ZEmb^{\sf top}(X_\ast,Y_\ast) \xra{~\cong~} \ZEmb^{\sf top}(Y_\ast^\neg,X_\ast^\neg)~.
\end{equation}

\end{lemma}

\begin{proof}
We first construct the map~(\ref{eq.neg.map}) between underlying sets; then we argue that this map is continuous, and is a homeomorphism.  
The map~(\ref{eq.neg.map}) assigns to a zero-pointed topological embedding $f\colon X_\ast \to Y_\ast$ the based map between underlying sets
\[
f^\neg \colon Y_\ast^\neg \longrightarrow X_\ast^\neg
\]
determined by declaring the preimage of $x\in X$ to be the subset $\{y\in Y\mid f(x) = y\}$.
Because $f$ is a zero-pointed topological embedding, for $x,x'\in X$, the intersection $\{y\in Y \mid f(x) = y\}\cap \{ y'\in Y \mid f(x')=y'\}$ is not empty if and only if $x=x'$.  
This verifies that the map $f^\neg$ is well-defined.  
We now argue that $f^\neg$ is continuous.  
In light of Corollary~\ref{neg.one.point}, we need only make the following checks:
\begin{itemize}
\item
The preimage $(f^\neg)^{-1} (K\smallsetminus \ast) \subset Y_\ast^\neg$ is closed for each compact subspace $\ast \in K\subset X_\ast$ containing the base point.  

\item 
The preimage $(f^\neg)^{-1}U \subset Y_\ast^\neg$ is open for each open subset $U\subset X$.

\end{itemize}
We examine the first case.
By definition of the map $f^\neg$, the preimage $(f^\neg)^{-1}(K\smallsetminus \ast) = \{f(K) \smallsetminus \ast \subset Y \subset Y_\ast^\neg\}$.
Because $f$ is a continuous based map, the image $f(K)\subset Y_\ast$ is a compact subspace containing the base point of $Y_\ast$.  
Therefore $f(K)\subset (Y_\ast)^+$ is a compact subspace of the 1-point compactification containing the base point $\ast \in (Y_\ast)^+$.
Thus, the complement $f(K)\smallsetminus \ast \subset (Y_\ast)^+ \smallsetminus \ast = Y_\ast^\neg$ is a closed subspace, as desired.  
We now examine the second case.
By definition of the map $f^\neg$, the preimage $(f^\neg)^{-1}(U) = \{f(U) \smallsetminus \ast \subset Y \subset Y_\ast^\neg\}$.
This preimage is the image of the restriction $f_{|}\colon U\cap f^{-1}Y \xra{f}Y \subset Y_\ast$.  
Precisely because $f$ is a zero-pointed topological embedding, this restriction is an open embedding.  
We conclude that the preimage $(f^\neg)^{-1}(U)\subset Y_\ast^\neg$ is open, as desired.  
This completes the proof that $f^\neg$ is continuous.

We now show that the map $f^\neg\colon Y_\ast^\neg \to X_\ast^\neg$ is a zero-pointed topological embedding.  
By definition of $f^\neg$, the preimage $(f^\neg)^{-1}X =  f(f^{-1} Y)$ is the image of the open embedding $f_{|}\colon f^{-1} Y \hookrightarrow Y$.
Through this identification, its restriction $(f^\neg)^{-1}X  \hookrightarrow Y_\ast^\neg \xra{f^\neg} X_\ast^\neg$ is identified as the inverse open embedding: $f^{-1} \colon f(f^{-1} Y) \xra{\cong} f^{-1} Y\subset X$.
This verifies that $f^\neg$ is a zero-pointed topological embedding, as desired.

We now prove that the map~(\ref{eq.neg.map}) is continuous.  
Both of the topological spaces $\ZEmb^{\sf top}(X_\ast,Y_\ast)$ and $\ZEmb^{\sf top}(Y_\ast^\neg , X_\ast^\neg)$ are endowed with compactly generated weak Hausdorff topologies.
Thus, continuity of the map~(\ref{eq.neg.map}) is equivalent to showing, for each continuous map $B \to \ZEmb^{\sf top}(X_\ast , Y_\ast)$ from a compact Hausdorff topological space, that the composite map 
\begin{equation}\label{27}
B\to \ZEmb^{\sf top}(X_\ast , Y_\ast) \xra{(\ref{eq.neg.map})} \ZEmb^{\sf top}(Y_\ast^\neg , X_\ast^\neg) 
\end{equation}
is continuous.  
So let $B$ be a compact Hausdorff topological space, and let $B \to \ZEmb^{\sf top}(X_\ast , Y_\ast)$ be a continuous map.  
Because the topology on $\ZEmb^{\sf top}(X_\ast , Y_\ast)$ is the compactly generated weak Hausdorff replacement of the subspace topology on the compact-open topology, this continuous map is the datum of a continuous map 
\begin{equation}\label{20}
f\colon B\times X_\ast \longrightarrow B\times Y_\ast
\end{equation}
over $B$ via projections with the following property
\begin{itemize}
\item[]
For each $b\in B$, the restriction $f_{|\{b\}\times X_\ast} \colon X_\ast \to Y_\ast = \{b\}\times Y_\ast$ is a zero-pointed topological embedding.
\end{itemize}
We now argue that the canonically associated based continuous map
\begin{equation}\label{23}
\ov{f}\colon B_+ \wedge X_\ast \longrightarrow B_+\wedge Y_\ast
\end{equation}
is, itself, a zero-pointed topological embedding.
Proposition~\ref{B.obs} gives that the domain and the codomain of $\ov{f}$ are indeed locally compact Hausdorff topological spaces.  
Next, note that the properties on $f$ immediately imply that the restriction
\begin{equation}\label{22}
\ov{f}_|\colon \ov{f}^{-1}(B\times Y) \longrightarrow B\times Y
\end{equation}
is injective.  
Being an open subspace of a product topology, for each $(b,x)\in f^{-1}(B\times Y)\subset B\times X$, there are open neighborhoods $b\in V \subset B$ and $x\in U \subset X$ for which $V\times U \subset f^{-1}(B\times Y)$.  
For such neighborhoods, consider the restriction
\begin{equation}\label{21}
f_{|}\colon V\times U \longrightarrow V\times Y~\subset ~B\times Y~,
\end{equation}
in which the first arrow lies over $V$ via projections.  
The highlighted property of $f$ implies that this first arrow in~(\ref{21}) has the property that, for each $b\in V$, the restriction $f_{|\{b\}\times U} \colon U \to Y = \{b\}\times Y$ is an open embedding.  
Because the 1-point compactification $(Y_\ast)^+$ is assumed locally connected, then so too is its open subspace $Y\subset (Y_\ast)^+$.
We can therefore apply the open-track Lemma~1.6 of~\cite{sieb}, which grants that the map~(\ref{22}) is an open embedding.  
This is to say that the map~(\ref{22}) is locally an open embedding.  
Because ths map~(\ref{22}) is injective, we conclude that it is an open embedding.  
We conclude that the based map~(\ref{23}) is a zero-pointed topological embedding, as desired.  

Now, established at the beginning of this proof is that the negation of a zero-pointed topological embedding is again a zero-pointed topological embedding.  
Applying this to the zero-pointed topological embedding $\ov{f}$ established just above gives the zero-pointed topological embedding
$
\ov{f}^{\neg}\colon (B_+\wedge Y_\ast)^\neg \to   (B_+\wedge X_\ast)^\neg~.
$
Proposition~\ref{B.obs} identifies this zero-pointed topological embedding as
\begin{equation}\label{24}
\ov{f}^{\neg}\colon B_+\wedge Y_\ast^\neg \longrightarrow B_+\wedge X_\ast^\neg~.
\end{equation}
Inspecting the definition of this negation $\ov{f}^\neg$ reveals a continuous factorization of the composition:
\[
\xymatrix{
B \times Y_\ast^\neg  \ar@{-->}[rr]^-{f^\neg}  \ar[d]
&&
B \times  X_\ast^\neg   \ar[d]
\\
B_+\wedge Y_\ast^\neg  \ar[rr]^-{\ov{f}^\neg}
&&
B_+\wedge X_\ast^\neg  ,
}
\]
in which this map $f^\neg$ lies over $B$ via projections.
This map $f^\neg$, then, is adjoint to a continuous map 
\begin{equation}\label{25}
B   \longrightarrow   \Map^{\ast/}(Y_\ast^\neg , X_\ast^\neg)
\end{equation}
to the compact-open topology.  
Because the map~(\ref{24}) is a zero-pointed embedding, and because the map $f^\neg$ above lies over $B$, this map~(\ref{25}) factors through the subset of zero-pointed topological embeddings:
\begin{equation}\label{26}
B\longrightarrow \ZEmb^{\sf top}(Y_\ast^\neg , X_\ast^\neg)~\subset~\Map^{\ast/}(Y_\ast^\neg , X_\ast^\neg)~.
\end{equation}
Finally, because $B$ is compact and Hausdorff, the definition of the topology on this subset of zero-pointed topological embeddings is just so that this factorization~(\ref{26}) is continuous.
We conclude, at last, that the composite map~(\ref{27}) is continuous.
This completes the proof that the map~(\ref{eq.neg.map}) is continuous.

\end{proof}

\subsection{Constructions}
We observe a series of constructions among pointed extensions.  
The verification of each statement therein is immediate from definitions, which we leave as an exercise for the insistent reader. 
Recall that the letters $X$, $Y$, and $Z$ denote locally compact Hausdorff topological spaces.  

\begin{observation}\label{constructions}
Fix pointed extensions $X_\ast$, $Y_\ast$, and $Z_\ast$ of $X$, $Y$, and $Z$, respectively.

\begin{enumerate}
\item[\textbf{Wedge:}]
The wedge sum $X_\ast \vee Y_\ast$ is a pointed extension of the coproduct $X\amalg Y$.
There is an equality between pointed extensions of $X\amalg Y$:
\[
X_\ast^\neg \vee Y_\ast^\neg ~=~ (X_\ast\vee Y_\ast)^\neg ~.
\]

\item[\textbf{Smash:}]
Suppose $X_\ast$, $Y_\ast$, and $Z_\ast$ have the property that the connected component containing the base point is compact.   
Each of their negations has this property as well.
The smash product $X_\ast\wedge Y_\ast := (X_\ast \times Y_\ast)/(X_\ast\vee Y_\ast)$ is a pointed extension of the product $X\times Y$.
Smash product distributes over wedge sum:
\[
X_\ast\wedge (Y_\ast \vee Z_\ast) = (X_\ast \wedge Y_\ast) \vee(X_\ast \wedge Z_\ast)~.
\]
There is an equality between pointed extensions of $X\times Y$:
\[
X_\ast^\neg \wedge Y_\ast^\neg ~=~ (X_\ast\wedge Y_\ast)^\neg~.
\]

\item[\textbf{Coinv:}]
Let $G$ be a finite group acting continuously on the topological space $X_\ast$.  
Suppose this action restricts as a free action on $X$.  
The coinvariants $(X_\ast)_{G}$ is a pointed extension of the coinvariants $X_G$.  
This action of $G$ on $X$ extends as a continuous action on $X_\ast^\neg$.
There is an equality between pointed extensions of the coinvariants $X_G$:
\[
(X_\ast^\neg)_G ~=~((X_\ast)_G)^\neg ~.
\]

\item[\textbf{Sub:}]
Let $W\subset X$ be a subspace for which the union $\ast \cup W \subset X_\ast$ is open. 
Then the subspace
\[
W_{X_\ast}~:=~\ast \cup W~\subset~X_\ast
\]
is a pointed extension of $W$.
This pointed extension has the following universal property:
\begin{itemize}
\item[~]
Let $f\colon Z_\ast \to X_\ast$ be a zero-pointed topological embedding.
Suppose $f(Z_\ast)\smallsetminus \ast \subset W\subset X$.  
Then $f$ factors through $W_{X_\ast} \to X_\ast$.  

\end{itemize}

\item[\textbf{Quot:}]
Let $W\subset X$ be a subset for which the union $\ast \cup W \subset X_\ast$ is open with compact closure.
The quotient 
\[
W^{X_\ast}~:=~X_\ast/(X_\ast \smallsetminus W_\ast)
\] 
is a pointed extension of $W$.
This pointed extension has the following universal property:
\begin{itemize}
\item[~]
Let $f\colon X_\ast\to Z_\ast$ be a zero-pointed topological embedding.
Suppose $f^{-1}(Z_\ast\smallsetminus \ast) \subset W\subset X$.  
Then $f$ factors through $X_\ast \to W^{X_\ast}$.  

\end{itemize}
There are equalities between pointed extensions of $W$:
\[
W^{X_\ast^\neg}~=~(W_{X_\ast})^\neg~{}~\text{ and }~{}~ (W^{X_\ast})^\neg ~=~W_{X_\ast^\neg}~.
\]

\end{enumerate}

\end{observation}

\subsection{Zero-pointed spaces}\label{sec.ZTop}
We define a category of \emph{zero-pointed spaces} and \emph{zero-pointed topological embeddings} among them.

Recall Definition~\ref{def:pointed-embeddings} of zero-pointed topological embeddings.  
In the next definition, we understand the enrichment as in the Cartesian monoidal category of compactly generated weak Hausdorff topological spaces.

\begin{definition}[$\ZTop$]\label{zero-definition}
The symmetric monoidal topological category 
\[
\ZTop
\]
of \emph{zero-pointed spaces} is the following.
\begin{itemize}
\item[\textbf{Ob:}] An object is a pointed locally compact Hausdorff topological space $X_\ast$ 
whose 1-point compactification $(X_\ast)^+$ is locally connected.
Given such an object $X_\ast$, its \emph{underlying topological space} is the complement $X:= X_\ast \smallsetminus \ast$ of the point. 

\item[\textbf{Mor:}]
The topological space of morphism from $X_\ast$ to $Y_\ast$ is the topological space $\ZEmb^{\sf top}(X_\ast,Y_\ast)$ of zero-pointed topological embeddings.  

\item[\textbf{$\bigotimes$:}]
The symmetric monoidal structure is given by wedge sum $\bigvee$ among based spaces.

\end{itemize}

\end{definition}

\begin{remark}[Zero-object = unit]\label{zero-object}
Notice that the zero-pointed space $\ast$, with underlying space $\emptyset$, is a zero-object in $\ZTop$.  In other words, for each zero-pointed space $X_\ast$ there are unique morphisms $\ast \to X_\ast \to \ast$ in $\ZTop$. 
Moreover, this zero-object $\ast$ is the unit of the symmetric monoidal structure $\bigvee$ on $\ZTop$.  
\end{remark}

Here is an immediate consequence of Lemma~\ref{neg-maps}.  
\begin{lemma}[Negation]\label{opposite}
Negation implements a contravariant involution 
\[
\neg\colon \ZTop~ \cong~ \ZTop^{\op} \colon \neg~,
\]
as a symmetric monoidal topological category.

\end{lemma}

\subsection{Zero-pointed singular manifolds}
We introduce the category $\ZMfld_n$ of \emph{zero-pointed $n$-manifolds}.
This is a modification of the category $\ZTop$ introduced in~\S\ref{sec.ZTop}.
This subsection makes light reference to the definition of a stratified space as in~\S3 of~\cite{aft1}, and of a category of basics as in~\S4 therein.  We briefly recall these notions.

In this section we fix a finite cardinality $n$.

\subsubsection{\bf Singular manifolds}\label{sec.review}
All of the material in this subsection is extracted from the works~\cite{aft1} and~\cite{aft2}, which are joint with Hiro Lee Tanaka.  

The topological category $\Snglr$ consists of conically smooth singular manifolds, and conically smooth open embeddings among them.
The full topological subcategory $\Bsc\subset \Snglr$ consists of the basic singularity types, or \emph{basics} for short, each of which is of the form $\RR^k \times \sC(L)$ for $k$ a non-negative integer and $L$ a compact singular manifold -- here
\[
\sC(L)~:=~ \ast \underset{L\times \{0\}}\amalg L\times [0,1)
\]
is the \emph{open cone}, which has a standard structure as a singular manifold.
There is a bijection between subcollections of basic singularity types and full subcategories $\cB\subset \Bsc$.
Such a subcollection $\cB$ is \emph{stable} if the following conditions holds:
\begin{itemize}
\item
For $B\cong B'$ an isomorphism between basics, then $B\in \cB$ is a member of this subcollection if and only if $B'\in \cB$ is as well.

\item
For each member $B$ of this collection, the collection of conically smooth open embeddings $\{B'\hookrightarrow B\mid B'\in \cB\}$ from members of $\cB$ forms a basis for the underlying topology of $B$. 

\end{itemize}
There is a bijection between stable subcollections of basics and of fully-faithful right fibrations $\cB\hookrightarrow \Bsc$ up to equivalence over $\Bsc$.  
Consequently, we proceed with this understanding: fully-faithful right fibrations over $\Bsc$ are stable collections of singularity types.  

Now, fix a fully-faithful right fibration $\cB\hookrightarrow \Bsc$.
There results the topological category of \emph{$\cB$-manifolds}, which is the full subcategory
\[
\Mfld(\cB)~\subset~\Snglr
\]
consisting of those singular manifolds $X$ there exists a finite hypercover of $X$ by conically smooth open embeddings from objects of $\cB$.  
In other words, a singular manifold $X$ belongs to $\Mfld(\cB)$ if and only if each of the singularity types witnessed in $X$ belongs to the collection $\cB$, and if it admits a compactification as a singular manifold with corners.  

\begin{remark}
Note that each object $X\in \Mfld(\cB)$ has the property that it is locally compact and Hausdorff, and its 1-point compactification $X^+$ is locally contractible, and in particular locally connected.  
\end{remark}

Disjoint union defines a symmetric monoidal structure on $\Snglr$, which restricts as a symmetric monoidal structure on $\Mfld(\cB)$.  
By definition, the full topological subcategory $\cB \subset \Bsc \subset \Snglr$ is contained in the full topological subcategory $\Mfld(\cB)\subset \Snglr$.  
The full symmetric monoidal topological subcategory of \emph{$\cB$-disks}~,
\[
\Disk(\cB)~\subset~\Mfld(\cB)~,
\]
is the smallest such containing $\cB\subset \Mfld(\cB)$.

\begin{remark}
In~\cite{aft1}, the $\infty$-category $\Mfld(\cB)$ was denoted $\Mfld(\cB)^{\sf fin}$, an object being referred to as a \emph{finitary} $\cB$-manifold.  We omit this decoration in this article because we will only concern with this class of $\cB$-manifolds.

\end{remark}

\begin{example}\label{boundary-examples}
Here are some examples of subcategories $\cB\subset \Bsc$ for which the inclusion $\cB\hookrightarrow \Bsc$ is a right fibration.  
\begin{itemize}
\item $\cD_n\subset \Bsc$ is the full subcategory consisting of the basic singularity type $\RR^n$.
The full symmetric monoidal topological subcategory 
\[
\Mfld_n:=\Mfld(\cD_n)~\subset ~\Snglr
\]
consists of smooth $n$-manifolds and smooth open embeddings among them (with the compact-open topology).  
The full symmetric monoidal topological subcategory
\[
\Disk_n:=\Disk(\cD_n)~\subset~\Mfld(\cD_n)=:\Mfld_n
\]
consists of those objects that are diffeomorphic to a finite disjoint union of Euclidean $n$-spaces.

\item $\cD_n^\partial \subset \Bsc$ is the full subcategory consisting of the basic singularity types $\RR^n$ and $\HH^n:=\RR_{\geq 0}\times \RR^{n-1}$.  
The full symmetric monoidal topological subcategory 
\[
\Mfld_n^\partial:=\Mfld(\cD_n^\partial)~\subset ~\Snglr
\]
consists of smooth $n$-manifolds with boundary and smooth open embeddings among them (with the compact-open topology). 
The full symmetric monoidal topological subcategory
\[
\Disk_n^\partial:=\Disk(\cD_n^\partial)~\subset~\Mfld(\cD_n^\partial)=:\Mfld_n^\partial
\]
consists of those objects each of whose connected components is diffeomorphic to $\RR^n$ or $\HH^n$.  

\item $\cD_{\lag n \rag} \subset \Bsc$ is the full subcategory consisting of those basic singularity types $\RR_{\geq 0}^{\times i}\times \RR^{n-i}$ for each $0\leq i\leq n$.  
The full symmetric monoidal topological subcategory 
\[
\Mfld_{\lag n\rag}:=\Mfld(\cD_{\lag n \rag})~\subset ~\Snglr
\]
consists of smooth $n$-manifolds with corners and smooth open embeddings among them (with the compact-open topology). 
The full symmetric monoidal topological subcategory
\[
\Disk_{\lag n \rag}:=\Disk(\cD_{\lag n \rag})~\subset~\Mfld(\cD_{\lag n \rag})=:\Mfld_{\lag n \rag}
\]
consists of those objects each of whose connected components is diffeomorphic to $\RR_{\geq 0}^{\times i}\times \RR^{n-i}$ for some $0\leq i \leq n$.  

\end{itemize}

\end{example}

\begin{remark}
While the notion of a $\cB$-manifold captures a very general class of structured topological spaces, we will mainly be interested in the cases of Example~\ref{boundary-examples}.

\end{remark}

\subsubsection{\bf Zero-pointed singular manifolds}
We now introduce zero-pointed manifolds, which form a symmetric monoidal topological category.  
We do this as a special case of zero-pointed singular manifolds -- a degree of generality that is convenient for later use.

\begin{definition}\label{zero.pointed.manifold}
Let $\cB\hookrightarrow \Bsc$ be a fully-faithful right fibration.  
A \emph{zero-pointed $\cB$-manifold} is an object $X\in \Mfld(\cB)$ together with a pointed extension $X_\ast$ of its underlying topological space, subject to the following condition.
\begin{itemize}

\item[~]
The 1-point compactification $(X_\ast)^+$ admits the structure of a conically smooth singular manifold with respect to which the inclusion $X\hookrightarrow (X_\ast)^+$ is a conically smooth open embedding.  

\end{itemize}
A \emph{zero-pointed $n$-manifold} is a zero-pointed $\cD_n$-manifold; a zero-pointed $n$-manifold with boundary is a zero-pointed $\cD_n^\partial$-manifold.

\end{definition}

\begin{example}\label{cobordism}
Let $\ov{M}$ be a smooth cobordism; this is to say that $\ov{M}$ is a compact smooth manifold with boundary $\partial \ov{M} = \partial_L \amalg \partial_R$ which identified as a coproduct. 
Then 
\[
(M_L)_\ast ~:=~ \ast \underset{\partial_L} \coprod (\ov{M}\smallsetminus \partial_R)\qquad \text{ and } \qquad (M_R)_\ast ~:=~ \ast \underset{\partial_R} \coprod (\ov{M}\smallsetminus \partial_L)
\]
are zero-pointed $n$-manifolds.  

\end{example}

\begin{remark}\label{all.cobs}
The extra condition on the datum of a zero-pointed $n$-manifold $M_\ast$ can be regarded as a tameness condition.  
Specifically, given such a conically smooth structure on $(M_\ast)^+$, spherical blow-ups at the two special points $\ast,+\in (M_\ast)^+$ results in a smooth cobordism $\ov{M}$ whose interior is canonically diffeomorphic to $M$.  
In other words, the named condition ensures that Example~\ref{cobordism} produces \emph{all} examples of zero-pointed $n$-manifolds.
We emphasize, however, there are many smooth cobordisms that determine the same pair of zero-pointed $n$-manifolds -- see Remark~\ref{h-cobordism}, adapted to the smooth setting.

\end{remark}

The next observation is manifest from definitions.
\begin{observation}\label{smooth.neg}
Let $\cB\hookrightarrow \Bsc$ be a fully-faithful right fibration.
For $M_\ast$ a zero-pointed $\cB$-manifold, the negation $M_\ast^\neg$ of its underlying zero-pointed topological space is canonically equipped with the structure of a zero-pointed $\cB$-manifold.

\end{observation}

\begin{notation}\label{def.smooth.neg}
Let $\cB\hookrightarrow \Bsc$ be a fully-faithful right fibration.
For $M_\ast$ a zero-pointed $\cB$-manifold, its \emph{negation} $M_\ast^\neg$ is the zero-pointed $\cB$-manifold from Observation~\ref{smooth.neg}.

\end{notation}

\begin{definition}\label{def.smooth.zemb}
Let $\cB\hookrightarrow \Bsc$ be a fully-faithful right fibration.
Let $M_\ast$ and $N_\ast$ be zero-pointed $\cB$-manifolds.
The topological space of \emph{zero-pointed embeddings}
from $M_\ast$ to $N_\ast$ is the (compactly generated weak Hausdorff replacement of the) subspace
\[
\ZEmb(M_\ast,N_\ast)~:=~
\bigl\{M_\ast\xra{f} N_\ast  \mid f_{|}\colon f^{-1}(N) \underset{\rm open~embedding}{\xra{\rm conically~smooth}} N \bigr\}
~\subset~\ZEmb^{\sf top}(M_\ast,N_\ast)
\]
consisting of those zero-pointed topological embeddings between their underlying pointed extensions whose restriction to the open subspace of $M$ over $N$ is conically smooth open embedding.

\end{definition}

\begin{notation}
Let $\cB\hookrightarrow \Bsc$ be a fully-faithful right fibration.
For $M_\ast$ a zero-pointed $\cB$-manifold, we typically do not keep the smooth structure on $M$ in the notation, and denote its underlying zero-pointed topological space again as $M_\ast$.

\end{notation}

\begin{observation}\label{smooths.compose}
Let $\cB\hookrightarrow \Bsc$ be a fully-faithful right fibration.
Let $M_\ast\xra{f}N_\ast\xra{g}K_\ast$ be zero-pointed embeddings between zero-pointed $\cB$-manifolds.
The composition $g\circ f\colon M_\ast \to K_\ast$ is again a zero-pointed embedding between zero-pointed $\cB$-manifolds.

\end{observation}

The next observation is manifest from definitions, and the construction of the functor of Lemma~\ref{neg-maps}.  
\begin{observation}\label{smooth.neg.maps}
Let $M_\ast$ and $N_\ast$ be zero-pointed $\cB$-manifolds.
The homeomorphism of Lemma~\ref{neg-maps} restricts as a homeomorphism of spaces of zero-pointed embeddings:
\[
\neg\colon \ZEmb(M_\ast,N_\ast)\xra{~\cong~}\ZEmb(N_\ast^\neg,M_\ast^\neg)~.  
\]

\end{observation}

This next definition organizes Definitions~\ref{zero.pointed.manifold} and~\ref{def.smooth.zemb}, making use of Observation~\ref{smooths.compose}.
We understand the enrichment as in the Cartesian monoidal category of compactly generated weak Hausdorff topological spaces.

\begin{definition}\label{def.ZMfld}
Let $\cB\hookrightarrow \Bsc$ be a fully-faithful right fibration.
The symmetric monoidal topological category 
\[
\ZMfld(\cB)
\]
is that for which an object is a zero-pointed $\cB$-manifold and, for $M_\ast$ and $N_\ast$ zero-pointed $\cB$-manifolds, the topological space of morphisms from $M_\ast$ to $N_\ast$ is that of zero-pointed embeddings from $M_\ast$ to $N_\ast$.  
Composition is given by composing zero-pointed embeddings.  
The symmetric monoidal structure is given by wedge sum.

\end{definition}

\begin{notation}\label{def.actually}
The symmetric monoidal topological category of \emph{zero-pointed $n$-manifolds} is 
\[
\ZMfld_n~:=~ \ZMfld(\cD_n)~.
\]
The symmetric monoidal topological category of \emph{zero-pointed $n$-manifolds with boundary} is 
\[
\ZMfld_n^{\partial}~:=~ \ZMfld(\cD_n^\partial)~.
\]

\end{notation}

\begin{remark}
We follow up on Remark~\ref{all.cobs}.
Let $\ov{M}$ and $\ov{N}$ be two smooth $n$-dimensional cobordisms.
Consider the zero-pointed $n$-manifolds $M_\ast:=(M_L)_\ast$ and $N_\ast:=(N_L)_\ast$ from Example~\ref{cobordism}.  
While there is a continuous injection $\Emb(\ov{M}\setminus \partial_R,\ov{N}\setminus \partial_R) \hookrightarrow \ZEmb(M_\ast,N_\ast)$, this map is far from being an equivalence of any sort.
This is demonstrated even in the simplest case of $\ov{M} =[-1,1]$ with $\partial_L = \{\pm 1\}$ and $\ov{N}=[-1,1]$ with $\partial_L = \emptyset$.
Namely, the topological space $\ZEmb\bigl((-1,1)_+,(-1,1)^+\bigr)$ retracts onto $(-1,1)^+$ whereas $\Emb\bigl((-1,1),[-1,1]\bigr)$ is homotopy discrete.  

\end{remark}

Forgetting conically smooth structures defines a symmetric monoidal continuous functor
\[
\ZMfld(\cB) \longrightarrow \ZTop~.
\]
By definition, this functor is an embedding on ${\sf Hom}$-topological spaces.  
Observations~\ref{smooth.neg} and~\ref{smooth.neg.maps} compile as the following.

\begin{observation}\label{mfld.neg}
The contravariant involution of Lemma~\ref{opposite} restricts as a contravariant involution
\[
\neg\colon \ZMfld(\cB)~\cong~\ZMfld(\cB)\colon \neg~,
\]
as a symmetric monoidal topological category.  

\end{observation}

\section{Reduced factorization (co)homology}
For coefficients in an augmented algebra, we extend factorization homology to zero-pointed manifolds; likewise for augmented coalgebras and factorization cohomology.  
\\

\noindent
In this section we fix the following parameters.
\begin{itemize}
\item A dimension $n$.
\item A symmetric monoidal $\infty$-category $\cV$.
\end{itemize}

\begin{terminology}[$\ot$-sifted cocomplete]\label{conditions}
We say $\cV$ is \emph{$\ot$-sifted cocomplete} if satisfies the following two conditions.
\begin{enumerate}
\item The underlying $\infty$-category of $\cV$ admits sifted colimits.
This is to say that each functor $\cK \to \cV$ from a \emph{sifted} $\infty$-category admits a colimit.
(Recall that an $\infty$-category $\cK$ is \emph{sifted} if it is not empty and the diagonal functor $\cK \to \cK\times \cK$ is final.)

\item  The symmetric monoidal structure of $\cV$ distributes over sifted colimits.
This is to say that, for each object $v\in \cV$, the functor $v\ot- \colon \cV \to \cV$ carries sifted colimit diagrams to sifted colimit diagrams.   
\end{enumerate}
We say $\cV$ is $\ot$-cosifted complete if the symmetric monoidal $\infty$-category $\cV^{\op}$ is $\ot$-sifted cocomplete. 

\end{terminology}

\begin{remark}[$B$-structures]
For $B\to \BO(n)$ a map between spaces, every result in this section is valid after an evident modification that accounts for a $B$-structure.  
More generally, for $\cB\to \Bsc$ a right fibration, every notion in this section is valid after an evident modification for (zero-pointed) $\cB$-manifolds (see \cite{aft2}), and $\Disk(\cB)$ in place of $\Disk_n$.  
We choose to not carry such clutter with the discussion.

\end{remark}

Recall from~\S\ref{sec.review} the symmetric monoidal $\infty$-category $\Disk(\cB)$ associated to each fully-faithful right fibration $\cB\hookrightarrow \Bsc$.  
\begin{definition}\label{def:alg-n}
Let $\cV$ be a symmetric monoidal $\infty$-category.
For $\cB \hookrightarrow \Bsc$ a fully-faithful right fibration, the respective $\infty$-categories of \emph{$\Disk(\cB)$-algebras (in $\cV$)} and of \emph{$\Disk(\cB)$-coalgebras (in $\cV$)} are those of symmetric monoidal functors:
\[
\Alg_{\Disk(\cB)}(\cV) := \Fun^\ot\bigl(\Disk(\cB),\cV\bigr)
\qquad\text{ and }\qquad 
\cAlg_{\Disk(\cB)}(\cV) := \Fun^\ot\bigl(\Disk(\cB)^{\op} ,\cV\bigr) ~,
\]
the latter which is alternatively identified as $\bigl(\Fun^\ot\bigl(\Disk(\cB),\cV^{\op}\bigr)\bigr)^{\op}$.
In the case $\cD_n\hookrightarrow \Bsc$, we refer to a $\Disk(\cD_n)$-(co)algebra simply as an $n$-disk (co)algebra, and simplify the notation:
\[
\Alg_n(\cV) := \Alg_{\Disk(\cD_n)}(\cV)\qquad\text{ and }\qquad \cAlg_n(\cV) := \cAlg_{\Disk(\cD_n)}(\cV) ~.
\]

\end{definition}

\begin{remark}
There is a close relationship between the topological operad $\cE_n$ of little $n$-disks and the symmetric monoidal topological category $\Disk_n$, as we now explain.
Change of framings defines an action of $\sO(n)$ on $\cE_n$.
There is a standard $\cE_n$-algebra in the symmetric monoidal $\infty$-category $\Disk_n$ which selects $\RR^n$.  
There results a symmetric monoidal functor ${\sf Env}(\cE_n) \to \Disk_n$ from the symmetric monoidal envelope.  
This symmetric monoidal functor is manifestly $\sO(n)$-invariant; the resulting factorization through the  coinvariants 
\[
{\sf Env}(\cE_n)_{\sO(n)} \xra{\simeq} \Disk_n
\]
is an equivalence between symmetric monoidal $\infty$-categories (where the coinvariants are calculated in symmetrical monoidal $\oo$-categories).
Thereafter, there is a canonical identification of the $\sO(n)$-invariants
\[
\Alg_n(\cV) ~\xra{~\simeq~}~\Alg_{\cE_n}(\cV)^{\sO(n)}
\]
for each symmetric monoidal $\infty$-category $\cV$.

\end{remark}

\subsection{Augmentation}
We establish a relationship between zero-pointed manifolds and augmentations.
Specifically, we characterize augmented $n$-disk algebras in terms of a full symmetric monoidal topological subcategory $\Disk_{n,+}\subset \ZMfld_n$.

\begin{definition}\label{def.augmented}
Let $\cM$ and $\cV$ be symmetric monoidal $\oo$-categories. 
The $\oo$-category of augmented symmetric monoidal functors is the $\infty$-overcategory
\[
\Fun^{\ot, \sf aug}(\cM,\cV) := \Fun^{\ot}(\cM,\cV)_{/\uno_\cV}
\]
over the constant symmetric monoidal functor at the symmetric monoidal unit of $\cV$.  

\end{definition}

\begin{notation}
Note the forgetful functor
\[
\Fun^{\ot, \sf aug}(\cM,\cV) \longrightarrow \Fun^{\ot}(\cM,\cV)~.
\]
We will typically not distinguish in notation between an augmented symmetric monoidal functor and the symmetric monoidal functor which is its value under this forgetful functor.

\end{notation}

\begin{example}
Let $\cM$ and $\cV$ be symmetric monoidal $\infty$-categories.
Suppose the symmetric monoidal structure on $\cV$ is Cartesian.  
Then, in particular, the symmetric monoidal unit $\uno_\cV = \ast$ is final in the underlying $\infty$-category of $\cV$.
Furthermore, the forgetful functor
\[
\Fun^{\ot, \sf aug}(\cM,\cV) \longrightarrow \Fun^{\ot}(\cM,\cV)
\]
is an equivalence between $\infty$-categories.
This is to say that every symmetric monoidal functor $\cM \to \cV$ is uniquely augmented.

\end{example}

Notice the continuous functors
\begin{equation}\label{manifolds-to-cones}
(-)_+ \colon \Mfld_n \hookrightarrow \ZTop \hookleftarrow \Mfld_n^{\op} \colon (-)^+~,
\end{equation}
given by adjoining a disjoint basepoint, and by 1-point compactification, respectively.  
Each of these functors is naturally symmetric monoidal.

\begin{observation}\label{faithful}
The expression in Example~\ref{pre.aug} reveals that the functor $(-)_+$ is faithful in the sense that, for $X$ and $Y$ finitary smooth manifolds, the continuous map between ${\sf Hom}$-topological spaces
\begin{equation}\label{zemb.+}
(-)_+\colon \Emb(X,Y)~\longrightarrow ~\ZEmb(X_+,Y_+)
\end{equation}
is an inclusion of connected components.  
After Observation~\ref{mfld.neg} it follows that the functor $(-)^+$, too, is faithful in this sense.

\end{observation}

\begin{definition}\label{Mfld.+}
The full symmetric monoidal topological subcategories
\[
\Disk_{n,+}~\subset~\Mfld_{n,+}~\subset~\ZMfld_n~\supset~\Mfld^{+}_n~\supset~\Disk_n^{+}
\]
respectively
consist of those objects in the image of the symmetric monoidal functors 
\[
\Disk_n\subset \Mfld_n \xra{(-)_+}\ZMfld_n   \xla{(-)^+} \Mfld_n^{\op}\supset \Disk_n^{\op}~.
\] 
The full symmetric monoidal topological subcategory
\[
\ZDisk_n~\subset~\ZMfld_n
\]
is the smallest such containing $\Disk_{n,+}$ and $\Disk_n^{+}$.  

\end{definition}

\begin{observation}\label{M.+.explicit}

Explicitly, an object in $\Mfld_{n,+}$ is the datum of a smooth $n$-manifold that is the interior of a compact smooth manifold with boundary.
For $X$ and $Y$ such, the topological space of morphisms
\[
\ZEmb(X_+,Y_+)~=~\underset{[X]_+\xra{f} [Y]_+}\coprod \underset{Y_j\in [Y]} \prod \Emb(f^{-1}Y_j,Y_j)
\]
where the coproduct is indexed by based maps between sets of connected components.
Composition is given by composing such based maps, and composing smooth open embeddings.  

\end{observation}

\begin{lemma}\label{actually.loc}
Let $M$ be a finitary smooth $n$-manifold.
The functor between $\infty$-overcategories
\[
(-)_+\colon \Disk_{n/M} \longrightarrow \Disk_{n,+/M_+}
\]
is a fully-faithful right adjoint.  

\end{lemma}

\begin{proof}
Observation~\ref{M.+.explicit} reveals that this functor is fully-faithful; its image consists of those zero-pointed embeddings $(f\colon U_+\to M_+)$ for which inverse image of the zero-point is the zero-point, $f^{-1}\{+\} = +$.
We now argue that this functor is a right adjoint.
Let $(f\colon U_+\hookrightarrow M_+)\in \bigl(\Disk_{n,+}\bigr)_{/M_+}$ be an object.
We must show that the $\infty$-undercategory $\Bigl(\bigl(\Disk_{n,+}\bigr)_{/M_+}\Bigr)^{f/}$ has an initial object.
Well, the object $(f\colon U_+\xra{\rm collapse}f^{-1}(M)_+ \hookrightarrow M_+)$ in this $\infty$-undercategory is initial, as seen by inspecting the expression in Observation~\ref{M.+.explicit} of the mapping spaces in $\Mfld_{n,+}$.

\end{proof}

Taking connected components defines a continuous symmetric monoidal functor
\begin{equation}\label{disk-to-fin}
[-]\colon \Disk_{n,+} \to \Fin_\ast~,\qquad \underset{I} \bigvee \RR^n_+ \mapsto I_+~,
\end{equation}
to based finite sets, with wedge sum.  
The following result follows immediately from Theorem~4.3.1 of~\cite{aft2}. 
\begin{observation}\label{disk-bijection}
As a functor between $\infty$-categories, $[-]\colon \Disk_{n,+}\to \Fin_\ast$ is conservative.  
In other words, the maximal $\oo$-subgroupoid of $\Disk_{n,+}$ is canonically identified as any of the following $\infty$-groupoids
\[
\Disk_{n,+}^{\sf bij} := (\Disk_{n,+})_{|\Fin_\ast^{\sf bij}} ~{}~{}~\simeq~{}~{}~ \underset{i\geq 0}\coprod \Disk_{n,+}^{=i}~{}~{}~\simeq~{}~{}~\sB(\Sigma\wr \sO(n)) ~{}~{}~\simeq~{}~{}~\coprod_{i\geq 0} \sB(\Sigma_i\wr \sO(n))~,
\]
the first which is the $\infty$-subgroupoid consisting of those morphisms which are bijections on connected components, the second which is the coproduct over finite cardinalities of the $\infty$-subgroupoids with a specified cardinality of components, and the others which are classical.  

\end{observation}

\begin{observation}\label{disk.dual}
Through Observation~\ref{mfld.neg}, negation implements isomorphisms between symmetric monoidal topological categories 
\[
\Mfld_{n,+}^{\op}~\cong~\Mfld_n^+
\qquad \text{ and }\qquad
\ZDisk_n^{\op}~\cong~\ZDisk_n
\qquad \text{ and } \qquad
\Disk_{n,+}^{\op}~\cong~\Disk_n^+~.  
\]
  
\end{observation}

\begin{prop}\label{M.+.correct}
Let $\cV$ be a symmetric monoidal $\infty$-category.
Restriction along the symmetric monoidal functors $(-)_+$ and $(-)^+$ define equivalences between $\infty$-categories
\[
\Fun^{\ot}\bigl(\Disk_{n,+},\cV\bigr) \xra{~\simeq~} \Alg_n^{\sf aug}(\cV)
\qquad \text{ and }\qquad
\Fun^{\ot}\bigl(\Disk_n^+,\cV\bigr) \xra{~\simeq~}\cAlg_n^{\sf aug}(\cV)~,
\]
\[
\Fun^{\ot}\bigl(\Mfld_{n,+},\cV\bigr) \xra{~\simeq~} \Fun^{\ot, \sf aug}\bigl(\Mfld_{n},\cV\bigr) 
\qquad \text{ and }\qquad
\Fun^{\ot}\bigl(\Mfld_n^+,\cV\bigr) \xra{~\simeq~} \Fun^{\ot, \sf aug}\bigl(\Mfld_{n}^{\op},\cV\bigr) ~.
\]

\end{prop}

\begin{proof}
This proof is outsourced to the Appendix (\S\ref{sec.appendix}).
Namely, the left two equivalences are a direct application of Proposition~\ref{M.+.adjoint} applied to Example~\ref{disk.disjunctive}, while making use of Example~\ref{actually.loc}.
The right two equivalences follow thereafter by replacing $\cV$ with $\cV^{\op}$ and implementing Observation~\ref{disk.dual}.

\end{proof}

\subsection{Homology and cohomology}
We extend factorization homology and cohomology to zero-pointed manifolds.  
\\

We first recall some notions within $\infty$-category theory.
Let $i: \cD \ra \cM$ and $A:\cD \ra \cV$ be functors between $\oo$-categories.
The left Kan extension of $A$ along $f$ is the functor $f_!A\colon \cM\ra \cV$ whose values, when they exist, can be computed as colimits:
\[
f_! A\colon \cM\ni M ~\mapsto ~ \colim\bigl(\cD_{/M}\ra \cD \xra{A} \cV\bigr)\in \cV~;
\]
here, the $\infty$-category $\cD_{/M} := \cD \underset{\cM}\times \cM_{/M}$ is the $\oo$-overcategory of $\cD$ over the object $M\in \cM$. 
Likewise, the right Kan extension of $A$ along $f$ is the functor $f_\ast A\colon \cM \to \cV$ whose values, when they exist, can be calculated as limits:
\[
f_* A\colon \cM\ni M ~\mapsto~ \limit\bigl(\cD^{M/}\ra \cD\xra{A} \cV\bigr)~;
\]
here, the $\infty$-category $\cD^{M/} := \cD \underset{\cM}\times \cM^{M/}$ is the $\oo$-undercategory of $\cD$ under $M$.

\begin{definition}[Factorization (co)homology for zero-pointed manifolds]\label{def:pointed-fact-homo}
Let $\cV$ be a symmetric monoidal $\infty$-category.
Let $M_\ast$ be a zero-pointed $n$-manifold.
Let $A\colon \Disk_{n,+} \to \cV$ and $C\colon \Disk_n^+ \to \cV$ be functors. 
Whenever they exist, we define the objects of $\cV$
\begin{eqnarray}\label{def.fact-hmlgy}
\int_{M_\ast} A  
&
:=  
&
\colim\bigl(\bigl(\Disk_{n, +}\bigr)_{/M_\ast} \to \Disk_{n,+} \xra{A} \cV \bigr)
\\
\nonumber
&
= 
&
\underset{U_+\to M_\ast}\colim A(U_+) 
\end{eqnarray}
and
\begin{eqnarray}\label{def.fact-coho}
\int^{M_\ast} C
&
:=  
&
\limit\bigl(\bigl(\Disk^+_n\bigr)^{M_\ast/}  \to \Disk_n^+ \xra{C} \cV \bigr)
\\
\nonumber
&
= 
&
\underset{M_\ast \to V^+}\limit C(V^+) 
\end{eqnarray}
and refer to the first as the \emph{factorization homology of $M_\ast$ (with coefficients in $A$)}, and the second as the \emph{factorization cohomology of $M_\ast$ (with coefficients in $C$)}.  

\end{definition}

We point out that the above notion of factorization homology agrees with that considered in previous work~\cite{fact}.
\begin{lemma}\label{same-fact}
Let $M$ be an $n$-manifold and let $A$ be an augmented $n$-disk algebra.  
Consider the symmetric monoidal functor $A\colon \Disk_{n,+}\to \cV$ associated to $A$ via Proposition~\ref{M.+.correct}.
There is a canonical equivalence
\[
\int_{M} A \xra{~\simeq~} \int_{M_+} A~.
\]

\end{lemma}

\begin{proof}
Lemma~\ref{actually.loc} implies the functor $(-)_+\colon \bigl(\Disk_n\bigr)_{/M} \to \bigl(\Disk_{n,+}\bigr)_{/M_+}$ is final.

\end{proof}

We conclude this subsection by stating a universal property that factorization (co)homology satisfies.
Recall Conditions~\ref{conditions} that a symmetric monoidal $\infty$-category might satisfy.  
We prove the following result in Section~\S\ref{EE-of-M-proof}, contingent on the key technical result Proposition~\ref{EE-of-M}.  
\begin{theorem}\label{fact-functor}
Let $\cV$ be a symmetric monoidal $\infty$-category.
The following two independent statements are true.  
\begin{enumerate}
\item 
If the underlying $\infty$-category of $\cV$ admits sifted colimits and the symmetric monoidal structure $\cV\times \cV \xra{\ot} \cV$ distributes over sifted colimits separately in each variable, then there are fully-faithful left adjoints to the horizontal restriction functors
\begin{equation}\label{pair-adjunctions}
\xymatrix{
\Alg_n^{\sf aug}(\cV)  \ar@(-,u)[r]^-{\int_-} \ar[dd]
&
\Fun^\ot\bigl(\zmfld_n, \cV\bigr)   \ar[l]   \ar[dd]
\\
&&
\\
\Fun\bigl(\Disk_{n,+},\cV\bigr)  \ar@(-,u)[r]^-{\int_-}
&
\Fun\bigl(\zmfld_n, \cV\bigr)  \ar[l]
}
\end{equation}
with respect to which the down-rightward square commutes.

\item 
If the underlying $\infty$-category of $\cV$ admits cosifted limits and the symmetric monoidal structure $\cV\times \cV \xra{\ot} \cV$ distributes over cosifted limits separately in each variable, then there are fully-faithful right adjoints to the horizontal restriction functors
\begin{equation}\label{pair-adjunctions'}
\xymatrix{
\Fun^\ot\bigl(\zmfld_n, \cV\bigr)  \ar[r]   \ar[dd]
&
\cAlg_n^{\sf aug}(\cV)~. \ar@(-,d)[l]^-{\int^-}   \ar[dd]
\\
&&
\\
\Fun\bigl(\zmfld_n, \cV\bigr)  \ar[r]  
&
\Fun\bigl(\Disk_n^+ , \cV\bigr)  \ar@(-,d)[l]^-{\int^-}
}
\end{equation}
with respect to which the down-leftward square commutes.  

\end{enumerate}

\end{theorem}

We prove Theorem~\ref{fact-functor} in section~\S\ref{EE-of-M-proof}.

\begin{remark}\label{partially-defined}
Even without the distribution assumptions on $\ot$, we understand that the lower adjoints in diagrams~(\ref{pair-adjunctions}) and~(\ref{pair-adjunctions'}) are always defined on some full, possibly empty,  subcategories of the respective domain (co)algebra categories.  
\end{remark}

\subsection{Exiting disks}
The $\infty$-overcategory $\Disk_{n,+/M_\ast}$ appears in the defining expression for factorization homology.  
We give a variant of this $\infty$-category $\Disk_+(M_\ast)$, of \emph{exiting disks} in $M_\ast$, which offers several conceptual and technical advantages.  
Heuristically, objects of $\Disk_+(M_\ast)$ are embeddings from finite disjoint unions of disks into $M$, while morphisms are isotopies that can witness sliding disks off to infinity where they are forgotten -- disks are not allowed to be created at infinity, unlike in $\Disk_{n,+/M_\ast}$.  
We make light use of some theory of stratified spaces as developed in~\cite{aft1}, and of some results thereabout in~\cite{aft2}.

For this subsection, fix a zero-pointed manifold $M_\ast$ together with a conically smooth structure on $(M_\ast)^+$ with respect to which the canonical continuous map $M \hookrightarrow (M_\ast)^+$ is a conically smooth open embedding.  
In~\S2.1 of~\cite{aft2} we define, for each stratified space $X$, the $\infty$-category 
\[
\Disk(\bsc)_{/X}
\]
of finite disjoint unions of basics embedding into $X$; this is a stratified version of $\Disk_{n/M}$.

\begin{definition}[$\Disk_+(M_\ast)$]\label{def.of-M}
The $\infty$-category of \emph{exiting disks} of $M_\ast$ is the full $\infty$-subcategory
\[
\Disk_+(M_\ast)~\subset~ \Disk(\bsc)_{/M_\ast}
\]
consisting of those $V\hookrightarrow M_\ast$ whose image contains $\ast$.  
We use the notation
\[
\Disk^+(M_\ast) ~:=~ \bigl(\Disk_+(M_\ast^\neg)\bigr)^{\op}~.  
\]

\end{definition}

\begin{remark}\label{rem.obj}
Explicitly, an object in $\Disk_+(M_\ast)$ is a conically smooth open embedding $B\sqcup U \hookrightarrow M_\ast$ where $B\cong \sC(L)$ is a cone-neighborhood of $\ast\in M_\ast$ and $U$ is abstractly diffeomorphic to a finite disjoint union of Euclidean spaces, and a morphism is an isotopy to an embedding among such.

\end{remark}

\begin{lemma}\label{slice-is-of}
Let $L$ be a compact stratified space, and let $\sC(L) \hookrightarrow M_\ast$ be a conically smooth open embedding whose image contains the base point of $M_\ast$.  
The projection from the $\infty$-overcategory defines an equivalence
\begin{equation}\label{of-M-over}
\bigl(\Disk(\bsc)_{/M_\ast}\bigr)^{\sC(L)/}
\xra{~\simeq~}  
\Disk_+(M_\ast)
\end{equation}
between $\infty$-categories.

\end{lemma}

\begin{proof}
Let $\sC(L)\hookrightarrow M_\ast$ be a basic centered at the base point.  
Conically smooth open embeddings $\sC(L) \to \sC(L)$ form a basis for the topology of $\sC(L)$ about the cone-point.  
Using this, it follows from Lemma~4.3.7 in~\cite{aft1} that any conically smooth open embedding from a basic $B\hookrightarrow M_\ast$ whose image contains $\ast$ is isotopic to one that factors through an isomorphism $B\cong \sC(L)\hookrightarrow M_\ast$.  
Stronger, it follows from that same reference that the space of such isotopies is contractible.  
We conclude that the projection from the slice
\[
\bigl(\Disk(\bsc)_{/M_\ast}\bigr)^{\sC(L)/} \xra{~\simeq~} \Disk_+(M_\ast)
\]
is an equivalence between $\infty$-categories.

\end{proof}

\begin{lemma}\label{of.over.final}
The fully-faithful functor $\Disk_+(M_\ast) \hookrightarrow \Disk(\Bsc)_{/M_\ast}$ is a right adjoint.
In particular, this functor is final.  

\end{lemma}

\begin{proof}
Let $L$ be a compact stratified space, and let $\sC(L) \hookrightarrow M_\ast$ be a conically smooth open embedding whose image contains the base point of $M_\ast$. 
By way of Lemma~\ref{slice-is-of}, the statement of this lemma is equivalent to showing the functor 
\begin{equation}\label{sifted-slice}
\bigl(\Disk(\bsc)_{/M_\ast}\bigr)^{\sC(L)/} \longrightarrow  \Disk(\bsc)_{/M_\ast}
\end{equation}
is a right adjoint.
This is the problem of proving the following assertion:
\begin{itemize}
\item[~]
For each object $(U\hookrightarrow M_\ast)$ of $\Disk(\bsc)_{/M_\ast}$, the $\infty$-undercategory $\bigl(\bigl(\Disk(\bsc)_{/M_\ast}\bigr)^{\sC(L)/}\bigr)^{U/}$ has an initial object.  
\end{itemize}
So let $(U\hookrightarrow M_\ast)$ be an object in $\Disk(\bsc)_{/M_\ast}$.  
Suppose the image of the embedding $U\hookrightarrow M_\ast$ contains the base point of $M_\ast$.
In this case, the object $(U\hookrightarrow M_\ast)$ of $\Disk(\bsc)_{/M_\ast}$ belongs to the essential image of the fully-faithful functor~(\ref{sifted-slice}).  
Consequently, this $\infty$-undercategory has an initial object, which is $(U=U\hookrightarrow M_\ast)$ itself.  
Now suppose the image of the embedding $U\hookrightarrow M_\ast$ does not contain the base point of $M_\ast$.
Choose an embedding $\sC(L)\sqcup U \hookrightarrow M_\ast$ making the diagram in $\Mfld(\bsc)$
\[
\xymatrix{
U  \ar[rr]^-{\rm inclusion}  \ar[dr]
&&
\sC(L) \sqcup U  \ar[dl]
\\
&
M_\ast
&
}
\]
commute -- such a diagram exists by the supposition on the embedding $U\hookrightarrow M_\ast$.  
This diagram defines an object in the $\infty$-undercategory of concern.
Making use of Lemma~4.3.7 from~\cite{aft1} again, notice that each conically smooth open embedding $U \hookrightarrow \sC(L)\sqcup V$ over $M_\ast$ canonically factors, over $M_\ast$, through the above inclusion $U\hookrightarrow \sC(L)\sqcup U$.
Consequently, the object in this $\infty$-undercategory depicted as the above diagram, is initial.  
This proves the above assertion, which concludes this proof.

\end{proof}

\begin{lemma}\label{loc.sifted}
Let $\cD_0 \hookrightarrow \cD$ be a fully-faithful right adjoint to a sifted $\infty$-category.
The $\infty$-category $\cD_0$, too, is sifted.

\end{lemma}

\begin{proof}
Siftedness of $\cD$ means, in particular, that $\cD$ is not empty.
Because $\cD_0$ is a localization of $\cD$, then $\cD_0$, too, is not empty.

It remains to prove the diagonal functor $\delta_0 \colon \cD_0\to \cD_0\times \cD_0$ is final.  
For this, it is sufficient to show that, for each cocomplete $\infty$-category $\cZ$, the natural transformation making the diagram
\[
\xymatrix{
\Fun(\cD_0\times \cD_0,\cZ) \ar[rr]^-{\delta_0}   \ar[dr]_-{\colim}
&&
\Fun(\cD_0,\cZ)    \ar[dl]^-{\colim}
\\
&
\cZ  
&
.
}
\]
commute is in fact a natural equivalence; in other words, that this diagram canonically commutes.

Denote the fully-faithful functor $i\colon \cD_0\hookrightarrow \cD$; denote the diagonal functor $\delta\colon \cD\to \cD\times \cD$.  
Consider the (a priori lax-commutative) diagram of $\infty$-categories
\[
\xymatrix{
\Fun(\cD_0\times \cD_0,\cZ) \ar[rr]^-{\delta_0^\ast}   \ar[dr]^-{\colim}
&&
\Fun(\cD_0,\cZ)   \ar[dl]_-{\colim}
\\
&
\cZ  
&
\\
\Fun(\cD\times \cD,\cZ)  \ar[ur]_-{\colim}     \ar[rr]_-{\delta^\ast}  \ar[uu]^-{(i\times i)^\ast}  
&&
\Fun(\cD,\cZ)  \ar[ul]^-{\colim}        \ar[uu]_-{i^\ast}       .
}
\]
Being a right adjoint in a localization, the functor $i$ is final.
Thus, the right lax-commutative triangle is in fact commutative. 
Because a product of final functors is a final functor, then so too is the left lax-commutative triangle in fact commutative.  
The assumption that the diagonal functor $\delta$ is final grants that the bottom lax commutative triangle is in fact commutative.  
Because $i$ is a right adjoint in a localization, the functor $i^\ast$ is fully-faithful; likewise, the functor $(i\times i)^\ast$ is fully-faithful.
We conclude that the upper lax-commutative triangle is in fact commutative, thereby completing this formal proof.

\end{proof}

\begin{cor}\label{thank-god}
The $\infty$-category $\Disk_+(M_\ast)$ is sifted.  

\end{cor}

\begin{proof}

Lemma~\ref{of.over.final} states that $\Disk_+(M_\ast) \to \Disk(\Bsc)_{/M_\ast}$ is a fully-faithful right adjoint.  
Using Lemma~\ref{loc.sifted}, the desired siftedness of $\Disk_+(M_\ast)$ therefore follows from siftedness of $\Disk(\Bsc)_{/M_\ast}$.  
This latter siftedness is Corollary~2.28 of~\cite{aft2}.  

\end{proof}

\begin{lemma}\label{of-to-over}
There is a functor
\begin{equation}\label{eq.of-to-over}
\Disk_+(M_\ast)\longrightarrow \Disk_{n,+/M_\ast}
\end{equation}
that evaluates on objects as
\[
\bigl(B\sqcup U\hookrightarrow M_\ast\bigr)\mapsto \bigl(U_+\xra{(f_{|U})_+} M_\ast \bigr)~;
\] 
here, the restricted embedding $B\hookrightarrow M_\ast$ is such that its image contains the base point of $M_\ast$.  
\end{lemma}

\begin{proof}
The strategy of this proof is as follows.
We first construct an auxiliary category $\diskd_+(M_\ast)$ that localizes to $\Disk_+(M_\ast)$.
We next construct a functor $\diskd_+(M_\ast) \to \Disk_{n,+/M_\ast'}$, where $M_\ast'$ is a zero-pointed manifold that is equivalent to $M_\ast$. 
We finish by observing that this constructed functor factors through the above localization.

Fix a conically smooth open embedding $\sC(L) \hookrightarrow M_\ast$.
Denote the compact subspace 
\[
\ov{\sC}(L):=\ast \underset{L\times \{0\}}\amalg L\times [0,\frac{1}{2}]\subset \sC(L)~;
\]
it is equipped with a topological embedding $\ov{\sC}(L) \hookrightarrow M_\ast$.  
Consider the zero-pointed manifold 
\[
M_\ast'~:=~\ast \underset{\ov{\sC}(L)} \amalg M_\ast~.
\]
Note that $M_\ast'$ is canonically equipped with a zero-pointed embedding $q\colon M_\ast \to M_\ast'$.  
Choose a smooth family of self-embeddings $\varphi_t\colon [0,1) \to [0,1)$, with $t \in [0, \frac{1}{2}]$, with the following properties.
The image of $\varphi_t$ is $[t,1)$; for each $t$, the map $\varphi_t$ is the identity near $1$; the map $\varphi_0 = {\sf id}$ is the identity map.
For $t=\frac{1}{2}$, the map $\varphi_{\frac{1}{2}}$ determines a zero-pointed embedding $q'\colon M_\ast' \to M_\ast$ as it affects the cone coordinate.
The family $\varphi_t$ for $t\in [0,\frac{1}{2}]$ implements identifications $q\circ q' \simeq {\sf id}_{M_\ast'}$ in the mapping space $\ZEmb(M_\ast',M_\ast')$ and $q'\circ q\simeq {\sf id}_{M_\ast}$ in $\ZEmb(M_\ast,M_\ast)$.  
We conclude that the zero-pointed embedding $q\colon M_\ast \to M_\ast'$ is an equivalence in the $\infty$-category $\ZMfld_n$.

Consider the poset $\diskd(\bsc)_{/M_\ast}$ whose objects are finite disjoint unions of basics conically smoothly and openly embedded into $M_\ast$, and whose morphisms are inclusions between basics embedded in $M_\ast$.  
Consider the full subposet $\diskd_+(M_\ast) \subset \diskd(\bsc)_{/M_\ast}$ consisting of those $(V \subset  M_\ast)$ that satisfy the following two conditions.
First, the base point belongs to $V$.
Second, the component $B\subset V$ containing the base point is contained in $\ov{\sC}(L)$.

Consider the subposet
\[
\cI~\subset~ \diskd_+(M_\ast)
\]
consisting of the same objects and those inclusions between finite disjoint unions of basics embedded in $M_\ast$ that are abstractly isotopy equivalences.
It is manifest is that the evident functor $\diskd_+(M_\ast) \to \Disk_+(M_\ast)$ factors through the localization
\[
\diskd_+(M_\ast) \longrightarrow \diskd_+(M_\ast)[\cI^{-1}]  \longrightarrow \Disk_+(M_\ast)~.
\]
It follows from Proposition~2.22 of~\cite{aft2} that the rightmost functor in the above display is an equivalence between $\infty$-categories.

We now construct a functor
\[
\diskd_+(M_\ast) \longrightarrow \Disk_{n,+/M_\ast'}~.
\]
Let $(e\colon B\sqcup U \hookrightarrow M_\ast)$ be an object in the domain; it is, in particular, the datum of a conically smooth open embedding to $M_\ast$.
To this object we assign the object $(qe\colon U_+\to M_\ast'$) of the codomain; it is the composite zero-pointed embedding
\[
qe\colon U_+ \xra{~(e_{|U})_+~} M_\ast \xra{~q~} M'_\ast~.
\]  
Now let $(e\colon B\sqcup U \hookrightarrow M_\ast) \to (e'\colon B'\sqcup U'\to M_\ast)$ be a morphism in the domain; it is, in particular, the datum of a conically smooth open embedding $f\colon B\sqcup U \hookrightarrow B'\sqcup U'$ over $M_\ast$. 
To this morphism we assign the morphism in the codomain which is the composite zero-pointed embedding 
\[
qf\colon U_+ \xra{\rm collapse} f^{-1}(U')_+ \xra{~f~} U_+' \xra{~e'~} M_\ast 
\]
-- it is quick to verify commutativity of the diagram
\[
\xymatrix{
U_+  \ar[rr]^-{qf}  \ar[dr]_-{qe}  
&&
U'_+  \ar[dl]^-{qe'}  
\\
&
M_\ast'
&
.
}
\]
Further, this assignment on objects and on morphisms respects compositions of morphisms, and carries identity morphisms to identity morphisms.  
In conclusion, we constructed a functor
\begin{equation}\label{the}
\diskd_+(M_\ast) \longrightarrow \Disk_{n,+/M_\ast}~.  
\end{equation}

By direct inspection, the functor~(\ref{the}) carries isotopy equivalences to equivalences.  
Consequently, there results a functor from the localization
\[
\diskd_+(M_\ast)[\cI^{-1}] \longrightarrow \Disk_{n,+/M_\ast'}~.  
\]
The desired functor~(\ref{eq.of-to-over}) is obtained by precomposing this functor by the equivalence
$\Disk_+(M_\ast) \xla{\simeq} \diskd_+(M_\ast)[\cI^{-1}]$, established above, and postcomposing this functor by the equivalence
$\Disk_+(M_\ast')\underset{\simeq}{\xla{~q~}} \Disk_+(M_\ast)$ induced by the equivalence $M_\ast \xra{q}M_\ast'$ above.  
Inspecting the values of the functor~(\ref{the}) on objects validates the asserted values of the functor
~(\ref{eq.of-to-over})
on objects.

\end{proof}

We will prove the next result, which makes reference to Lemma~\ref{of-to-over}, as~\S\ref{EE-of-M-proof}. 
\begin{prop}\label{EE-of-M}
The functor
\begin{equation}\label{of-M-to-slice}
\Disk_+(M_\ast)\longrightarrow \Disk_{n,+/M_\ast}
\end{equation}
is final. 
Likewise, the functor
\[
\Disk^+(M_\ast) \longrightarrow (\Disk_n^+)^{M_\ast/}
\]
is initial.

\end{prop}

Consider the composite functor
\begin{equation}\label{extend-of-M}
\Alg^{\sf aug}_n(\cV) \longrightarrow \Fun(\Disk_{n,+/M_\ast}, \cV)\longrightarrow \Fun\bigl(\Disk_+(M_\ast), \cV\bigr)~:
\end{equation}
the first arrow is restriction along the projection $\Disk_{n,+/M_\ast} \to \Disk_{n,+}$; the second arrow is restriction along the functor of Proposition~\ref{EE-of-M}.  

\begin{notation}\label{just-as-A}
Given an augmented $n$-disk algebra $A\colon \Disk_{n,+} \to \cV$, we will use the same notation 
$
A\colon \Disk_+(M_\ast)\to \cV
$
for the value of the functor~(\ref{extend-of-M}) on $A$.  

\end{notation}
We content ourselves with this Notation~\ref{just-as-A} because of the following immediate corollary of Proposition~\ref{EE-of-M}.
\begin{cor}\label{still-computes}
Let $\cV$ be a symmetric monoidal $\infty$-category whose underlying $\infty$-category admits sifted colimits.
Let $A\colon \Disk_{n,+}\to \cV$ be an augmented $n$-disk algebra, and let $C\colon \Disk_n^+\to \cV$ be an augmented $n$-disk coalgebra.  There are canonical identifications in $\cV$:
\[
\int_{M_\ast}  A 
~{}~\simeq~{}~
\colim\Bigl( \Disk_+(M_\ast) \xra{~A~} \cV \Bigr)
~{}~\simeq~{}~
\underset{(B\sqcup U\hookrightarrow M_\ast)\in \Disk_+(M_\ast)}\colim~ A(U_+)~,
\]
and
\[
\int^{M_\ast}  C
~{}~\simeq~{}~\limit\Bigl( \Disk^+(M_\ast) \xra{~C~} \cV \Bigr)
~{}~\simeq~{}~
\underset{(B\sqcup V\hookrightarrow M_\ast)\in \Disk^+(M_\ast)}\limit~ C(V^+)~,
\]
where the righthand expressions for the objects of the indexing categories are from Remark~\ref{rem.obj}.
\end{cor}

\subsection{Proofs of Theorem~\ref{fact-functor} and Proposition~\ref{EE-of-M}}\label{EE-of-M-proof}

Fix a weakly pre-constructible bundle $\pi \colon M_\ast \to [0,1]$ between stratified spaces (see \cite{aft1}) for which $\pi^{-1}0$ contains a neighborhood of $\ast$.
Note that the subspace $\pi^{-1}(\frac{1}{2})\subset M$ is a smooth $(n-1)$-dimensional submanifold. 
For each open subset $U\subset [0,1]$, consider the subspace $\pi^{-1}(U)_\ast:=\pi^{-1}(U)_\ast:= \{\ast\}\cup \pi^{-1}(U) \subset M_\ast$.
The assumption on $\pi$ is just so that $\pi^{-1}(U)_\ast$ is equipped with a canonical structure of a zero-pointed manifold with respect to which the inclusion $\pi^{-1}(U)_\ast \to M_\ast$ is a zero-pointed embedding.
Consider the morphism between $\infty$-operads
\[
\diskd_{1/[0,1]}^{\partial, \sf or}\xra{~\pi^{-1}~} \ZMfld_{n/M_\ast}~,\qquad [0,1]\supset U\mapsto \pi^{-1}(U)_\ast
~.
\]
Notice that $\pi^{-1}(U)_\ast = \pi^{-1}(U)_+$ if $0\notin U$.  
We have the composite morphism between $\infty$-operads
\[
 \diskd_{1/[0,1]}^{\partial, \sf or}  \xra{~\pi^{-1}~} \ZMfld_{n/M_\ast}   \xra{~\Disk_{n,+/-}~}   \\ (\Cat_{\oo})_{/\Disk_{n,+/M_\ast}}~.
\]
The next result makes reference to the colimit of this functor.

\begin{lemma}\label{interval-push}

The canonical functor
\begin{equation}\label{l-pre-excision}
\underset{(U\hookrightarrow [0,1])\in \diskd_{1/[0,1]}^{\partial, \sf or}}\colim ~\Disk_{n,+/\pi^{-1}(U)_\ast}  \longrightarrow~ \Disk_{n,+/M_\ast}~.
\end{equation}
is an equivalence between $\oo$-categories.  

\end{lemma}

\begin{proof}
Since (\ref{l-pre-excision}) is a functor between right fibrations over $\Disk_{n,+}$, it is enough to show, for each finite set $I$, that the induced map between fiber spaces
\[
\underset{U\hookrightarrow [0,1]} \colim \ZEmb\bigl((\RR^n_+)^{\bigvee I}, \pi^{-1}(U)_\ast\bigr) \longrightarrow \ZEmb\bigl((\RR^n_+)^{\bigvee I}, M_\ast\bigr)
\]
is an equivalence. 

Let $I$ be a finite set.  
Consider a zero-pointed $n$-manifold $Z_\ast$.  
For $\DD^n\subset \RR^n$ the closed $n$-disk, consider the topological space $\ZEmb\bigl((\DD^n_+)^{\bigvee I}, Z_\ast\bigr)$ which is the subspace of the topological space of pointed maps (with the 
compactly generated weak Hausdorff replacement of the subspace topology of the compact-open topology) consisting of those $f\colon (\DD^n_+)^{\bigvee I} \to Z_\ast$ for which the restriction $f_|\colon f^{-1}Z \to Z$ is a smooth embedding.  
In a standard manner, the evident restriction $\ZEmb\bigl((\RR^n_+)^{\bigvee I}, Z_\ast\bigr)\xra{\simeq}\ZEmb\bigl((\DD^n_+)^{\bigvee I}, Z_\ast\bigr)$ is a weak homotopy equivalence, and it is functorial in the argument $Z_\ast$.  
So it is enough to argue that the likewise map between spaces as displayed above in which each instance of $\RR^n$ is replaced by one of $\DD^n$, is an equivalence between spaces.  

Each map $\pi^{-1}(U)_\ast \hookrightarrow M_\ast$ appearing in the above colimit is an open embedding.
It follows from the topology on the set of zero-pointed embeddings that the collection
\begin{equation}\label{zemb-hyper}
\Bigl\{\ZEmb\bigl((\DD^n_+)^{\bigvee I}, \pi^{-1}(U)_\ast\bigr) \longrightarrow \ZEmb\bigl((\DD^n_+)^{\bigvee I}, M_\ast\bigr)\mid (U\hookrightarrow[0,1])\in\diskd^{\partial, \sf or}_{1/[0,1]} \Bigr\}
\end{equation}
is comprised of open embeddings among topological spaces.  
Consider the union 
\[
A ~\subset ~ \ZEmb\bigl((\DD^n_+)^{\bigvee I}, M_\ast\bigr)
\]
of these open embeddings.
We next show that the open inclusion $A\hookrightarrow \ZEmb\bigl((\DD^n_+)^{\bigvee I}, M_\ast\bigr)$ is a weak homotopy equivalence.  
For this, it is sufficient to prove that the following statement is true.
\begin{itemize}
\item[$(\dagger)$]
Let $D$ be a closed disk in some Euclidean space; let $\partial D \subset D$ be its boundary sphere.  
For each pair of horizontal continuous maps making the diagram
\[
\xymatrix{
\partial D  \ar[rr]^-{F_{|\partial D}}  \ar[d]
&&
A  \ar[d]
\\
D  \ar[rr]_-{F}   \ar@{-->}[urr]^-{\w{f}}
&&
\ZEmb\bigl((\DD^n_+)^{\bigvee I}, M_\ast\bigr)
}
\]
commute, there is a dashed continuous map for which the resulting triangles commute up to homotopy.  

\end{itemize}
Consider the continuous map
\[
\sigma\colon (0,1]\times \ZEmb\bigl((\DD^n_+)^{\bigvee I}, M_\ast\bigr) 
\longrightarrow
\ZEmb\bigl((\DD^n_+)^{\bigvee I}, M_\ast\bigr)~,
\]
\[
\bigl( \epsilon~ , ~ (\DD^n_+)^{\bigvee I} \xra{f} M_\ast \bigr)  \mapsto \bigl(  \sigma_\epsilon(f)  \colon (\DD^n_+)^{\bigvee I} \xra{v\mapsto \epsilon \cdot v} (\DD^n_+)^{\bigvee I} \xra{f} M_\ast   \bigr)  ~,
\]
given by pre-scaling a zero-pointed embedding.  
The statement~$(\dagger)$ is implied by the following statement.
\begin{itemize}
\item[$(\dagger \dagger)$]
Let $F\colon K \to \ZEmb\bigl((\DD^n_+)^{\bigvee I}, M_\ast\bigr)$ be a continuous map from a compact topological space.  
There is a $\epsilon_0 \in (0,1]$ for which there is a continuous factorization as in the commutative diagram:
\[
\xymatrix{
(0,\epsilon]  \times K  \ar@{-->}[rr]  \ar[d]
&&
A  \ar[d]
\\
(0,1]  \times  K  \ar[r]^-{ F }
&
(0,1]\times \ZEmb\bigl((\DD^n_+)^{\bigvee I}, M_\ast\bigr)   \ar[r]^-{\sigma}
&
\ZEmb\bigl((\DD^n_+)^{\bigvee I}, M_\ast\bigr).
}
\]

\end{itemize}
We proceed, now, to prove statement~$(\dagger \dagger)$. 
So let $F\colon K \to  \ZEmb\bigl((\DD^n_+)^{\bigvee I}, M_\ast\bigr)$ be a continuous map from a compact topological space.

The canonical projection $(\DD^n_+)^{\vee I} \to I_+$ has a preferred section $I_+ \to (\DD^n_+)^{\vee I}$ that selects the center $0\in \DD^n$ of each disk.  
Precomposing by this section defines a continuous map
\[
{\sf ev}_0\colon \ZEmb\bigl((\DD^n_+)^{\bigvee I}, M_\ast\bigr) 
\longrightarrow
\Map^{\ast/}(I_+ , M_\ast)~.
\]

Let $k\in K$.
Consider the zero-pointed embedding $F_k\colon (\DD^n_+)^{\vee I} \to M_\ast$, which is the value of $F$ on $k$.
Choose an element $\bigl( U_k \hookrightarrow [0,1]\bigr) \in \Disk_{1/[0,1]}$ for which $\pi( {\sf ev}_0(F_k)) \subset U_k$.  
There is a $\epsilon_k \in (0,1]$ for which there is a continuous factorization as in the commutative diagram:
\[
\xymatrix{
(0,\epsilon_k]  \times (\DD^n_+)^{\vee I}  \ar@{-->}[rrr]  \ar[d]
&&&
\pi^{-1} U_k  \ar[d]
\\
(0,1]  \times  (\DD^n_+)^{\vee I}   \ar[rr]^-{(\epsilon , v) \mapsto \epsilon v}  
&&
(\DD^n_+)^{\vee I}    \ar[r]^-{ F_k }
&
M_\ast .
}
\]
Because $F$ is continuous, the given topology on $ \ZEmb\bigl((\DD^n_+)^{\bigvee I}, M_\ast\bigr)$ grants the existence of an open neighborhood $k\in W_k \subset K$ for which the restriction of the adjoint of $F$ factors as in the commutative diagram:
\[
\xymatrix{
W_k\times (0,\epsilon_k]  \times (\DD^n_+)^{\vee I}  \ar@{-->}[rrr]  \ar[d]
&&&
\pi^{-1} U_k  \ar[d]
\\
K\times (0,1]  \times  (\DD^n_+)^{\vee I}   \ar[rr]^-{(\epsilon , v) \mapsto \epsilon v}  
&&
K\times (\DD^n_+)^{\vee I}    \ar[r]^-{ F }
&
M_\ast .
}
\]
Choose such a $W_k$ for each $k\in K$.
The colleciton $\{W_k\}_{k\in K}$ is an open cover of $K$.
Using that $K$ is compact, choose a finite subset $\{k_1,\dots,k_r\}\in K$ for which $\{W_{k_j}\}_{1\leq j\leq r}$ is an open cover of $K$.
Choose $\epsilon >0$ less than each $\epsilon_{k_j}$.  
The above diagram then determines, for each $1\leq j \leq r$, a continuous factorization as in the commutative diagram:
\[
\xymatrix{
W_{k_j}\times (0,\epsilon]  \times (\DD^n_+)^{\vee I}  \ar@{-->}[rrr]  \ar[d]
&&&
\pi^{-1} U_{k_j}  \ar[d]
\\
K\times (0,1]  \times  (\DD^n_+)^{\vee I}   \ar[rr]^-{(\epsilon , v) \mapsto \epsilon v}  
&&
K\times (\DD^n_+)^{\vee I}    \ar[r]^-{ F }
&
M_\ast .
}
\]
For $1\leq j\leq r$, this diagram is adjoint to a continuous factorization as in the commutative diagram:
\[
\xymatrix{
(0,\epsilon]  \times W_{k_j}  \ar@{-->}[rr]  \ar[d]
&&
\ZEmb\bigl((\DD^n_+)^{\bigvee I}, \pi^{-1}(U_{k_j})_\ast\bigr)  \ar[d]
\\
(0,1]  \times  K  \ar[r]^-{ F }
&
(0,1]\times \ZEmb\bigl((\DD^n_+)^{\bigvee I}, M_\ast\bigr)   \ar[r]^-{\sigma}
&
\ZEmb\bigl((\DD^n_+)^{\bigvee I}, M_\ast\bigr).
}
\]
Taking a union indexed by $1\leq j\leq r$ determines the sought commutative diagram~$(\dagger \dagger)$.
This completes the proof that the continuous map $A\hookrightarrow \ZEmb\bigl((\DD^n_+)^{\bigvee I}, M_\ast\bigr)$ is a weak homotopy equivalence.

It remains to show that the union $A$ is a homotopy colimit of its terms in~(\ref{zemb-hyper}).  
The collection of open embeddings $\bigl\{ U\hookrightarrow [0,1]\}$, indexed by the objects of $\diskd^{\partial, \sf or}_{1/[0,1]}$, is a hypercover of $[0,1]$.  
It follows that the collection $\{\pi^{-1}(U)_\ast \hookrightarrow M_\ast\}$, too, is a hypercover; and thereafter that the collection~(\ref{zemb-hyper}), too, is a hypercover of $A$.  
It follows from A.3.1 of~\cite{HA} that the canonical map
\[
\underset{U\hookrightarrow [0,1]} \colim \ZEmb\bigl((\DD^n_+)^{\bigvee I}, \pi^{-1}(U)_\ast\bigr) \longrightarrow A
\]
is an equivalence in $\Spaces$.

\end{proof}

Now, consider the likewise composite representation
\[
\diskd_{1/[0,1]}^{\partial, \sf or}\xra{~f^{-1}~} \snglrd_{n/M_\ast}\xra{~\Disk(\bsc)_{/-}~} (\Cat_{\oo})_{/\Disk(\bsc)_{/M_\ast}}~.
\]
A main result of~\cite{aft2} (Corollary~2.38) states that the likewise canonical functor from the colimit
\begin{equation}\label{aft2-excision}
\underset{(U\hookrightarrow [0,1])\in \diskd_{1/[0,1]}^{\partial, \sf or}}\colim ~\Disk(\bsc)_{/f^{-1}(U)}
\xra{~\simeq~} 
\Disk(\bsc)_{/M_\ast}
\end{equation}
is an equivalence. 
We highlight the following consequence of this equivalence.
Denote the full subcategory 
\[
\diskd_{1/(0\in [0,1])}^{\partial, \sf or}
~\subset~ 
\diskd_{1/[0,1]}^{\partial, \sf or}
\]
consisting of those $U\hookrightarrow [0,1]$ for which $0\in U$.  
\begin{lemma}\label{of-M-excision}
The equivalence~(\ref{aft2-excision}) restricts as an equivalence between $\oo$-categories:
\begin{equation}\label{aft2-pre-excision}
\underset{(0\in U\hookrightarrow [0,1])\in \diskd_{1/(0\in [0,1])}^{\partial, \sf or}}\colim ~\Disk(\bsc)_{/f^{-1}(U)}
\overset{\sim}\longrightarrow \Disk_+(M_\ast)~.
\end{equation}
\end{lemma}
\begin{proof}
The identification of the colimit~(\ref{aft2-excision}) is one in the $\infty$-category of right fibrations over $\Disk(\Bsc)$.  
Therefore, it restricts as an equivalence in the $\infty$-category of right fibrations over the full $\infty$-subcategory of $\Disk(\Bsc)$ consisting of the objects in the image of the forgetful functor $\Disk_+(M_\ast) \hookrightarrow \Disk(\Bsc)_{/M_\ast} \to \Disk(\Bsc)$. 
The resulting restricted identification is the desired one.

\end{proof}

\begin{lemma}\label{finallemma}
Let $H\colon \cK \times [1] \to \Cat_\infty$ be a natural transformation between functors.
Suppose, for each $k\in \cK$, that the restriction $H_|\colon \{k\}\times [1] \to \Cat_\infty$ selects a final functor between $\infty$-categories.
The canonical functor between colimits
\[
H_{0<1}\colon \colim(\cK\xra{H_0}\Cat_\infty) \longrightarrow \colim(\cK\xra{H_1} \Cat_\infty)
\]
is final.
\end{lemma}

\begin{proof}
It suffices to show that, for every functor 
$
F: \underset{k\in \cK}\colim H_1(k) \longrightarrow \cX
$
that admits a colimit,
the canonical morphism in $\cX$,
\[
\colim\Bigl(\underset{k\in \cK}\colim H_0(k) \xra{F\circ H_{0<1}} \cX\Bigr)\longrightarrow \colim\Bigl(\underset{k\in \cK}\colim H_1(k) \xra{F}\cX\Bigr)~,
\]
is an equivalence.
This assertion follows from the sequence of equivalences in $\cX$,
\[
\colim\Bigl(\underset{k\in \cK}\colim H_0(k) \xra{F\circ H_{0<1}} \cX\Bigr)
~\simeq~ 
\colim_{k\in\cK}\Bigl(\colim\bigl(H_0(k) \xra{H_{0<1}(k)\circ F_{|\{k\}}} \cX\bigr)\Bigr) 
\]
\[
\xra{~\simeq~}
\colim_{k\in\cK}\Bigl(\colim\bigl(H_1(k) \xra{F_{|\{k\}}} \cX\bigr)\Bigr) 
~\simeq ~
\colim\Bigl(\underset{k\in \cK}\colim H_1(k)\xra{F} \cX\Bigr)~,
\]
which we now explain.
The middle equivalence is the finality of $H_0(k) \ra H_1(k)$ for each $k\in \cK$.
The outer equivalences use a formal commutation of colimits: left Kan extensions compose.

\end{proof}

\begin{proof}[Proof of Proposition \ref{EE-of-M}]
The two assertions in the statement of the proposition are equivalent.  
To see this we note the following identifications.  
First, there is the definitional identification $\Disk_+(M_\ast^\neg) \simeq \Disk^+(M_\ast^\neg)^{\op}$.  
Second is the identification $\neg\colon \ZMfld_n \simeq \ZMfld_n^{\op}$ of Observation~\ref{mfld.neg}, which lies under an identification $\neg\colon \Disk_{n,+}\simeq (\Disk_n^+)^{\op}$.

There is a natural transformation between $\Cat_{\oo}$-valued functors on $\Disk_{1/(0\in [0,1])}^{\partial, \sf or}$ which assigns to each $U \in \Disk_{1/(0\in [0,1])}^{\partial, \sf or}$ the functor
\begin{equation}\label{pre.U}
\Disk(\bsc)_{/f^{-1}(U)}\longrightarrow \Disk_{n,+/f^{-1}(U)_\ast} 
\end{equation}
given by Lemma~\ref{of-to-over}. 
We prove that this functor~(\ref{pre.U}) is final for each such $U$.
There are two cases.
Suppose $0\in U$.
In this case, $f^{-1}(U)$ is itself a finite disjoint unions of basics.
Consequently, the identity morphism $(f^{-1}(U) \xra{=} f^{-1}(U))$ is a final object in the $\infty$-category $\Disk(\bsc)_{/f^{-1}(U)}$.
Also in this case, the morphism between zero-pointed manifolds $\bigl(f^{-1}(U) \xra{=} f^{-1}(U)\bigr)$ is a final object in the $\infty$-category $\Disk_{n,+/f^{-1}(U)_\ast}$.
Clearly, the functor~(\ref{pre.U}) carries the first final object to the second.
We conclude that the functor~(\ref{pre.U}) is final in this case that $0\in U$, as desired.
\\
We now consider the other case: $0\notin U$.
In this case, the natural functor $\Disk_{n/f^{-1}(U)} \hookrightarrow \Disk(\bsc)_{/f^{-1}(U)}$ is an equivalence between $\infty$-categories.  
Also in this case, the canonical zero-pointed embedding $f^{-1}(U)_+ \to f^{-1}(U)_\ast$ is an equivalence between zero-pointed manifolds.  
In this way, we identify~(\ref{pre.U}) as the standard functor
\[
(-)_+\colon \Disk_{n/f^{-1}(U)} \longrightarrow \Disk_{n,+/f^{-1}(U)_+}~.
\]
Lemma~\ref{actually.loc} gives that the functor~(\ref{pre.U}) is a fully-faithful right adjoint.
In particular, the functor~(\ref{pre.U}) is final in this case that $0\notin U$, as desired.
We conclude that the functor~(\ref{pre.U}) is final in all cases for $U$.

Now, applying Lemma~\ref{finallemma}, we obtain from the conclusion of the previous paragraph that the functor
\[
\underset{(0\in U\hookrightarrow [0,1])\in \diskd_{1/(0\in [0,1])}^{\partial, \sf or}}\colim ~\Disk(\bsc)_{/f^{-1}(U)}
\longrightarrow
\underset{(U\hookrightarrow [0,1])\in \diskd_{1/[0,1]}^{\partial, \sf or}}\colim ~\Disk_{n,+/f^{-1}(U)_\ast}
\]
is final. Applying Lemma \ref{of-M-excision} and Lemma~\ref{interval-push} we identify this functor as
\[
\Disk_+(M_\ast)\longrightarrow \Disk_{n,+/M_\ast}~,
\]
which we thus conclude is a final functor.

\end{proof}

\begin{proof}[Proof of Theorem~\ref{fact-functor}]
We only concern ourselves with statement~(1), for statement~(2) follows from statement~(1) upon replacing $\cV$ by $\cV^{\op}$.  
Corollary~\ref{still-computes} states that factorization homology $\int_{M_\ast}A$ can be computed as a colimit over the $\infty$-category $\Disk_+(M_\ast)$;
Corolary~\ref{thank-god} states that this $\infty$-category is sifted.
In this way, we conclude that the factorization homology functor $\int_-\colon \Fun(\Disk_{n,+},\cV) \to \Fun(\ZMfld_n,\cV)$ exists provided $\cV$ admits sifted colimits.  

To argue the existence of the factorization homology functor $\int_-\colon \Alg_n^{\sf aug}(\cV) \to \Fun^\ot(\ZMfld_n,\cV)$ over the one examined in the previous paragraph, we appeal to Lemma~2.16 of~\cite{aft2}.
Namely, must show that the functor between $\infty$-categories
\begin{equation}\label{equiv?}
\Disk_+(M_\ast)\times \Disk_+(M_\ast')\longrightarrow \Disk_+(M_\ast\vee M_\ast')~,
\end{equation}
which sends a pair $(U\sqcup \sC(L)\hookrightarrow M_\ast)$ and $(U'\sqcup \sC(L')\hookrightarrow M_\ast')$ to
$(U\sqcup U'\sqcup \sC(L\sqcup L')\hookrightarrow M_\ast\vee M_\ast')$,
is final.

Consider the poset $\diskd_+(M_\ast\vee M_\ast')$ of open neighborhoods of $\ast \in M_\ast \vee M_\ast'$ that are abstractly isomorphic to a finite disjoint union of basics, and inclusions among them.  
Inside the proof of Lemma~\ref{of-to-over} we show that the $\infty$-categorical localization of this poset at those inclusions that are isotopy equivalences, is canonically equivalent to the $\infty$-category $\Disk_+(M_\ast \vee M_\ast')$.
Consider the functor
\[
\diskd_+(M_\ast\vee M_\ast') \longrightarrow   \Disk_+(M_\ast)\times \Disk_+(M_\ast')
\]
whose projection onto the first factor is given by 
\[
(B\sqcup U \hookrightarrow M_\ast\vee M_\ast')\mapsto  \bigl((B\sqcup U)\setminus M'\hookrightarrow (M_\ast \vee M_\ast')\setminus M' = M_\ast\bigr)~,
\]
and whose projection onto the second factor is similar. 
Notice that this functor carries isotopy equivalences to equivalences.  
Therefore, this functor canonically determines a functor from the localization:
\[
\Disk_+(M_\ast\vee M_\ast') \longrightarrow   \Disk_+(M_\ast)\times \Disk_+(M_\ast')~.
\]
By direct inspection, this functor is an inverse to the functor~(\ref{equiv?}).  
In particular, the functor~(\ref{equiv?}) is final.

\end{proof}

\subsection{Reduced homology theories}\label{sec.reduced}

We use zero-pointed manifolds to implement additional functorialities of \emph{reduced} homology theories.  

Recall the symmetric monoidal topological categories $\Disk_n^\partial\subset \Mfld_n^{\partial}$ of Example~\ref{boundary-examples}.
The concept of a homology theory for smooth $n$-manifolds with boundary is defined in~\cite{aft2} -- this is a symmetric monoidal functor $H\colon \Mfld_n^{\partial} \to \cV$ satisfying an $\ot$-excision axiom.  
We concern ourselves with an augmented version of this notion, defined momentarily.

For the following definition, recall the notion of a collar-gluing from~\cite{aft1} (Definition~8.3.2), and of augmented symmetric monoidal functors (Definition~\ref{def.augmented} of this paper).

\begin{definition}[Reduced homology theories]\label{def.reduced}
For a symmetric monoidal $\oo$-category $\cV$, the $\oo$-category of \emph{augmented homology theories} for $n$-manifolds with boundary is the full $\oo$-subcategory
\[
\bH^{\sf aug}\bigl(  \Mfld_n^{\partial} , \cV\bigr)
~\subset ~
\Fun^{\ot, \sf aug}(\Mfld_n^{\partial}, \cV)
\]
consisting of those augmented symmetric monoidal functors $H\colon \Mfld_n^\partial \to \cV$ that satisfy the following
\begin{itemize}
\item {\bf $\ot$-Excision:} 
For $M \cong M_L \underset{\RR\times M_0} \bigcup M_R$ a collar-gluing among manifolds with boundary, the canonical morphism in $\cV$,
\begin{equation}\label{ot-excision}
H(M_L)\underset{H(M_0)}\bigotimes H(M_R)\xra{~\simeq~} H(M)~,
\end{equation}
is an equivalence.
\end{itemize}
The $\infty$-category of \emph{reduced} homology theories is the full $\oo$-subcategory
\[
\bH^{\sf aug}_{\sf red}\bigl(\Mfld_n^{\partial}, \cV \bigr)  ~\subset~ \bH^{\sf aug}\bigl(  \Mfld_n^{\partial} , \cV\bigr)
\]
consisting of those $H$ for which, for each finitary smooth $(n-1)$-manifold $N$, the morphism in $\cV$ induced by the augmentation of $H$,
\[
H\bigl(\RR_{\geq 0} \times N\bigr) \xra{~\simeq~}\uno~.
\]

\end{definition}

The following is an immediate consequence of our previous work with Hiro Lee Tanaka (\cite{aft2}).

\begin{prop}\label{reduced-augmented}
Let $\cV$ be a symmetric monoidal $\infty$-category that is $\ot$-sifted cocomplete.
There is a fully-faithful functor 
\[
\Alg_{\Disk_n}^{\sf aug}(\cV) ~\hookrightarrow~ \Alg_{\Disk_n^\partial}^{\sf aug}(\cV)~.
\]
Composing this functor with factorization homology defines an equivalence between $\infty$-categories:
\[
\int_-\colon \Alg_{\Disk_n}^{\sf aug}(\cV) \xra{~\simeq~} \bH_{\sf red}^{\sf aug}\bigl(\Mfld_n^{\partial}, \cV\bigr)~.
\]

\end{prop}
\begin{proof}
Consider the fully-faithful symmetric monoidal functor
\[
i\colon \Disk_{n,+} \longrightarrow \Disk_{n,+}^{\partial}~.
\]
For each object $V_+\in \Disk_{n,+}^\partial$, the $\infty$-undercategory $\Disk_{n,+}^{V_+/}$ has an initial object.
Namely, writing $V_+\simeq U_+ \vee U'_+$ as a wedge sum with each connected component of $U$ diffeomorphic to $\RR^n$ and each connected component of $U'$ diffeomorphic to $\HH^n$, this initial object is the collapse map $(V_+ \xra{c} U_+)$ onto the Euclidean components.  
In this way we conclude that the functor $i$ is a right adjoint in a localization:
\[
q\colon \Disk_{n,+}^\partial \rightleftarrows \Disk_{n,+} \colon i~.
\]
Noting that this left adjoint carries wedge sums to wedge sums, this adjunction is one among symmetric monoidal $\infty$-categories.  
Implementing Proposition~\ref{M.+.correct}, we conclude that the restriction functor
\begin{equation}\label{q.restrict}
q^\ast \colon \Alg_{\Disk_n}^{\sf aug}(\cV) \longrightarrow \Alg_{\Disk_n^\partial}^{\sf aug}(\cV)
\end{equation}
is fully-faithful.
We now identify the image of this fully-faithful functor.

Consider the $\infty$-subcategory $W := q^{-1}\bigl(\Disk_{n,+}\bigr)^\sim \subset \Disk_{n,+}^\partial$ which is the preimage of the maximal $\oo$-subgroupoid.
Because $q$ is a symmetric monoidal left adjoint in a symmetric monoidal localization, the symmetric monoidal functor $q$ canonically factors
\[
\ov{q}\colon \Disk_{n,+}^{\partial}[W^{-1}] \xra{~\simeq~} \Disk_{n,+}
\]
as an equivalence between symmetric monoidal $\infty$-categories. 
By inspection, $W \subset \Disk_{n,+}^{\partial}$ is the smallest symmetric monoidal subcategory containing the equivalences as well as the morphism $\bigl(\HH^n_+\to +\bigr)$.
In this way, we identify the image of~(\ref{q.restrict}) as those augmented $\Disk_n^\partial$-algebras $A\colon \Disk_{n,+}^\partial \to \cV$ that carry each morphism in $W$ to an equivalence in $\cV$, which is to say that $A(\HH^n) \xra{\simeq} \uno$.

Now, the main result (Theorem~2.43) of~\cite{aft2} implies that the adjunction
\[
\int_-\colon \Alg_{\Disk_n^{\partial}}^{\sf aug}(\cV)  
~\rightleftarrows~
\bH^{\sf aug}(\Mfld_n^{\partial}, \cV) \colon {\rm Restriction}
\]
is an equivalence between $\infty$-categories.  
Therefore, restriction defines a fully-faithful functor
\begin{equation}\label{red.rest}
\bH^{\sf aug}_{\sf red}(\Mfld_n^{\partial}, \cV) 
~\hookrightarrow~
\Alg_{\Disk_n^{\partial}}^{\sf aug}(\cV)  ~.
\end{equation}
The conclusion of the previous paragraph, which characterizes the image of the functor~(\ref{q.restrict}), verifies that the fully-faithful functor~(\ref{red.rest}) factors:
\[
\bH^{\sf aug}_{\sf red}(\Mfld_n^{\partial}, \cV) 
~\hookrightarrow~
\Alg_{\Disk_n}^{\sf aug}(\cV) 
~\underset{(\ref{q.restrict})}\hookrightarrow~
\Alg_{\Disk_n^{\partial}}^{\sf aug}(\cV) ~.
\]
This result is proved upon showing that the left fully-faithful functor is surjective.  
This is implied by the following assertion.
\begin{itemize}
\item[~] 
Let $A\colon \Disk_{n,+}^\partial \to \cV$ be an augmented $\Disk_n^\partial$-algebra.
Suppose the augmentation $A(\HH^n) \xra{\simeq} \uno$ is an equivalence.
Then the factorization homology $\int_-A \colon \Mfld_{n,+}^{\partial} \to \cV$ is a \emph{reduced} $\ot$-excisive functor.

\end{itemize}

Let $A$ be as in the above assertion.
Let $N$ be a finitary smooth $(n-1)$-manifold.  
We must show the augmentation $\int_{\RR_{\geq 0}\times N} A \xra{\simeq} \uno$ is an equivalence.  
The assumption on $A$ grants so provided $N$ is isomorphic to a finite (possibly empty) disjoint union of Euclidean $(n-1)$-spaces.  
Because $N$ is finitary, it can be witnessed as via a finite iteration of collar-gluings from $\RR^{n-1}$.  
Let $r$ be the minimal number of such iterations for witnessing $N$.  
We proceed by induction on $r$.  
If $r=0$, then $N=\emptyset$, and thus $\uno \simeq \int_\emptyset A \simeq \int_{\RR_{\geq 0}\times N} A$, as desired.
So assume $r>0$.
Then $N \cong N_R\underset{\RR\times N_0}\bigcup N_L$ where $N_R$ and $\RR\times N_0$ and $N_L$ can be witnessed via less than $r$ iterations of collar-gluings from $\RR^{n-1}$.  
Now, factorization homology for manifolds with boundary satisfies $\ot$-excision, which is to say that the canonical morphism in $\cV_{/\uno}$
\[
\int_{N_L} A\underset{\int_{N_0}A}\bigotimes \int_{N_R} A \xra{~\simeq~} \int_N A
\]
is an equivalence. 
By induction on $r$, the augmentations $\int_{N_L} A \xra{\simeq} \uno$ and $\int_{N_0}A \xra{\simeq} \uno$ and $\int_{N_R} A \xra{\simeq} \uno$ are each equivalences.  
We conclude that the augmentation $\int_N A \xra{\simeq} \uno$ is an equivalence, as desired.

\end{proof}

Consider the full $\infty$-subcategory 
\[
\Mfld_n^{\partial_{\sf cpt}}~\subset ~\Mfld_n^\partial
\]
consisting of those smooth $n$-manifolds with boundary whose boundary is compact.   
Collapsing boundary to a point defines a symmetric monoidal functor
\begin{equation}\label{eq.quotient}
\Mfld_n^{\partial_{\sf cpt}}\longrightarrow \ZMfld_n~,\qquad \ov{M}\mapsto  \ast  \underset{\partial \ov{M}} \amalg \ov{M}~.
\end{equation}

\begin{theorem}\label{reduced-to-cones}
Let $\cV$ be a symmetric monoidal $\infty$-category that is $\ot$-sifted cocomplete.
The diagram of $\infty$-categories
\[
\xymatrix{
&&
\Alg_n^{\sf aug}(\cV)  \ar[dll]_-{\int_-}  \ar[drr]^-{\int_-}  
&&
\\
\Fun^{\ot, \sf aug}\bigl(\Mfld_n^{\partial_{\sf cpt}} , \cV  \bigr)
&&&&   \ar[llll]^-{(\ref{eq.quotient})^\ast}
\Fun^\ot\bigl(\ZMfld_n , \cV \bigr)
}
\]
canonically commutes --
here, the leftward diagonal arrow is through Proposition~\ref{reduced-augmented}; the rightward diagonal arrow term is as in Theorem~\ref{fact-functor}.
In other words, for each augmented $n$-disk algebra $A$, and for each finitary smooth $n$-manifold $\ov{M}$ with compact boundary, there is an equivalence in $\cV$
\[
\int_{\ov{M}} A~\simeq~\int_{\ast \underset{\partial \ov{M}}\amalg \ov{M}}  A
\]
which is functorial in the arguments $\ov{M}$ and $A$.

\end{theorem}

\begin{proof}
Let $A$ be an augmented $n$-disk algebra in $\cV$.
Let $\ov{M}$ be a finitary smooth $n$-manifold with compact boundary.
Denote the zero-pointed $n$-manifold $M_\ast :=\ast \underset{\partial \ov{M}} \amalg \ov{M}$.
The smooth structure on $\ov{M}$ determines a conically smooth structure on $M_\ast$.

We prove this result by establishing the zig-zag of canonical equivalences in $\cV$:
\begin{eqnarray}
\nonumber
\int_{\ov{M}} A
&
\underset{(a)}{\xra{\simeq}}
&
\colim_{\Disk(\bsc)_{/M_\ast}} \pi_\ast A
\\
\nonumber
&
\underset{(b)}{\xla{\simeq}}
&
\colim_{\Disk_+(M_\ast)} (\pi_\ast A)_{|}
\\
\nonumber
&
\underset{(c)}{\xra{\simeq}}
&
\colim_{\Disk_+(M_\ast)} A_{|}
\\
\nonumber
&
\underset{(d)}{\xra{\simeq}}
&
\colim_{\Disk_{n,+/M_\ast}} A~\simeq~\int_{M_\ast} A~,
\end{eqnarray}
which we explain as we go.
The canonical map $\pi\colon \ov{M} \to M_\ast$ is a constructible bundle.  
Consequently, the pushforward formula for factorization homology (Theorem~2.25 of ~\cite{aft2}) can be applied, thereby granting the canonical equivalence $(a)$.
Here, $\pi_\ast A$ is the functor
\[
\pi_\ast A\colon \Disk(\bsc)_{/M_\ast} \overset{\pi^{-1}}
\longrightarrow 
\Mfld^{\partial}_{n/\ov{M}}\overset{\int A}\longrightarrow \cV ~.
\]
Also here, we denote the composite functor
\[
(\pi_\ast A)_|\colon \Disk_+(M_\ast) \to \Disk(\Bsc)_{/M_\ast} \xra{~\pi_\ast A~} \cV~,\qquad \bigl(\sC(\partial \ov{M})\sqcup U \hookrightarrow M_\ast\bigr)\mapsto \bigl(\int_{\partial \ov{M} \times [0,1)} A\bigr)\otimes A(U)~.
\]
The equivalence $(b)$ therefore follows from the finality of Lemma~\ref{of.over.final}.

Here, we doublebook the notation for the composite functor
\[
A\colon \Disk_{n,+/M_\ast} \longrightarrow \Disk_{n,+} \xra{~A~} \cV~.
\]
We denote the restriction 
\[
A_|\colon \Disk_+(M_\ast) \xra{~(\ref{eq.of-to-over})~} \Disk_{n,+/M_\ast} \longrightarrow \Disk_{n,+} \xra{~A~} \cV~,\qquad \bigl(\sC(\partial \ov{M})\sqcup U\hookrightarrow M_\ast\bigr)\mapsto A(U)~.
\]
The equivalence $(d)$ therefore follows from the finality of Proposition~\ref{EE-of-M}.

Finally, the augmentation of $A$ defines a canonical natural transformation 
\[
(\pi_\ast A)_{|} \longrightarrow A_|
\]
between functors $\Disk_+(M_\ast) \to \cV$.
The second statement of Proposition~\ref{reduced-augmented} implies this natural transformation is by equivalences in $\cV$.
This establishes the equivalence $(c)$.

\end{proof}

\begin{remark}
Theorem~\ref{reduced-to-cones} implies that a reduced homology theory for $n$-manifolds with boundary has additional functorialities.
For instance, consider a properly embedded codimension-zero submanifold $\ov{U}\subset M$ with compact boundary.
For a reduced augmented homology theory $H$ there is a canonically associated morphism
\[
H(M) \longrightarrow H(\ast \underset{\partial \ov{U}}\amalg \ov{U})\simeq \uno_\cV \underset{H(\partial \ov{U})}\bigotimes H(U)~,
\]
which is induced by the zero-pointed embedding $M \to \ast \underset{\partial \ov{U}} \amalg \ov{U}$.  

\end{remark}

In the following statement we consider a collar-gluing $\ov{M} \cong \ov{C}\underset{\RR\times L}\bigcup I$ of a manifold with compact boundary with the property that both $\RR\times L$ and $I$ are disjoint from an open neighborhood of the boundary $\partial \ov{M}$.
In particular, the boundaries $\partial L = \emptyset = \partial I$ are empty, so that the inclusion $\partial \ov{C} \subset \partial \ov{M}$ is an equality.  

\begin{cor}[Reduced factorization (co)homology satisfies $\ot$-(co)excision]\label{reduced-excision}
Let $\cV$ be a symmetric monoidal $\infty$-category.
Let $\ov{M}$ be a smooth $n$-manifold with compact boundary, and let $\ov{M} \cong \ov{C} \underset{\RR\times L}\bigcup I$ be a collar-gluing among smooth $n$-manifolds with boundary for which both $I$ and $L$ are disjoint from an open neighborhood of the boundary $\partial \ov{M}$.
Consider the associated zero-pointed $n$-manifolds $M_\ast:=\ast \underset{\partial \ov{M}} \amalg \ov{M}$ and $C_\ast := \ast \underset{\partial \ov{M}} \amalg \ov{C}$ and $I_+$ and $L_+$.  
Provided $\cV$ is $\ot$-sifted cocomplete, for each augmented $n$-disk algebra $\cA$ in $\cV$, there is a canonical equivalence in $\cV$,
\[
\int_{C_\ast}  {\cA}~{}~  \underset{\int_{L_+} {\cA}} \bigotimes ~\int_{I_+} {\cA}
\xra{~\simeq~}
\int_{M_\ast} {\cA}~,
\]
among reduced factorization homologies, from a two-sided bar construction.
Likewise, provided $\cV^{\op}$ is $\ot$-sifted cocomplete, for each augmented $n$-coalgebra $\cC$ in $\cV$, there is a canonical equivalence in $\cV$
\[
\int^{M_\ast} \cC
\xra{~\simeq~}
\int^{C_\ast}  \cC~{}~  \overset{\int^{L_+} {\cC}} \bigotimes ~\int^{I_+} {\cC}
\]
among reduced factorization cohomologies, to a two-sided cobar construction.

\end{cor}

\begin{proof}
Replacing $\cV$ by $\cV^{\op}$ implements an equivalence between the two statements of the theorem.
So we only prove the first statement.
By way of Theorem~\ref{reduced-to-cones}, the problem is to prove that the canonical morphism 
\[
\int_{\ov{C}}  {\cA}~{}~  \underset{\int_{L} {\cA}} \bigotimes ~\int_{I} {\cA}
\xra{~\simeq~}
\int_{\ov{M}} {\cA}
\]
in $\cV$ is an equivalence. 
This is the case because factorization homology for smooth $n$-manifolds with boundary satisfies $\ot$-excision (Corollary~2.40 of~\cite{aft2}).

\end{proof}

\begin{example}\label{augmentation-excision}
Let $\cV$ be a $\ot$-sifted cocomplete symmetric monoidal $\infty$-category, and let $A$ be an augmented $\Disk_n$-algebra in $\cV$.
Let $\ov{M}$ be a smooth manifold with compact boundary.
As in the proof of Theorem~\ref{reduced-to-cones}, there is a constructible bundle $\ov{M} \to M_\ast$ which restricts to the interior as a diffeomorphism onto $M$.
Corollary~\ref{reduced-excision} gives the identification
\[
\int_{M_\ast} A~\simeq~ \uno_\cV \underset{\int_{\partial \ov{M}} A}\bigotimes \int_{M} A~.
\]

\end{example}

\section{Duality}
Our setup is ripe for depicting a number of dualities: we will see Koszul duality among $n$-disk (co)algebras, as well as Poincar\'e duality among manifolds.  
Here, we recover a twisted version of Atiyah duality.  
\\

\noindent
In this section we fix the following parameters.
\begin{itemize}
\item A dimension $n$.
\item A symmetric monoidal $\infty$-category $\cV$ whose underlying $\infty$-category admits sifted colimits and cosifted limits.
\end{itemize}

\subsection{Poincar\'e/Koszul duality map}
We now construct the \emph{Poincar\'e/Koszul duality map}.

Recall Definition~\ref{Mfld.+}, introducing the $\infty$-categories $\ZMfld_n$, $\Mfld_{n,+}$, and $\Mfld_n^+$.  
Consider the solid diagram of $\infty$-categories
\begin{equation}\label{notation-+}
\xymatrix{
\Fun(\Disk_{n,+},\cV)     \ar@{-->}@(-,u)[rr]^-{\int_-}   
&&
\Fun\bigl(\ZMfld_n,\cV\bigr)    \ar[ll]^-{~(-)_+~}       \ar[rr]^-{~(-)^+~}
&&
\Fun(\Disk_n^{+},\cV)      \ar@{-->}@(-,d)[ll]^-{\int^{-}}   
}
\end{equation}
given by the evident restrictions.
By way of Theorem~\ref{fact-functor}, the assumption that the underlying $\infty$-category of $\cV$ is sifted cocomplete and cosifted complete grants that the left functor has a left adjoint and the right functor has a right adjoint, as indicated by the dashed arrows.
There results a functor involving the arrow $\infty$-category of $\cV$
\[
\Fun^\ot\bigl(\ZMfld_n, \cV) \longrightarrow \Fun\bigl(\ZMfld_n, \Ar(\cV)\bigr)~;
\]
the value of this functor on $\cA$ evaluates on a zero-pointed $n$-manifold $M_\ast$ as the composite arrow in $\cV$
\begin{equation}\label{PD-map}
\framebox{
$\qquad \displaystyle\int_{M_\ast} \cA_+ \xra{~\rm counit~} \cA(M_\ast) \xra{~\rm unit~} \int^{M_\ast}\cA^+ ~,\qquad$
}
\end{equation}
termed the \emph{Poincar\'e/Koszul duality map}.

The following question drives this work and the sequel~\cite{pkd}.
\begin{q}\label{main-question}
What conditions on $\cA$ guarantee that the Poincar\'e/Koszul duality map~(\ref{PD-map})
is an equivalence?
\end{q}

\begin{remark}
The definitions present in this work culminate as the duality map~(\ref{PD-map}) above, which is functorial in all arguments.
\end{remark}

\begin{remark}

We point out that the Poincar\'e/Koszul duality maps are utterly ambidextrous in the background symmetric monoidal $\infty$-category $\cV$ in the sense that these maps are equivalences if and only if they are when $\cV$ is replaced by $\cV^{\op}$.

\end{remark}

\begin{remark}[Scanning]\label{scanning}
The Poincar\'e/Koszul duality map in the case $\cV = (\Spaces,\times)$ is equivalent to the scanning map of~\cite{mcduff},~\cite{segalrational},~\cite{bodig}. In those works the map is defined one manifold at a time, in compact families, upon making contractible choices; this makes the establishment of continuous functoriality in the manifold a nuisance to verify. 
To sketch this identification with scanning maps, for simplicity, we fix a smooth framed $n$-manifold equipped with a complete Riemannian metric for which there is a uniform radius of injectivity $\epsilon>0$.
Again for simplicity, consider $A$ to be a (discrete) commutative group.
In this case, we can identify the defining colimit for factorization homology as a labeled configuration space:
\begin{equation}\label{dold.thom}
\int_{M_\ast} A ~\simeq ~ \underset{S\in \Ran(M_\ast)}\colim A^S \simeq \Bigl(\bigvee_{i\geq 0} {M_\ast}^{\times i}\underset{\Sigma_i}\wedge A^{\times i}\Bigr)_{/\sim}~{}~\Bigl( = ~\bigl\{(S\underset{\rm finite}\subset M~ , ~S\xra{l}A)\bigr\}~\Bigr)~,
\end{equation}
where the equivalence relation in the third term is determined by declaring
\[
[(x_1,a_1),\dots, (x_{i-2},a_{i-2}),(x,a),(x,b)] ~ \sim ~ [(x_1,a_1),\dots, (x_{i-2},a_{i-2}),(x,a+b)]~, 
\]
and the fourth term is just a convenient description of the underlying set of the third space.
Dold--Thom theory identifies the homotopy groups
\[
\pi_\ast\int_{M_\ast}A~\cong~\ov{\sH}_\ast(M_\ast;A)
\]
of this space as the reduced homology of $M_\ast$. (See \cite{bandklayder} for a proof of the Dold--Thom theorem in terms of factorization homology.)
Through the same theory, we know $\int_{(\RR^n)^+} A \simeq \sB^n A \simeq K(A,n)$ is an Eilenberg--MacLane space.
Because we are working in the Cartesian symmetric monoidal $\infty$-category $\Spaces$, and using that $M$ is framed, factorization cohomology
\begin{equation}\label{E.ML}
\int^{M_\ast} \sB^n A~ \simeq ~\Map^{\ast/}\bigl(M_\ast^\neg, K(A,n)\bigr)
\end{equation}
is weakly equivalent to the space of based maps from the negation of $M_\ast$.
Consequently, we identify the homotopy groups
\[
\pi_\ast \int^{M_\ast} \sB^n A~\cong~\ov{\sH}^{n-\ast}(M_\ast^\neg;A)
\]
as the shifted reduced cohomology groups of $M_\ast^\neg$.
Through the identifications~(\ref{dold.thom}) and~(\ref{E.ML}), the map~(\ref{PD-map}) is weakly equivalent to the assignment
\[
\bigl(M\supset S\xra{l}A\bigr) \mapsto \Bigl(M \ni x\mapsto \bigl(B_\epsilon(x) \cap S\xra{l_|} A\bigr) \in \int_{B_\epsilon(x)^\ast} A \simeq K(A,n) \Bigr)
\]
-- here $x\in B_\epsilon(x)\subset M$ is the $\epsilon$-ball about $x$.
This assignment is continuous, and is the \emph{scanning map} as mentioned.
Applying homotopy groups to this map results in the classical Poincar\'e duality isomorphism
\[
\ov{\sH}_\ast(M_\ast;A)~\cong~ \ov{\sH}^{n-\ast}(M_\ast^\neg;A)~.
\]

\end{remark}

\subsection{Koszul duality}\label{sec:koszul}
Evaluating the Poincar\'e/Koszul duality map~(\ref{PD-map}) on pointed Euclidean spaces provokes a meaningful examination: Koszul duality.
Here we geometrically define a procedure for assigning to an augmented $n$-disk algebra an augmented $n$-disk coalgebra (and vice-versa) -- this is the Bar-coBar adjunction.

\begin{definition}[Koszul duality]\label{def:koszul-duality}
Consider a symmetric monoidal $\infty$-category $\cV$.
Say a symmetric monoidal functor $\cA\colon \zdisk_n \longrightarrow \cV$ is a \emph{Koszul duality} if it has the following two properties.
\begin{itemize}
\item $\cA$ is initial among all such whose restriction to $\Disk_{n,+}$ is $\cA_+$.  This is to say, $\cA$ is initial in the $\infty$-category that is the fiber over $\cA_+$ of the restriction $\Fun^\ot\bigl(\zdisk_n,\cV\bigr) \xra{(-)_+} \Alg_n^{\sf aug}(\cV)$.  

\item $\cA$ is final among all such whose restriction to $\Disk_n^+$ is $\cA^+$.  This is to say, $\cA$ is final in the $\infty$-category that is the fiber over $\cA^+$ of the restriction $\Fun^\ot\bigl(\zdisk_n,\cV\bigr) \xra{(-)^+} \cAlg_n^{\sf aug}(\cV)$.  

\end{itemize}

\end{definition}

\begin{remark}\label{left-right-determine}
A key feature of a Koszul duality, $\cA$, is that it is determined by its restriction to either $\Disk_{n,+}$ or to $\Disk_n^+$.  
Lemma~\ref{equivalent-koszul} makes this explicit.

\end{remark}

\begin{remark}\label{justify-koszul}
The notion of a Koszul duality has been developed in other works, such as~\cite{gk},~\cite{getzlerjones}, and~\cite{dag10}.  
We will leave it to another work to explain the relationship between the notion presented here and that of~\cite{dag10}.  

\end{remark}

The diagram~(\ref{notation-+}) provokes the following definition, the notation for which is justified as Theorem~\ref{bar-consistent} to come.  
Recall Terminology~\ref{conditions} of \emph{sifted (co)complete}.
\begin{definition}[Bar-coBar]\label{def:bar-cobar}
Consider a symmetric monoidal $\infty$-category $\cV$ that is sifted cocomplete and cosifted complete.
Define the composite functors
\[
\bBar^n\colon \Fun(\Disk_{n,+},\cV) \xra{\int_-} \Fun\bigl(\zmfld_n,\cV\bigr) \longrightarrow  \Fun\bigl(\Disk_n^+,\cV\bigr)
\]
and
\[
\cBar^n\colon \Fun(\Disk_n^+,\cV) \xra{\int^{-}} \Fun\bigl(\zmfld_n,\cV\bigr) \longrightarrow  \Fun\bigl(\Disk_{n,+},\cV\bigr)~,
\]
in which the unlabeled arrows are restrictions.  
\end{definition}

Recall the notation for the functors displayed in~(\ref{notation-+}).  
The following result is a simple rephrasing of universal properties, premised on Proposition~\ref{bar-cobar-idem}.

\begin{lemma}\label{equivalent-koszul}
Let $\cV$ be a symmetric monoidal $\infty$-category that is $\ot$-sifted cocomplete and $\ot$-cosifted complete.  
Let $\cA\colon \zdisk_n \to \cV$ be a symmetric monoidal functor.  
The following statements are equivalent.
\begin{enumerate}
\item $\cA$ is a Koszul duality.

\item Both of the universal arrows
\[
\bBar^n \cA_+ \xra{~\simeq~} \cA^+
 ~{}~{}~{}~{}~\text{ and }~{}~{}~{}~{}~ 
\cA_+ \xra{~\simeq~} \cBar^n \cA^+
\]
are equivalences.

\item The universal triangle in $\Fun^\ot\bigl(\zdisk_n,\cV\bigr)$
\[
\xymatrix{
&
\cA  \ar[dr]^{\simeq}
&
\\
\int_{(-)} \cA_+ \ar[rr]^{\simeq} \ar[ur]^\simeq
&
&
\int^{(-)}\cA^+ 
}
\]
is of equivalences.

\end{enumerate}
\end{lemma}

\begin{proof}
The universal property of factorization (co)homology (Theorem~\ref{fact-functor}) determines the triangle of point~(3).  
By Definition~\ref{Mfld.+}, any full symmetric monoidal $\infty$-subcategory of $\ZDisk_n$ containing both $\RR^n_+$ and $(\RR^n)^+$ is entire.
Therefore, using that factorization homology and factorization cohomology define symmetric monoidal functors (Theorem~\ref{fact-functor}), this triangle is comprised of equivalences if and only if it is upon evaluating on $\RR^n_+$ and on $(\RR^n)^+$.  
Inspecting Definition~\ref{def:bar-cobar} of $\bBar^n$ and $\cBar^n$, this establishes the equivalence between~(2) and~(3).  

Now, because both of the symmetric monoidal functors $\Disk_{n,+}\hookrightarrow \ZDisk_n\hookleftarrow \Disk_n^+$ are fully-faithful, both factorization homology and factorization cohomology are fully-faithful.  
Because these functors are left and right adjoints, respectively, it follows that $\int_{(-)} \cA_+\in \Fun^\ot(\ZDisk_n,\cV)$ is initial among such symmetric monoidal functors whose restriction to $\Disk_{n,+}$ is $\cA_+$; likewise, $\int^{(-)}\cA^+$ is final among such symmetric monoidal functors whose restriction to $\Disk_n^+$ is $\cA^+$.  
From the Definition~\ref{def:koszul-duality} of a Koszul duality, that $\cA$ is a Koszul duality if and only if the above triangle in $\Fun^{\ot}(\ZDisk_n,\cV)$ is comprised of equivalences.
This establishes the equivalence between~(1) and~(3).  

\end{proof}

\begin{remark}\label{PD.on.basics}
Recall the central Question~\ref{main-question}, asking for conditions for when the universal transformation $\int_{(-)}\cA_+ \to \int^{(-)} \cA^+$ is an equivalence.  
Evaluating on $\RR^n_+$ and $(\RR^n)^+$, we see that a necessary condition is that $\cA$ is a Koszul duality.  
As we will see, in controlled situations, this is a sufficient condition as well.  

\end{remark}

\begin{prop}\label{bar-cobar-idem}
Let $\cV$ be a symmetric monoidal $\infty$-category whose underlying $\infty$-category admits sifted colimits and cosifted limits.  
The pair of functors
\begin{equation}\label{bar-cobar-crapy}
\bBar^n \colon \Fun(\Disk_{n,+},\cV) \leftrightarrows \Fun(\Disk_n^+,\cV) \colon \cBar^n
\end{equation}
is an adjunction.
If $\cV$ is $\ot$-sifted cocomplete, then $\bBar^n$ canonically factors through $\cAlg^{\sf aug}_{n}(\cV)$; dually, if $\cV$ is $\ot$-cosifted cocomplete, then $\cBar^n$ canonically factors through $\Alg^{\sf aug}_{n}$.  
If $\cV$ is $\ot$-sifted cocomplete and $\ot$-cosifted complete then the adjunction~(\ref{bar-cobar-crapy}) restricts as an adjunction
\begin{equation}\label{bar-cobar-adjunction}
\bBar^n \colon \Alg_n^{\sf aug}(\cV) \leftrightarrows \cAlg_n^{\sf aug}(\cV) \colon \cBar^n~.
\end{equation}

\end{prop}
\begin{proof}
The first statement follows immediately by composing the pair of adjunctions in~(\ref{notation-+}).
The final two statements, regarding symmetric monoidal extensions of the first two statements, follow thereafter from Theorem~\ref{fact-functor}.  

\end{proof}

Checking for a Koszul duality can be reduced to just a condition either on the algebra, or on the coalgebra.
\begin{lemma}\label{unit.reduce}
Let $\cV$ be a symmetric monoidal $\infty$-category that is $\ot$-sifted cocomplete and $\ot$-cosifted complete.
\begin{itemize}
\item An augmented $\Disk_n$-algebra $A$ in $\cV$ is a member of a Koszul duality if and only if the unit morphism 
\begin{equation}\label{53}
{\rm unit}\colon A \longrightarrow \cBar^n \circ \bBar^n(A)
\end{equation}
is an equivalence.

\item
An augmented $\Disk_n$-coalgebra $C$ in $\cV$ is a member of a Koszul duality if and only if the counit morphism
\begin{equation}\label{52}
\bBar^n \circ \cBar^n \circ \bBar^n(A) \longrightarrow \bBar^n(A)
\end{equation}
is an equivalence.  

\end{itemize}

\end{lemma}

\begin{proof}
The two assertions are equivalent, as implemented by replacing $\cV$ by $\cV^{\op}$.  
We are therefore reduced to proving the first assertion, concerning an augmented $\Disk_n$-algebra $A$ in $\cV$.
Through Lemma~\ref{equivalent-koszul}, $A$ is a member of a Koszul duality if and only if both the unit morphism~(\ref{53}) and the counit morphism~(\ref{52}) are equivalences.  
The functors $\bBar^n$ and $\cBar^n$ being adjoints to one another, there is a commutative triangle
\[
\xymatrix{
\bBar^n(A)  \ar[rr]^-=  \ar[dr]_-{\bBar^n({\rm unit}) }
&&
\bBar^n(A)  
\\
&
\bBar^n\circ \cBar^n \circ \bBar^n(A)  \ar[ur]_-{{\rm counit}(\bBar^n)}
&
.
}
\]
From the 2-out-of-3 property for equivalences in an $\infty$-category, the morphism~(\ref{52}) is an equivalence provided the morphism~(\ref{53}) is an equivalence.

\end{proof}

\subsection{The bar construction}

Let $\cV$ be a symmetric monoidal $\infty$-category that is $\ot$-sifted cocomplete.  
The main result in this section is Theorem~\ref{bar-consistent}, which justifes the notation 
\[
\bBar^n\colon \Alg_n^{\sf aug}(\cV) \longrightarrow \cV~,\qquad A\mapsto \bigl(\RR^n\mapsto \int_{(\RR^n)^+}A\bigr) ~,
\]
as an $n$-fold iteration of a bar construction.

Recall from~\S5.2.1 of \cite{HA} the bar construction ${\sf Bar}(A)\simeq \uno\underset{A}\otimes \uno$ of an augmented associative algebra $A\ra \uno$ in $\cV$.
There, it is explained that ${\sf Bar}(A)$ is equivalent to the geometric realization of a simplicial object ${\sf Bar}_\bullet(A)$ in $\cV$, which is a two-sided bar construction.
Pointwise explicitly, the object of $p$-simplices is canonically equivalent to $A^{\ot p}$, and through this identification the inner face maps can be identified as (a choice of) the associative multiplication map for $A$, the outer face maps can be identified as the augmentation of $A$, and the degeneracy maps can be identified as (a choice of) the unit of $A$.

\begin{remark}\label{naive.comult}
Let $A$ be a $1$-disk algebra in $\Mod_{\Bbbk}$, chain complxes over a field $\Bbbk$.  
There is a naive comultiplication
\[
\uno\underset{A}\ot\uno \simeq \uno\underset{A}\ot A\underset{A}\ot\uno\longrightarrow \uno\underset{A}\ot \uno\underset{A}\ot\uno
\]
given by the augmentation of $A$ in the middle term.  
It is a classical result that one can choose a model specific representation which admits a strict coalgebra refinement of this homotopy associative map. 
\end{remark}

Let $A$ be an $n$-disk algebra in $\cV$.  
Consider the continuous functor between topological categories:
\begin{equation}\label{7}
\Disk_1\times \Disk_{n-1}\longrightarrow \Disk_n~,\qquad (U,V) \mapsto U\times V~.  
\end{equation}
For each $U\in \Disk_1$, the restricted functor $\Disk_{n-1}\xra{U\times -} \Disk_n$ is canonically symmtric monoidal; likewise, for each $V\in \Disk_{n-1}$, the restricted functor $\Disk_1 \xra{-\times V} \Disk_n$ is canonically symmetric monoidal.  
This is to say that the functor~(\ref{7}) is symmetric bi-monoidal.
Therefore, the restriction of $A\colon \Disk_n \to \cV$ along~(\ref{7}) is adjoint to a symmetric monoidal functor $\Disk_1\to \Alg_{n-1}(\cV)$, that we will again denote as $A$.  
The \emph{$n$-fold} bar construction is inductively defined as the object in $\cV$
\[
{\sf Bar}^n(A)~:=~{\sf Bar}\bigl({\sf Bar}^{n-1}(A)\bigr)~. 
\]
(See~\S5.2.2 of~\cite{HA} for a thorough discussion of this iterated Bar construction.)
Through similar considerations as the case $n=1$ of Remark~\ref{naive.comult}, one can expect an $n$-disk coalgebra structure on ${\sf Bar}^n(A)$.  
The non-iterative nature of an $n$-disk (co)algebra puts tension against this expectation, particularly when considering the $\sO(n)$-module structure on the underlying objects of $n$-disk (co)algebras.
The coming results validate this expectation.

\begin{theorem}\label{bar-consistent}
Let $A$ be an augmented $n$-disk algebra in a symmetric monoidal $\oo$-category $\cV$.
Provided $\cV$ is $\ot$-sifted cocomplete, there is a canonical equivalence
\[
\int_{(\RR^n)^+}A~{}~ \simeq ~{}~{\sf Bar}^n (A)
\]
between the factorization homology of the $1$-point compactification of $\RR^n$ with coefficients in $A$, and the $n$-fold iteration of the bar construction applied to $A$.
\\
Likewise, let $C$ be an augmented $n$-disk coalgebra in $\cV$.
Provided $\cV$ is $\ot$-cosifted complete, there is a canonical equivalence
\[
\int^{\RR^n_+} C~{}~\simeq~{}~{\sf cBar}^n(C)~.
\]
\end{theorem}
\begin{proof}
The first statement implies the second by replacing $\cV$ by $\cV^{\op}$, so we only establish the first.
Theorem~\ref{reduced-to-cones} gives the canonical identification
\[
\int_{(\RR^n)^+}A~\simeq~ \int_{\DD^n}A~.
\] 
We proceed by induction on $n$.
Consider the base case $n=1$.
The conditions on $\cV$ give that factorization homology for smooth manifolds with boundary satisfies $\ot$-excision (Corollary~2.40 of~\cite{aft2}).
Applying this $\ot$-excision for manifolds with boundary from \cite{aft2} to the collar-gluing $\DD^1 \cong [-1,1)\underset{(-1,1)\times \{0\}}\bigcup (-1,1]$, we have an identification
\[
\int_{\DD^1}A~ \simeq ~\int_{[-1,1)}A\underset{\int_{\{0\}}A}\bigotimes\int_{(-1,1]}A~\simeq \uno \underset{A} \bigotimes \uno ~\simeq ~{\sf Bar}(A)~.
\]
This establishes the $n=1$ case.  

Now, the standard projection $\DD^n \xra{\sf pr} \DD^1$ onto the first coordinate is a weakly constructible bundle (see \cite{aft1}).
Consequently, the pushforward formula for factorization homology (Theorem~2.25 of~\cite{aft2}) gives a canonical identification between objects in $\cV$,
\[
\int_{\DD^n} A~\simeq ~ \int_{\DD^1} {\sf pr}_\ast A~,
\]
where ${\sf pr}_\ast A$ evaluates on $U\hookrightarrow \DD^1$ as $\int_{{\sf pr}^{-1} U} A$.
The $\ot$-excision formula applied to the collar-gluing $\DD^1 \cong [-1,1)\underset{(-1,1)\times \{0\}}\bigcup (-1,1]$, gives the canonical identification
\[
\int_{\DD^1} {\sf pr}_\ast A
~\simeq~ 
\int_{[-1,1)} {\sf pr}_\ast A\underset{\int_{\{0\}} {\sf pr}_\ast A}\bigotimes \int_{(-1,1]} {\sf pr}_\ast A~.
\]
We land at a canonical identification between objects of $\cV$,
\[
\int_{\DD^n} A~\simeq~ \uno \underset{\int_{\DD^{n-1}} A} \bigotimes \uno~\simeq~\uno\underset{{\sf Bar}^{n-1}A}\bigotimes \uno~\simeq {\sf Bar}^n A~,
\]
where the left equivalence is by inspection of the previous display, the middle equivalence is by induction on $n$, and the right equivalence is by definition of the iterated bar construction.  

\end{proof}

Theorem~\ref{bar-consistent} allows us to see the naive comultiplication above as exactly the fold map $(\RR^n)^+ \ra (\RR^n)^+\vee (\RR^n)^+$, the Pontryagin--Thom collapse map of an embedding $\RR^n \sqcup \RR^n \hookrightarrow \RR^n$.

\begin{cor}
For $A$ an augmented $n$-disk algebra in $\cV$, a symmetric monoidal $\oo$-category which is $\ot$-sifted cocomplete, the $n$-times iterated bar construction ${\sf Bar}^n (A)$ carries a natural $n$-disk coalgebra structure.
\end{cor}
\begin{proof}
Through Theorem~\ref{bar-consistent}, it is sufficient to exhibit an augmented $n$-disk coalgebra structure on $\int_{(\RR^n)^+}A$. 
Using the assumed $\ot$-cosifted complete property $\cV$, Theorem~\ref{fact-functor} applies for the effect that the factorization homology functor $\int_- A\colon \ZMfld_n\to \cV$ as a symmetric monoidal functor.
The desired $n$-disk coalgebra structure is the composite symmetric monoidal functor
\[
\Disk_n^+ \hookrightarrow \ZMfld_n\xra{~\int_- A~} \cV~,\qquad (\RR^n)^+\mapsto \int_{(\RR^n)^+}A\simeq \bBar^n(A)~.
\]
\end{proof}

\begin{remark}\label{operads}
For a general operad $\cO$ together with a left $\cO$-module $A$, such as an $\cO$-algebra, and a right $\cO$-module $M$, one can define an analogue of factorization homology\[\int_MA := M\underset{{\sf Env}(\cO)}\bigotimes A\] as the coend of $A$ and $M$ over the symmetric monoidal envelope of $\cO$. If $\cO$ is augmented, then one can construct a likewise analogue of the map
\begin{equation}\label{operadmap}
\int_M A \longrightarrow \int^{\sB M} \sB A\end{equation}
to the factorization cohomology (i.e., the end) of the left $\uno\circ_{\cO}\uno$-module $\sB A := |{\sf Bar}(\uno, \cO, A)|$ and the right $\uno\circ_{\cO}\uno$-module $\sB M := |{\sf Bar}_\bullet(M, \cO, \uno)|$. This map does not reflect a phenomenon of Poincar\'e duality, however. In the case $\cO = \cE_n$, Poincar\'e duality takes place in the identification of the Bar construction $\uno\circ_{\cE_n}\uno$ with a stable shift of $\cE_n$ and, thus, in the identification of the righthand side of (\ref{operadmap}) as factorization cohomology. In particular, an operadic approach would not obviously account for non-abelian Poincar\'e duality and the unstable Koszul self-duality of $n$-disk algebra provided by Proposition~\ref{semi.simplify}. However, should the map (\ref{operadmap}) be of interest, the same tools used here and in the sequel \cite{pkd} to address when the Poincar\'e/Koszul duality map is equivalence also apply to it. In short, one requires certain (co)connectivity bounds on the objects $A$ and $M$.
\end{remark}

\subsection{Koszul dualities in bicomplete Cartesian-sifted $\infty$-categories}
Here we aim toward an answer to Question~\ref{main-question} in the case that the symmetric monoidal $\infty$-category $\cV$ has certain (co)completeness and (co)continuity properties.
In this section, we simplify reduced factorization (co)homology, as well as characterize Koszul dualities, in such situations.

\begin{definition}[Definition~6.1.2.7 of~\cite{HTT}]\label{groupoid-object}
Let $\cX$ be an $\infty$-category.
A simplicial object $\cG\colon \bdelta^{\op} \to \cX$ is a \emph{groupoid object (in $\cX$)} if, for each finite non-empty linearly ordered set $L$, and each pair of subsets $S$, $T\subset L$ whose union $S\cup T = L$ is entire and whose intersection $S\cap T = \{l\}\subset L$ is a singleton, the diagram in $\cX$
\[
\xymatrix{
\cG(L) \ar[r]  \ar[d]
&
\cG(T) \ar[d]
\\
\cG(S)  \ar[r]
&
\cG(\{l\})
}
\]
is a pullback.
A groupoid object $\cG$ in $\cX$ is \emph{effective} if the canonical diagram in $\cX$
\[
\xymatrix{
\cG(\{0<1\}) \ar[r]  \ar[d]
&
\cG(\{1\}) \ar[d]
\\
\cG(\{0\})  \ar[r]
&
|\cG|
}
\]
is a pullback; here, $|\cG| :=\colim(\bDelta^{\op}\xra{\cG} \cX)\in \cX$ is the colimit.

\end{definition}

\begin{definition}\label{sifted-topos}
An $\infty$-category $\cS$ is \emph{bicomplete Cartesian-sifted} if the following conditions are satisfied.
\begin{itemize}
\item $\cS$ admits limits and colimits.

\item 
Sifted colimits in $\cS$ are universal: for each morphism $f\colon X\to Y$ in $\cS$, the base change functor
\[
f^\ast \colon \cS_{/Y} \longrightarrow \cS_{/X}~,\qquad (Z\to Y)\mapsto (X\underset{Y}\times Z \to X)~,
\]
preserves sifted colimits.  

\item 
Each groupoid object in $\cS$ is effective. 

\end{itemize}

\end{definition}

\begin{example}\label{ex.cart.pres}
Here are some examples of bicomplete Cartesian-sifted $\oo$-categories. 
\begin{itemize}

\item A presentable stable $\infty$-category $\cS$ is a bicomplete Cartesian-sifted $\oo$-category. 
In particular, for $\Bbbk$ a ring spectrum, $\Mod_{\Bbbk}({\sf Spectra})$ is a bicomplete Cartesian-sifted $\oo$-category.

\item
The opposite $\cS^{\op}$ of a stable presentable $\infty$-category is a bicomplete Cartesian-sifted $\infty$-category. 

\item An $\infty$-topos $\cE$ is a bicomplete Cartesian-sifted $\oo$-category. 
In particular, for any small $\infty$-category $\cC$, the $\infty$-category $\Psh(\cC)$ is bicomplete Cartesian-sifted.  
As the case $\cC \simeq \ast$ is final, we see that the $\infty$-category $\Spaces$ of spaces is bicomplete Cartesian-sifted.  

\item
Let $X\in \cS$ an object in a bicomplete Cartesian-sifted $\infty$-category.
The $\infty$-overcategory $\cS_{/X}$ is a bicomplete Cartesian-sifted $\infty$-category; likewise, the $\infty$-undercategory $\cS^{X/}$ is a bicomplete Cartesian-sifted $\infty$-category.

\end{itemize}

\end{example}

\begin{notation}\label{def.tensor}
Each bicomplete $\infty$-category $\cS$ has a final object, $\ast\in \cS$.
The $\infty$-undercategory $\cS^{\ast/}$ is canonically tensored and cotensored over pointed spaces:
\[
\otimes \colon \Spaces^{\ast/} \times \cS^{\ast/} \longrightarrow \cS^{\ast/}~,
\]
\[
(\ast \to Z , X)\mapsto  Z\otimes X := \colim\bigl(Z \xra{!} \ast \xra{\{X\}} \cS^{\ast/}\bigr)  
\simeq
\ast \underset{X} \amalg \colim\bigl(Z\xra{!} \ast \xra{\{X\}} \cS\bigr) ~\in \cS^{\ast/}~,
\]
and
\[
\Map_\ast(-,-)\colon (\Spaces^{\ast/})^{\op} \times \cS^{\ast/} \longrightarrow \cS^{\ast/}~,
\]
\[
(\ast \to Z , X)  \mapsto   \Map_\ast(Z,X) := X^{Z} :=  \limit\bigl(Z \xra{!} \ast \xra{\{X\}} \cS^{\ast/}\bigr) \simeq \ast \underset{X} \times \limit\bigl(Z \xra{!} \ast \xra{\{X\}} \cS\bigr) ~.
\]

\end{notation}

\begin{convention}\label{semi-sym}
We adopt the convention to regard a bicomplete Cartesian-sifted $\oo$-category $\cS$ as a symmetric monoidal $\infty$-category whose underlying $\infty$-category is $\cS$ and whose symmetric monoidal structure is the Cartesian one.

\end{convention}

\begin{lemma}\label{alg.cart.pres}
Let $\cS$ be a bicomplete Cartesian-sifted $\infty$-category.
Let $\cO$ be a unital $\infty$-operad.
The $\infty$-category $\Alg_\cO(\cS)$ of $\cO$-algebras in the Cartesian symmetric monoidal $\infty$-category $\cS$ is bicomplete Cartesian-sifted.

\end{lemma}

\begin{proof}
This follows after the following results from~\cite{HA}.
Corollary~3.2.3.5 thereby grants that $\Alg_\cO(\cS)$ is bicomplete. 
Corollary~3.2.2.5 thereby grants that the forgetful functor $\Alg_\cO(\cS) \to \cS$ preserves limits.
Corollary~3.2.3.2 thereby grants that the forgetful functor $\Alg_{\cO}(\cS) \to \cS$ preserves sifted colimits.

\end{proof}

\begin{remark}
Few symmetric monoidal $\infty$-categories arise as an instance of Convention~\ref{semi-sym}.
We view bicomplete Cartesian-sifted $\oo$-categories as degenerate examples of symmetric monoidal $\infty$-categories because comultiplication is unique and commutative.   
For instance, for $\Bbbk$ a ring spectrum, tensor product over $\Bbbk$ defines a symmetric monoidal structure on the $\infty$-category $\Mod_{\Bbbk}$ of chain complexes over $\Bbbk$. 
This tensor product does \emph{not}, in general, distribute over totalizations.  
This case of considerable interest is the subject of the sequal~\cite{pkd}.  

\end{remark}

\begin{observation}\label{cosifted-complete}
Let $\cS$ be a bicomplete Cartesian-sifted $\oo$-category.
Then, as a symmetric monoidal $\infty$-category, it is $\ot$-sifted cocomplete and $\ot$-cosifted complete.  
Theorem~\ref{fact-functor} ensures the existence of the two adjunctions
\[
\xymatrix{
\int_-   \colon \Alg_{n}^{\sf aug}(\cS)  \ar@(-,u)[r]
&
\Fun^\ot\bigl(\ZMfld_n, \cS\bigr)  \ar[r]  \ar[l]   
&
\cAlg_{n}^{\sf aug}(\cS) \ar@(-,d)[l] \colon \int^-~.
}
\]

\end{observation}

The existence of a zero-object in the $\infty$-category $\ZMfld_n$ determines a canonical lift of the Yoneda functor
\[
\xymatrix{
&&
\Fun^{\ast/}(\ZMfld_n^{\op} , \Spaces^{\ast/})  \ar[d]
\\
\ZMfld_n    \ar@{-->}[urr]     \ar[rr]^-{\rm Yoneda}
&&
\Psh(\ZMfld_n)   .
}
\]
\begin{definition}\label{frame-bundle}
The \emph{frame bundle} functor is the restricted Yoneda functor
\[
{\sf Fr}_-\colon \ZMfld_n \xra{ \rm Yoneda } \Fun^{\ast/}(\ZMfld_n^{\op} , \Spaces^{\ast/})
\xra{ \rm restriction }
\Fun\bigl(\BO(n)^{\op} , \Spaces^{\ast/}\bigr)
~=:~
\Mod_{\sO(n)}(\Spaces^{\ast/})~.
\]

\end{definition}

\begin{remark}\label{rem.lastlabel51} 
The unstraightening construction identifies the $\infty$-category of $\sO(n)$-modules in pointed spaces
\begin{equation}\label{51}
\Mod_{\sO(n)}(\Spaces^{\ast/})~\simeq~\Spaces^{\BO(n)/}_{/\BO(n)}
\end{equation}
with the $\infty$-category of retractive spaces over $\BO(n)$.
We describe the frame bundle of a zero-pointed $n$-manifold through this identification~(\ref{51}).
Let $M_\ast$ be a zero-pointed $n$-manifold.
Choose a smooth $n$-manifold $\ov{M}$ with compact booundary $\partial \ov{M}$ together with an identification $\ast \underset{\partial \ov{M}} \amalg \ov{M} \cong M_\ast$ between pointed extensions of the interior $M$.  
Through the identification~(\ref{51}), the frame bundle is the pushout
\[
\xymatrix{
\partial M  \ar[rr]  \ar[d]_-{(\tau_M)_{|\partial M}}  
&&
M   \ar[d]
\\
\BO(n)  \ar[rr]
&&
{\sf Fr}_{M_\ast}
}
\]
in spaces over $\BO(n)$, as it is equipped with the section offered by the bottom horizontal map.
In particular, a framing of a neighborhood of $\partial M \subset M$ determines an identification
\[
{\sf Fr}_{M_\ast}~\simeq~ \BO(n) \bigvee M_\ast~.
\]
As special cases, we identify 
\[
{\sf Fr}_{(\RR^n)^+}~\simeq ~\BO(n) \bigvee (\RR^n)^+~,
\qquad
\text{ as well as }
\qquad
{\sf Fr}_{M_+}~ \simeq ~\bigl(\BO(n) \amalg M \xra{~\tau_M~} \BO(n)\bigr)~.
\]

\end{remark}

We reference the following notion in Proposition~\ref{semi.simplify}; compare with the recognition principle of \cite{may}.
\begin{definition}\label{def.gplike}
Let $\cV$ be a symmetric monoidal $\infty$-category.  
An augmented $n$-disk algebra $A\colon \Disk_n \to \cV$ is \emph{grouplike} if there is a framed open embedding $e\colon \RR^n\sqcup \RR^n \hookrightarrow \RR^n$ for which the two squares in the diagram in $\cS$
\[
\xymatrix{
A(\RR^n)  \ar[d]_-{\sf aug}
&&
A(\RR^n)\otimes A(\RR^n) \ar[rr]^-{{\sf aug}\ot {\sf id}}  \ar[d]^{A(e)}  \ar[ll]_-{{\sf id}\ot {\sf aug}}
&&
A(\RR^n)  \ar[d]^-{\sf aug}
\\
\uno 
&&
A(\RR^n)  \ar[rr]_--{\sf aug}    \ar[ll]^--{\sf aug}
&&
\uno
}
\]
are pullback, where ${\sf aug}$ is the augmentation morphism.

\end{definition}

\begin{prop}\label{semi.simplify}
The following statements are true concerning a bicomplete Cartesian-sifted $\oo$-category $\cS$.
\begin{enumerate}
\item 
The canonical functors between $\infty$-categories
\[
\Alg_{\Disk_n}^{\sf aug}(\cS) \xra{~\simeq~}\Alg_{\Disk_n}(\cS) \qquad \text{ and }\qquad  \cAlg_{\Disk_n}^{\sf aug}(\cS)  \xra{~\simeq~} \Mod_{\sO(n)}(\cS^{\ast/})
\]
are equivalences.

\item
Let $R\colon \BO(n) \to \cS^{\ast/}$ be a $\sO(n)$-module, which we regard as an augmented $\Disk_n$-coalgebra in $\cS$ through the equivalence above.
Let $M_\ast$ be a zero-pointed $n$-manifold.
There is a canonical identification
\[
\int^{M_\ast} R  \xra{~\simeq~}  \Map_\ast^{\sO(n)}\bigl({\sf Fr}_{M_\ast^\neg}, R\bigr) ~,
\]
to the cotensor under $\BO(n)$.  

\item 
Through the above equivalences, the Bar-coBar adjunction becomes the adjunction 
\[
\bBar^n  \colon \Alg_{\Disk_n}(\cS) \rightleftarrows \Mod_{\sO(n)}(\cS^{\ast/})  \colon \Omega^n~,
\]
in which each value of the left adjoint is an $n$-fold Bar construction, and the value of the right adjoint on an $\sO(n)$-module $R$ in $\cS^{\ast/}$ restricts to $\BO(n)$ as $n$-fold loops:
\[
\Omega^n(R) \colon \BO(n) \ni V \mapsto \Map_\ast\bigl(V^+ , R(V)\bigr)=: \Omega^V R(V)~.
\]

\item 
Let $A$ be an augmented $\Disk_n$-algebra in $\cS$.
If $n=0$, then $A$ belongs to a Koszul duality.
If $n>0$ is positive, 
then $A$ belongs to a Koszul duality if and only if $A$ is grouplike.

\end{enumerate}

\end{prop}

\begin{proof}

We prove statement~(1).  
Note that the symmetric monoidal structure of $\cS$ is Cartesian.
It follows that the symmetric monoidal unit is final, which proves the first part of statement~(1).
It also follows from Proposition~2.4.3.9 of~\cite{HA}, applied to the $\infty$-operad $\cO^{\ot} = \Fin_\ast$, that the restriction functor $\cAlg^{\sf aug}_{\Disk_n}(\cS) \xla{~\simeq~}
\cAlg_{\Disk_n}(\cS^{\ast/})
\xra{~\simeq~} \Fun\bigl(\BO(n) , \cS^{\ast/}\bigr)$ is an equivalence.   
This concludes the proof of statement~(1).

We now prove statement~(2).
Consider the solid commutative diagram of $\infty$-categories:
\[
\xymatrix{
\Fun^{\ot}\bigl( \ZMfld_n , \cS\bigr)  \ar[rr]   \ar[d]
&&
\cAlg^{\sf aug}_{\Disk_n}(\cS)    \ar[d]    \ar@(u,u)@{-->}[ll]^-{\int}    \ar@{-->}[dll]^-{\int}
\\
\Fun^{\ast/}\bigl( \ZMfld_n , \cS^{\ast/}\bigr)   \ar[rr]
&&
\Fun\bigl(\BO(n) , \cS^{\ast/}\bigr)  .    \ar@(d,d)@{-->}[ll]_-{\sf RKan}
}
\]
The dashed arrows are right given by right Kan extensions.
Using that $\cS^{\ast/}$ is $\ot$-cosifted complete, Theorem~\ref{fact-functor} gives that these dashed arrows exists in such a way that the upper leftward triangle commutes.  
Defined as right Kan extensions, the lower leftward triangle therefore commutes as well.  
Finally, the value of the bottom dashed functor on a zero-pointed functor $R\colon \BO(n) \to \cS^{\ast/}$ evaluates on a zero-pointed manifold $M_\ast$ as the end:
\[
{\sf RKan}(R)\colon M_\ast ~  \mapsto ~ \limit\bigl( \BO(n)^{M_\ast/} \to \BO(n) \xra{R} \cS^{\ast/}\bigr)
\]
\[
\simeq~
\limit\bigl( \BO(n)_{/M_\ast^\neg} \to \BO(n) \xra{R} \cS^{\ast/}\bigr)
~\simeq~\Map^{\sO(n)}\bigl( {\sf Fr}_{M_\ast^{\neg}} , R \bigr)~\in ~\cS^{\ast/}.
\]
This establishes statement~(2).

Statement~(3) follows upon establishing, for each functor $R\colon \BO(n) \to \cS^{\ast/}$, and for each vector space $V\in \BO(n)$, a canonical sequence of equivalences in $\cS^{\ast/}$:
\begin{equation}\label{50}
\int^{V_+} R \xra{~ \simeq~} \Map_\ast^{\sO(n)}\bigl({\sf Fr}_{V^+} , R\bigr)\xra{ ~\simeq~} \Map_\ast\bigl(V^+ , R(V)\bigr)~=:~ \Omega^V R(V)~.
\end{equation}
The leftmost identification is statement~(2).  
Now, for each vector space $V\in \BO(n)$, the canonical map between $\sO(n)$-modules in based spaces,
\[
{\sf Hom}_{\BO(n)}(\RR^n,V)_+\wedge V^+ \longrightarrow {\sf Fr}_{V^+}~,
\]
is an equivalence.  
In particular, for each vector space $V\in \BO(n)$, the $\sO(n)$-module ${\sf Fr}_{V^+}$ in pointed spaces is free on the pointed space $V^+$.
This establishes the rightmost equivalence in~(\ref{50}), via the free-forgetful adjunction for $\sO(n)$-modules.

We now turn to proving statement~(4).
In the case $n=0$, this statement is trivially true; so we assume $n>0$.
We first establish the implication that, if $A$ is a Koszul duality, then $A$ is grouplike.  
So suppose $A$ is a member of a Koszul duality.
By Lemma~\ref{equivalent-koszul}, the unit morphism $A \to \cBar^n\circ \bBar^n(A)$ is an equivalence.  
In light of statement~(3), it is enough to show that each value of the functor $\Omega^n \colon \Mod_{\sO(n)}(\cS^{\ast/}) \to \Alg_{\Disk_n}(\cS)$ is grouplike.  
Statement~(3) identifies this unit morphism as
\begin{equation}\label{unit-loops}
A \to \Omega^n \bBar^n A~.
\end{equation}
There is a functor $\BO(n)\otimes \Disk_n^{\sf fr} \to \Disk_n$ from the tensor among symmetric monoidal $\infty$-categories with the space $\BO(n)$.
There results the functor
\[
\Alg_{\Disk_n}(\cS)\xra{~\simeq~} \Map_{\BO(n)}\bigl(\BO(n),\Alg_{\Disk_n^{\sf fr}}(\cS)\bigr)
\]
which is an equivalence.  
And so, this unit morphism~(\ref{unit-loops}) is an equivalence if and only if its restriction $A_{|\Disk_n^{\sf fr}} \to \Omega^n \bBar^n A_{|\Disk_n^{\sf fr}}$ is an equivalence for each point $\ast \xra{b} \BO(n)$.
In this way we are reduced to the framed case, as in~\cite{fact}.

Suppose this unit morphism~(\ref{unit-loops}) is an equivalence.
Consider the commutative diagram of based spaces
\[
\xymatrix{
(-1,0)^+    \ar[r]
&
\bigl((-1,0) \cup   (0,1)\bigr)^+ 
&
(0,1)^+  \ar[l]
\\
+  \ar[u]  \ar[r]
&
\RR^+   \ar[u]
&
+  \ar[u]  \ar[l]
}
\]
in which the middle vertical map is the evident collapse-map, and the other maps are the evident inclusions.  
Each square in this diagram is a pushout square in the $\infty$-category $\Spaces^{\ast/}$ of based spaces -- this can be seen, for instance, by applying the fundamental group functor to this diagram, and using that each term is a $1$-type.
Now, for any object $Z\in \cS^{\ast/}$, the cotensor functor $\Map_\ast(-,Z)\colon \Spaces^{\ast/} \to \cS^{\ast/}$ carries the (opposites of) colimit diagrams to limit diagrams.
In particular, applying this cotensor functor to the above pushout diagrams among based spaces gives the pullback diagrams
\[
\xymatrix{
\Omega Z  \ar[d]
&
\Omega Z \times \Omega Z \ar[r]  \ar[d] \ar[l]
&
\Omega Z  \ar[d]
\\
\ast  
&
\Omega Z  \ar[r]  \ar[l]
&
\ast
}
\]
in $\cS$.
Applying this to the case that $Z=\Omega^{n-1} \bBar^n (A)\in \cS^{\ast/}$, we see that $\Omega^n \bBar^n(A)$ is grouplike.  
Under our supposition that the unit morphism~(\ref{unit-loops}) is an equivalence, we conclude that $A$ is grouplike, as desired.

We now show that $A$ being grouplike implies $A$ is a member of a Koszul duality.  
So suppose $A$ is grouplike.
Through Lemma~\ref{unit.reduce}, we need only prove that the unit morphism~(\ref{unit-loops}) is an equivalence.  
Because the forgetful functor $\Alg^{\sf aug}_{\Disk_n}(\cS) \to \cS$ is conservative, such an equivalence can be detected on underlying objects of $\cS$.
We show this by induction on $n$.  
Suppose $n=1$.
We show the simplicial object $\bBar_\bullet(A)\colon \bDelta^{\op} \to \cS$ is a groupoid object in $\cS$.
By definition of this simplicial object, for each $0<i<p$, the canonical diagram in $\cS$,
\[
\xymatrix{
\bBar_{[p]}(A)  \ar[r] \ar[d]
&
\bBar_{\{0<\dots<i\}}(A)  \ar[d]
\\
\bBar_{\{i<\dots<p\}}(A)  \ar[r]
&
\bBar_{\{i\}}(A)\simeq \ast   ,
}
\]
is a pullback.
The assumption that $A$ is grouplike precisely implies that each square in the canonical diagram in $\cS$,
\[
\xymatrix{
\bBar_{\{0<1\}}(A)\simeq A \ar[d]
&
\bBar_{\{0<1<2\}}(A)\simeq A\times A \ar[r]  \ar[d] \ar[l]
&
\bBar_{\{1<2\}}(A)\simeq A  \ar[d]
\\
\ast  
&
\bBar_{\{0<2\}}(A)\simeq A  \ar[r]  \ar[l]
&
\ast   ,
}
\]
is a pullback.  
We conclude from these last two sentences that $\bBar_\bullet(A)$ is a groupoid object in $\cS$.
Because groupoids are effective in $\cS$, the canonical morphism in $\cS$ to the pullback
\[
A\simeq \bBar_{[1]}(A)\longrightarrow \bBar_{\{0\}}(A) \underset{|\bBar_\bullet(A)|}\times \bBar_{\{1\}}(A)
~\simeq ~
\ast \underset{\bBar(A)}\times \ast
~\simeq~ 
\Omega \bBar(A)
\]
is an equivalence.
This canonical morphism agrees with the unit morphism~(\ref{unit-loops}) on underlying objects of $\cS$.
In this way, we conclude that the morphism~(\ref{unit-loops}) is an equivalence, as desired.

Now suppose $n>1$.
Consider the symmetric monoidal restriction 
\begin{equation}\label{pre.add}
A_|\colon  \Disk_1^{\sf fr} \times \Disk_{n-1}^{\sf fr} \longrightarrow \Disk_n^{\sf fr} \xra{~A~} \cS
\end{equation}
along the symmetric monoidal functor given by taking (ordered) producs of smooth framed manifolds.  
This symmetric monoidal functor is adjoint to a symmetric monoidal functor $A^\dagger\colon \Disk_1^{\sf fr}\to \Alg_{\Disk_{n-1}^{\sf fr}}(\cS)$.
Note that the restriction of the symmetric monoidal functor~(\ref{pre.add}) to each factor is surjective on mapping spaces.
Because the $\Disk_n^{\sf fr}$-algebra $A$ in $\cS$ is grouplike, it follows that $A^\dagger$ is a $\Disk_1^{\sf fr}$-algebra in $\Alg_{\Disk_{n-1}^{\sf fr}}^{\sf gp.like}(\cS)$, grouplike $\Disk_{n-1}^{\sf fr}$-algebras in $\cS$.
Furthermore, because the forgetful functor $\Alg_{\Disk_{n-1}^{\sf fr}}(\cS) \to \cS$ preserves pullbacks, it also follows that this $\Disk_1^{\sf fr}$-algebra $A^\dagger$ in $\Alg_{\Disk_{n-1}^{\sf fr}}^{\sf gp.like}(\cS)$ is itself grouplike.
Lemma~\ref{alg.cart.pres} gives that $\Alg_{\Disk_{n-1}^{\sf fr}}(\cS)$ is bicomplete Cartesian-sifted.
Therefore, the $n=1$ case established above, and the inductive hypothesis on $n$, gives that the two unit morphisms
\begin{equation}\label{57}
A^\dagger \xra{~\simeq~} \Omega \bBar \bigl(A_{|\Disk_{n-1}^{\sf fr}}\bigr) \xra{~\simeq~} \Omega \bigl(\Omega^{n-1}\bBar^{n-1} (\bBar A_{|\Disk_{n-1}^{\sf fr}})\bigr)
\end{equation}
are each equivalences between $\Disk_1^{\sf fr}$-algebras in $\Alg_{\Disk_{n-1}^{\sf fr}}(\cS)$.  
Via the commutativity of the diagram of $\infty$-categories
\[
\xymatrix{
\Alg_{\Disk_n^{\sf fr}}(\cS)   \ar[rr]  \ar[dr]
&&
\Alg_{\Disk_1^{\fr}}\bigl(\Alg_{\Disk_n^{\sf fr}}(\cS) \bigr)  \ar[dl]
\\
&
\cS
&
,
}
\]
this equivalence~(\ref{57}) forgets to the desired equivalence in $\cS$:
\[
A \xra{~\simeq~} \Omega^n \bBar^n(A)~.
\]

\end{proof}

The next result is the specalization of Proposition~\ref{semi.simplify} to the case that $\cS$ is stable.  
\begin{cor}\label{stable-koszul}
Let $\cS$ be a stable presentable $\infty$-category.
The following statements are true.
\begin{enumerate}
\item 
Each of the projections to underlying $\sO(n)$-modules, 
\[
\Alg_n^{\sf aug}(\cS) \xra{\simeq}\Mod_{\sO(n)}(\cS) \xla{\simeq} \cAlg_n^{\sf aug}(\cS)~, 
\]
is an equivalence between $\infty$-categories from augmented $n$-disk (co)algebras in $\cS$.

\item
Let $E$ and $F$ be $\sO(n)$-modules in $\cS$.
Through the identifications of~(1) above, consider the unique extension of $E$ as an augmented $\Disk_n$-algebra $A_E$ in $\cS$, and the unique extension of $F$ an augmented $\Disk_n$-coalgebra $C^F$ in $\cS$.  
Each of the canonical morphisms in $\cS$,
\[
{\sf Fr}_{M_\ast} \underset{\sO(n)}\bigotimes E \xra{~\simeq~}\int_{M_\ast} A_E
\]
and
\[
\int^{M_\ast} C^F  \xra{~\simeq~}  \Map^{\sO(n)}\bigl({\sf Fr}_{M_\ast^\neg}, F\bigr) ~,
\]
is an equivalence.

\item 
Through the above equivalences, the Bar-coBar adjunction becomes the adjunction 
\[
(\RR^n)^+ \ot (-) \colon \Mod_{\sO(n)}(\cS) \rightleftarrows \Mod_{\sO(n)}(\cS)\colon (-)^{(\RR^n)^+}~,
\]
the left adjoint with the diagonal $\sO(n)$-module structure, and the right adjoint with the conjugation $\sO(n)$-module structure.
Each of these adjoint functors is, in fact, an equivalence between $\infty$-categories.

\item 
Every augmented $\Disk_n$-algebra in $\cS$ belongs to a Koszul duality.  

\item
Every augmented $\Disk_n$-coalgebra in $\cS$ belongs to a Koszul duality.

\end{enumerate}

\end{cor}

\subsection{Interval duality}
In this subsection we examine the Poincar\'e/Koszul duality map~(\ref{PD-map}) in the special case of a closed interval.  
Following through with Remark~\ref{PD.on.basics}, examining the values of the Poincar\'e/Koszul duality maps on basics gives rise to Koszul duality, definitionally.

In the sense of~\S5 of~\cite{aft1}, consider the category of basics $\sD_1^{\partial, \fr}$ for whose manifolds are oriented $1$-manifolds with boundary.  
In~\S2.6 of~\cite{aft2} we proved that the symmetric monoidal $\infty$-category $\Disk_1^{\partial, \fr}$ corepresents the data of an associative algebra $A$ together with a unital right $A$-module $P$ and a unital left $A$-module $Q$.
Therefore, the symmetric monoidal $\infty$-category $\Disk_{1,+}^{\partial, \sf fr}$ corepresents the likewise augmented data:
\begin{equation}\label{10}
\Alg_{\Disk_{1,+}^{\partial, \sf fr}} ( \cV )~\simeq~\Alg^{\sf aug}_{{\sf Assoc^{R, L}}} (\cV)~=~\Bigl\{ (P,A,Q) {\rm ~in~}\cV_{/\uno}  \Bigr\}~.
\end{equation}
Likewise, $\Disk_1^{\partial, \sf fr, +}$ corepresents the data of a coaugmented coassociative algebra $C$ together with a counital and coaugmented right $C$-comodule $R$ and a counital and coaugmented left $C$-comodule $S$:
\begin{equation}\label{11}
\Alg_{\Disk_1^{\partial, \sf fr, +}} ( \cV )~\simeq~
\cAlg^{\sf aug}_{{\sf Assoc^{R, L}}} (\cV)~=~
\Bigl\{ (R,C,S) {\rm ~in~}\cV^{\uno/}  \Bigr\}~.
\end{equation}

\begin{notation}\label{simple.RL.alg}
Through the identification~(\ref{10}), we notate a symmetric monoidal functor 
\[
(P,A,Q)\colon \Disk_{1,+}^{\partial, \sf fr} \longrightarrow   \cV~.
\]
Likewise, through the identification~(\ref{11}), we notate a symmetric monoidal functor 
\[
(R,C,S)\colon \Disk_1^{\partial, \sf fr, +}\longrightarrow   \cV~.
\]

\end{notation}

Here is an analogue to Definition~\ref{def:koszul-duality}.
\begin{definition}\label{def.interval.koszul}
Let $\cV$ be a symmetric monoidal $\infty$-category.
An \emph{interval Koszul duality} is a symmetric monoidal functor
\[
\cA\colon \ZDisk_1^{\partial, \fr} \longrightarrow \cV
\]
with the following properties.
Use the notation $(P,A,Q)$ for the restriction $\cA_{|\disk_{1,+}^{\partial, \sf fr}}$, and $(R,C,S)$ for the restriction $\cA_{|\disk_1^{\partial, {\sf fr},+}}$.
\begin{itemize}
\item 
$\cA$ is initial among all such whose restriction to $\Disk_{1,+}^{\partial, \sf fr}$ is $(P,A,Q)$.  
This is to say, $\cA$ is initial in the $\infty$-category that is the fiber over the restriction $(P,A,Q)$ of the restriction $\Fun^\ot\bigl(\ZDisk_1^{\partial, \fr},\cV\bigr) \xra{(-)_+}   \Alg^{\sf aug}_{{\sf Assoc^{R, L}}} (\cV)$.  

\item $\cA$ is final among all such whose restriction to $\Disk_1^{\partial, \sf fr, +}$ is $(R,C,S)$.  
This is to say, $\cA$ is final in the $\infty$-category that is the fiber over $(R,C,S)$ of the restriction $\Fun^\ot\bigl(\zdisk_n,\cV\bigr) \xra{(-)^+} \cAlg^{\sf aug}_{{\sf Assoc^{R, L}}} (\cV)$.  

\end{itemize}

\end{definition}

There is this direct result, which is analogous to Lemma~\ref{equivalent-koszul}.
\begin{lemma}\label{interval-koszul}
Let $\cV$ be a symmetric monoidal $\infty$-category that is $\ot$-sifted cocomplete and $\ot$-cosifted complete.
Let 
\[
\cA\colon \zdisk_1^{\partial, \fr} \longrightarrow \cV
\]
be a symmetric monoidal functor.  
Use the notation $(P,A,Q)$ for the restriction $\cA_{|\disk_{1,+}^{\partial, \sf fr}}$, and $(R,C,S)$ for the restriction $\cA_{|\disk_1^{\partial, {\sf fr},+}}$. 
Then $\cA$ is an \emph{interval} Koszul duality if and only if the following canonical comparison maps are equivalences in $\cV$:
\begin{eqnarray}
\nonumber
A \xra{~\simeq~} \cBar(\uno , C , \uno)  =: \uno \overset{ C } \bigotimes \uno
&
\text{ and }
&
\uno \underset{A}\bigotimes \uno  := \bBar(\uno , A , \uno) \xra{~\simeq~} C~,
\\
\nonumber
P \xra{~\simeq~} \cBar(R,C,\uno)  =: R \overset{ C } \bigotimes \uno
&
\text{ and }
&
P \underset{A}\bigotimes \uno :=  \bBar(P, A , \uno) \xra{~\simeq~} R~,
\\
\nonumber 
Q \xra{~\simeq~}  \cBar(\uno , C , S )  =: \uno \overset{ C } \bigotimes S
&
\text{ and }
&
\uno \underset{A}\bigotimes Q :=  \bBar(\uno , A , Q )  \xra{~\simeq~} S~.
\end{eqnarray}

\end{lemma}

\begin{lemma}\label{interval-PD}
Let $\cS$ be an $\infty$-topos.
Let $\cA\colon \zdisk_1^{\partial, \sf fr} \longrightarrow \cS$ be a symmeric monoidal functor.
Use the notation $(P,A,Q)$ for the restriction $\cA_{|\disk_{1,+}^{\partial, \sf fr}}$, and $(R,C,S)$ for the restriction $\cA_{|\disk_1^{\partial, {\sf fr},+}}$. 
Suppose $\cA$ is an interval Koszul duality.
Then $A$ is grouplike, and the Poincar\'e/Koszul duality map~(\ref{PD-map})
\[
\int_{[-1,1]_+} (P,A,Q)~{}~\xra{~\simeq~}~{}~ \int^{[-1,1]^+} (R,C,S)
\]
is an equivalence.  

\end{lemma}

\begin{proof}
Through $\ot$-excision, this Poincar\'e/Koszul duality morphism is identified as the morphism 
\[
\bBar(P,A,Q)  \longrightarrow     \cBar(R , C , S)
\]
from a Bar-construction to a coBar-construction.  
We must, then, verify that the canonical morphism in $\cS$,
\[
\bBar(P,A,Q) 
\longrightarrow 
\cBar(R , C , S)
\simeq
R\underset{C}\times S~,
\]
from the two-sided bar construction to the pullback, is an equivalence.  
The data $(P,A,Q)$ determines the evident diagram $\cS$
\begin{equation}\label{topos-square}
\xymatrix{
\bBar(P,A,Q)   \ar[d]  \ar[r]
&
\bBar(\ast,A,Q)  \ar[d]
\\
\bBar(P,A,\ast) \ar[r]
&
\bBar(\ast,A,\ast)~.
}
\end{equation}
Inspecting the expressions displayed in Lemma~\ref{interval-koszul}, because $\cA$ is assumed to be an interval Koszul duality, the problem is to verify that this square~(\ref{topos-square}) is a pullback.

The diagram~(\ref{topos-square}) is the geometric realization of the simplicial square-diagram in $\cS$ whose value on $[p]$ is the square of projections
\begin{equation}\label{topos-p-square}
\xymatrix{
P\times A^{\times p} \times Q  \ar[r]   \ar[d] 
&
A^{\times p} \times Q  \ar[d]
\\
P\times A^{\times p} \ar[r]
&
A^{\times p}~,
}
\end{equation}
which is certainly pullback.  
Therefore, the square~(\ref{topos-square}) is a pullback provided, for each $[p]\in \bDelta$, the both of the canonical squares,
\begin{equation}\label{54}
\xymatrix{
P\times A^{\times p} \times Q  \ar[r]   \ar[d] 
&
\bBar(P,A,Q)  \ar[d]
&&
A^{\times p} \times Q  \ar[r]   \ar[d] 
&
\bBar(\ast,A,Q)  \ar[d]
\\
P\times A^{\times p} \ar[r]
&
\bBar(P,A,\ast)
&
\text{ and }
&
A^{\times p} \ar[r]
&
\bBar(\ast,A,\ast)
~,
}
\end{equation}
are pullbacks.
We show as much for each such left square; each such right square follows via an identical argument (replacing $P$ with $\ast$).
Consider the natural transformation 
\begin{equation}\label{55}
\xymatrix{
(\bDelta^{\op})^{\tr}  \ar@(u,u)[rr]^-{\bBar_\bullet(P,A,Q) \to \bBar(P,A,Q)}  \ar@(d,d)[rr]_-{\bBar_\bullet(P,A,\ast) \to \Bar(P,A,\ast)}
&
\Downarrow
&
\cS
}
\end{equation}
between colimit diagrams.  
Using that $\cS$ is assumed an $\infty$-topos, Theorem~6.1.0.6 of~\cite{HTT} applies to this natural transformation~(\ref{55}) between colimit diagrams.
The results is that each left square in~(\ref{54}) is a pullback provided, for each morphism $\rho\colon [p]\to [q]$ in $\bDelta$, the square
\[
\xymatrix{
P\times A^{\times q} \times Q  \ar[r]   \ar[d] 
&
P\times A^{\times p} \times Q  \ar[d]
\\
P\times A^{\times q} \ar[r]
&
P\times A^{\times p}
}
\]
pullback.
This is always the case for $\rho$ degenerate.
This is the case for $\rho$ an arbitrary face map if and only if $A$ acts invertibly on $Q$ and on $P$, as well as on itself by both left and right translation.
This is the case if and only if $A$ is grouplike.
Because $\cA$ is assumed an interval Koszul duality, Lemma~\ref{interval-koszul} gives that, in particular the augmented associative algebra $A$ in $\cS$ is a member of a Koszul duality.  
Using that $A$ is grouplike, Proposition~\ref{semi.simplify} implies that $A$ is indeed a member of a Koszul duality.

\end{proof}

\subsection{Atiyah duality and non-abelian Poincar\'e duality}\label{sec.duality}
We prove that the Poincar\'e/Koszul duality map is an equivalence for coefficients in a Koszul duality in a bicomplete Cartesian-sifted $\oo$-category.  
This immediately implies the classical Atiyah duality, as well as the non-abelian Poincar\'e duality of~\cite{HA}.

\begin{theorem}[Non-abelian Poincar\'e duality]\label{NAPD}
Let $\cA\colon \zdisk_n \longrightarrow \cS$ be a Koszul duality in a bicomplete Cartesian-sifted $\oo$-category $\cS$.  
For each  zero-pointed $n$-manifold $M_\ast$, the Poincar\'e/Koszul duality map
\[
\int_{M_\ast} \cA_+ \xra{~(\ref{PD-map})~} \int^{M_\ast} \cA^+
\]
is an equivalence.

\end{theorem}

\begin{proof}
Consider the full $\infty$-subcategory $\cM \subset \ZMfld_n$ consisting of all zero-pointed $n$-manifolds $M_\ast$ for which the Poincar\'e/Koszul duality map is an equivalence.  
Being bicomplete Cartesian-sifted, the Cartesian symmetric monoidal $\infty$-category $\cS$ is both $\ot$-sifted cocomplete and $\ot$-cosifted complete.  
It follows using Theorem~\ref{fact-functor} that $\cM$ is closed under the formation of wedge sum; in other words, $\cM\subset \ZMfld_n$ is a symmetric monoidal $\oo$-subcategory.
Precisely because $\cA$ is a Koszul duality, this $\cM$ collection contains the objects of $\ZDisk_n$, each of which is a finite wedge sum of $\RR^n_+$ and $(\RR^n)_+$. Consequently, from the very definition of \emph{finitary}, once we show this collection $\cM$ is closed under collar-gluings, the argument is complete.  
After Observation~\ref{cosifted-complete}, and Theorem~\ref{reduced-to-cones}, it is enough to prove that 
\[
M_\ast ~\mapsto~\int_{M_\ast} \cA_+
\]
satisfies $\ot$-coexcision.  

Choose a conical smoothing of a conically finite zero-pointed $n$-manifold $M_\ast$.  
Choose a weakly constructible bundle $M_\ast \xra{f} [\text{-}1,1]$ so that $f$ is constant in a neighborhood of $\ast\in M_\ast$.  
Because $\cS$ is bicomplete Cartesian-sifted, Corollary~\ref{reduced-excision} ensures $\ot$-excision for factorization homology and $\ot$-coexcision for factorization cohomology.
With this, the naturality of the Poincar\'e/Koszul duality map gives the diagram in $\cS$:
\[
\xymatrix{
\scriptstyle
\int_{M_\ast} \cA_+  \ar[rrr]^-{(\ref{PD-map})}
&&&
\scriptstyle
\int^{M_\ast} \cA^+    \ar[d]^-{\simeq}
\\
\scriptstyle
\int_{f^{\text{-}1}[\text{-}1,1)_{M_\ast}} \cA_+\underset{\int_{f^{\text{-}1}(\text{-}1,1)_{M_\ast}} \cA_+}\bigotimes \int_{f^{\text{-}1}(\text{-}1,1]_{M_\ast}} \cA_+  \ar[rrr]^-{(\ref{PD-map})}  \ar[u]^-{\simeq}
&&&
\scriptstyle
\int^{f^{\text{-}1}[\text{-}1,1)^{M_\ast}} \cA^+\overset{\int^{f^{\text{-}1}(\text{-}1,1)^{M_\ast}}\cA^+}\bigotimes \int^{f^{\text{-}1}(\text{-}1,1]^{M_\ast}} \cA^+
}
\] 
(here we have used the super- and sub-script notation for the zero-pointed manifolds of Observation~\ref{constructions}).
So we must show the bottom horizontal map is an equivalence.  
For this, we will apply Lemma~\ref{interval-PD}.

In~\cite{aft2} we established a pushforward formula for factorization homology; the pullback formula for factorization cohomology follows dually.  
In~\cite{aft2} we also showed that the $\infty$-category of $[-1,1]$-algebras is canonically identified as that of $\Disk_1^{\partial, \sf fr}$-algebras, and so likewise for their augmented versions, as well as their dual versions.  
Through these means, the pushforward $f_\ast \cA$ is canonically identified as a $\zdisk_1^{\partial, \sf fr}$-algebra in $\cS$.  
To apply Lemma~\ref{interval-PD}, we need only show that $f_\ast \cA$ is an interval Koszul duality.  From Lemma~\ref{interval-koszul}, this amounts to verifying that the canonical arrow
\begin{equation}\label{half-right}
\uno\underset{\int_{f^{\text{-}1}(\text{-}1,1)_{M_\ast}} \cA_+}\bigotimes \int_{f^{\text{-}1}(\text{-}1,1]_{M_\ast}} \cA_+~ \longrightarrow~ \int^{f^{\text{-}1}(\text{-}1,1]^{M_\ast}}\cA^+~,
\end{equation}
is an equivalence, and likewise for the five other terms presented in the conditions of the lemma.  
Through $\ot$-excision, we recognize the left hand side of this expression~(\ref{half-right}) as 
\[
\uno\underset{\int_{f^{\text{-}1}(\text{-}1,1)_{M_\ast}} \cA_+}\bigotimes \int_{f^{\text{-}1}(\text{-}1,1]_{M_\ast}} \cA_+~ \simeq~\int_{f^{\text{-}1}(\text{-}1,1]^{M_\ast}} \cA_+~.
\]
Through the negation relations of Observation \ref{constructions}, we recognize $\bigl(f^{\text{-}1}(\text{-}1,1]^{M_\ast}\bigr)^\neg = f^{\text{-}1}(\text{-}1,1]_{M_\ast}$.
And so, should the Poincar\'e/Koszul duality map be an equivalence for each of $f^{\text{-}1}(\text{-}1,1]_{M_\ast}$ and $f^{\text{-}1}[\text{-}1,1)_{M_\ast}$ and $f^{\text{-}1}(\text{-}1,1)_{M_\ast}$, then the Poincar\'e/Koszul duality map is an equivalence for $M_\ast$. Thus, $\cM$ is closed under the formation of collar-gluings, which completes the proof.

\end{proof}

The next result makes use of the following construction.
Let $Z\colon \BO(n)  \to \Spaces^{\ast/}$ and $G\colon \BO(n)\to \cS$ be functors from an $\infty$-groupoid.
We denote the composite functor
\[
Z\otimes G\colon \BO(n) \xra{\rm diagonal} \BO(n)\times \BO(n) \xra{Z\times G}\Spaces^{\ast/}\times \cS \xra{\otimes} \cS
\]
in which the rightmost arrow is tensoring with pointed spaces as in Notation~\ref{def.tensor}.
Its colimit is denoted
\[
Z\underset{\sO(n)}\otimes G~:=~\colim(\BO(n)\xra{Z\otimes G} \cS)~.
\]
We denote the composite functor
\[
\Map(Z,G)\colon \BO(n) \xra{\rm diagonal} \BO(n)\times \BO(n) \xra{Z\times G}(\Spaces^{\ast/})^{\op}\times \cS \xra{\Map(-,-)} \cS
\]
in which the rightmost arrow is cotensoring with pointed spaces as in Notation~\ref{def.tensor}.
Its limit is denoted
\[
\Map^{\sO(n)}(Z,G)~:=~\lim(\BO(n)\xra{\Map(Z,G)} \cS)~.
\]

\begin{cor}[Linear Poincar\'e duality]\label{PKD-linear}
Let $\cS$ be a stable presentable $\infty$-category.
Let $E,F\colon \BO(n)\to \cS$ be a pair of functors.  
Suppose there is an equivalence between functors $\BO(n) \to \cS$:
\[
(\RR^n)^+\otimes E~ \simeq ~F\qquad\text{ or equivalently }\qquad E~\simeq ~F^{(\RR^n)^+}~.
\]
For each zero-pointed $n$-manifold $M_\ast$, there is an equivalence in $\cS$:
\[
{\sf Fr}_{M_\ast} \underset{\sO(n)}\bigotimes E ~\simeq~\Map^{\sO(n)}\bigl({\sf Fr}_{M_\ast^\neg}, F\bigr)~.
\]
\end{cor}

\begin{proof}

Left Kan extension of $E$ along the canonical monomorphism $\BO(n) \hookrightarrow \Disk_{n,+}$ defines an augmented $n$-disk algebra in $\cS$, with respect to the direct sum monoidal structure. Since direct sum is a colimit, there is a natural equivalence
\[
\int_{M_\ast} E ~   \simeq ~  {\sf Fr}_{M_\ast} \underset{\sO(n)}\bigotimes E
\]
for every zero-pointed manifold $M_\ast$. Likewise, right Kan extension of $F$ define an augmented $n$-disk coalgebra in $\cS$, again with respect to the direct sum monoidal structure. Since direct sum is a limit, since $\cS$ is stable, there is a natural equivalence
\[
\int^{M_\ast} F~ \simeq ~  \Map^{\sO(n)}\bigl({\sf Fr}_{M_\ast^\neg}, F\bigr)
\]
for every zero-pointed manifold $M_\ast$. By condition (2) of Lemma \ref{equivalent-koszul}, the condition of the lemma exactly give that $E$ and $F$ form a Koszul duality. The result follows by non-abelian Poincar\'e duality, Theorem \ref{NAPD}.

\end{proof}

\begin{cor}[Atiyah duality]
Let $\ov{M}$ be a compact smooth $n$-manifold with boundary $\partial \ov{M} = \partial_L \amalg \partial_R$ which is partitioned by connected components.  
Denote $M_\ast := \ast \underset{\partial_L} \amalg (\ov{M}\smallsetminus \partial_R)$ and $M_\ast^\neg := \ast \underset{\partial_R} \amalg (\ov{M}\smallsetminus \partial_L)$.
There is an equivalence between spectra
\[
(M_\ast)^{-\tau_M}~\simeq~\SS^{M_\ast^\neg}
\]
between the Thom spectrum of the normal bundle of $M_\ast$ and the Spanier--Whitehead dual of the based space $M_\ast^\neg$.  

\end{cor}

\begin{proof}
We apply Corollary~\ref{PKD-linear} to the case that $\cS$ is the $\infty$-category of spectra and $F = \SS$ is the constant functor from $\BO(n)$ at the sphere spectrum.
So the functor $E \colon \BO(n) \xra{V^n\mapsto \SS^{V^+}} \Spectra$.  
To prove this corollary we must therefore establish these identifications among spectra
\begin{equation}\label{first.second}
(M_\ast)^{-\tau_M}~\simeq~ {\sf Fr}_{M_\ast} \underset{\sO(n)}\bigotimes E 
\qquad \text{ and }\qquad
\Map^{\sO(n)}\bigl({\sf Fr}_{M_\ast^\neg}, F\bigr) ~\simeq ~\SS^{M_\ast^\neg}~.
\end{equation}
Choose a smooth manifold $\ov{M}$ with compact boundary together with an isomorphism $\ast\underset{\partial \ov{M}} \amalg \ov{M} \cong M_\ast$ between pointed extensions of the interior $M$.  
We utilize Remark~\ref{rem.lastlabel51}.
The first identification in~(\ref{first.second}) is the concatenation of the following identifications:
\begin{eqnarray}
\nonumber
(M_\ast)^{-\tau_M}
&
\simeq 
&
\colim \bigl(  M \xra{\tau_{M}} \BO(n) \xra{V^n\mapsto \SS^{V^+}}\Spectra\bigr)  \underset{\colim \bigl(  \partial M \xra{(\tau_{M})_{|\partial \ov{M}}} \BO(n) \xra{V^n\mapsto \SS^{V^+}}\Spectra\bigr)} \bigoplus
0
\\
\nonumber
&
\simeq
&
{\sf Fr}_{M_\ast} \underset{\sO(n)}\bigotimes E ~.
\end{eqnarray}
The first identification is the definition of the Thom spectrum of the virtual negative of the tangent bundle.
The second identification is the definition of the reduced coend.

The latter identification in~(\ref{first.second}) is the concatenation of the following identifications:
\begin{eqnarray}
\nonumber
\Map^{\sO(n)}\bigl({\sf Fr}_{M_\ast^\neg}, F\bigr)
&
\simeq
&
\Map_\ast\Bigl(\colim \bigl(\BO(n) \xra{{\sf Fr}_{M_\ast^\neg}}\Spaces^{\ast/}\bigr),\SS\Bigr)
\\
\nonumber
&
\simeq
&
\Map_\ast\bigl(M_\ast^\neg,\SS\bigr)~=:~\SS^{M_\ast^\neg}~.  
\end{eqnarray}
The first identification makes use of the fact that the functor $F$ is constant at the sphere spectrum $\SS$, using the universal property of colimits.
The second identification follows from the canonical identification of the underlying based space of $M_\ast^\neg$ as the colimit:
\[
\colim  \bigl(\BO(n) \xra{{\sf Fr}_{M_\ast^\neg}} \Spaces^{\ast/}\bigr)~\simeq~M_\ast^\neg~.
\]
The final line is the definition of the Spanier--Whitehead dual, or linear dual.

\end{proof}

Here is an immediate corollary of Theorem~\ref{NAPD}, which is a gentle generalization of the non-abelian Poincar\'e duality of Lurie (see Theorem~5.5.6.6 of~\cite{HA}).
\begin{cor}[Poincar\'e/Koszul duality for $\infty$-topoi]\label{PKD-topoi}
Let $\cE$ be an $\infty$-topos.
Let $A$ be a grouplike $\Disk_n$-algebra in $\cE$.
Let $C \to \BO(n)$ be an $n$-connective morphism in $\cE$, equipped with a section.  
Let $M_\ast$ be a zero-pointed $n$-manifold.
There are canonical equivalences in $\cE$:
\[
\int_{M_\ast} A \xra{~\simeq~} \Map^{\sO(n)}\bigl( {\sf Fr}_{M_\ast^\neg}, \bBar^n A \bigr)\qquad \text{ and }\qquad  \int_{M_\ast} \Omega^n C \xra{~\simeq~}\Map_{/\BO(n)}\bigl(({\sf Fr}_{M_\ast^\neg})_{\sO(n)}, C\bigr)~.
\]
In particular, taking $\cS = \Spaces$ and $C = \BO(n)\times Z$ with $Z$ an $n$-connective pointed space, there is a canonical equivalence between spaces
\[
\int_{M_\ast} \Omega^n Z\xra{~\simeq~}\Map_\ast \bigl(M_\ast^\neg , Z\bigr)
\]
from reduced factorization homology to the based mapping space.

\end{cor}

\section{Appendix: making units final}\label{sec.appendix}
In this appendix we characterize some symmetric monoidal $\infty$-categories whose symmetric monoidal unit is final.  
We do this so as to give a construction for how to minimally modify certain symmetric monoidal $\infty$-categories to this effect -- this is phrased as a left adjoint construction.  
The main result here supports the proof of Proposition~\ref{M.+.correct}.

For this section, $\cX$ is a presentable $\infty$-category.

\subsection{Final objects in internal categories}
We give a definition of a category internal to $\cX$, and of a final object in such.
These developments are tailored just for the purposes of this article; specifically, for Example~\ref{X=ComSpaces}.

\begin{definition}\label{def.internals}
The $\infty$-category of \emph{categories internal to $\cX$} is the full $\infty$-subcategory 
\[
\Cat[\cX]~\subset~\Fun(\bDelta^{\op},\cX)~,\qquad \cC\mapsto \cC^{(\bullet)}~,
\]
consisting of those simplicial objects $\cC$ in $\cX$ that satisfy following conditions (compare with~\cite{rezk}):
\begin{enumerate}
\item {\bf Segal:} 
The functor $\cC\colon \bDelta^{\op}\to \cX$ carries (the opposite of) each pushout diagram in $\bDelta$ comprised of convex inclusions,
\[
\xymatrix{
K \ar[r]  \ar[d]
&
J \ar[d]
\\
I \ar[r]
&
L,
}
\]
to a pullback diagram in $\cX$.

\item {\bf Univalence:} 
The functor $\cC\colon \bDelta^{\op}\to \cX$ carries the (opposite of the) diagram in $\bdelta$
\[
\xymatrix{
&
\{1<3\}  \ar[r]  \ar[d]
&
\ast  \ar[dd]
\\
\{0<2\}  \ar[r] \ar[d]
&
\{0<1<2<3\} \ar[dr]
&
\\
\ast   \ar[rr]
&&
\ast
}
\]
to a limit diagram in $\cX$

\end{enumerate}

\end{definition}

Our next goal is to define, for each category $\cC\in \Cat[\cX]$, and each morphism $\ast \xra{c} \cC$ from the final category internal to $\cX$, a category $\cC_{/c}$ internal to $\cX$.
Consider the subcategory 
\[
\bDelta_+~ \subset~\bDelta
\]
consisting of the same objects and those order-preserving maps that preserve maxima.  
Notice the zero-object $+:=[0]\in \bDelta_+$.
Adjoining a maximum to each finite non-empty linearly ordered set defines a functor
\[
\tr\colon \bDelta \longrightarrow \bDelta_+~,\qquad [p]\mapsto [p]^\tr:=\{0<1<\dots<p<+\}~.
\]
This functor $\tr$ is left adjoint to the inclusion ${\sf inc}\colon \bDelta_+ \hookrightarrow \bDelta$. 
Therefore, for each $\infty$-category $\cX$, restriction along these adjoint functors defines an adjunction
\[
\tr^\ast \colon \Fun(\bDelta^{\op}_+,\cX) \rightleftarrows \Fun(\bDelta^{\op},\cX)\colon {\sf inc}^\ast~.
\]
Also, the inclusion of the initial object $!\colon \{+\}\to \bDelta_+$ is a left adjoint, thereby determining another adjunction 
\[
!^\ast \colon \Fun(\bDelta_+^{\op},\cX)
\rightleftarrows 
\cX \colon !_\ast~.
\]
The counit for the $(\tr^\ast,{\sf inc}^\ast)$-adjunction, and the unit for the $(!^\ast,!_\ast)$-adjunction, together define a functor
\begin{equation}\label{to.spans}
\Fun(\bDelta^{\op},\cX) 
\longrightarrow 
\Fun(\bDelta^{\op},\cX)^{(\bullet \la \bullet \to \bullet)}
\end{equation}
given by 
\[
\cC\mapsto   \Bigl(  \cC  \xla{\rm counit}  \tr^\ast {\sf inc}^\ast \cC  \xra{\rm unit} \tr^\ast !_\ast !^\ast {\sf inc}^\ast \cC\Bigr)~.
\]
Recognize the value $\tr^\ast !_\ast !^\ast {\sf inc}^\ast \cC$ as simply the constant functor $\bDelta^{\op} \xra{\cC^{(0)}} \cX$ at the object $\cC^{(0)}\in \cX$.  
Suggestively, we denote the values of the functor~(\ref{to.spans}) as
\[
\cC\mapsto \bigl(\cC \xla{{\sf ev}_s}\Ar(\cC)_{|\cC^\sim}\xra{{\sf ev}_t} \cC^\sim \bigr)~.
\]
Using that $\cX$ is presentable, and in particular admits base change, there results a functor
\begin{equation}\label{4}
\Fun(\bDelta^{\op},\cX)^{\ast/} \xra{(\ref{to.spans})} 
\bigl(\Fun(\bDelta^{\op},\cX)^{(\bullet \la \bullet \to \bullet)}\bigr)^{\ast/}
\xra{\rm base~change} \Ar\bigl(\Fun(\bDelta^{\op},\cX)\bigr)^{\ast/}~,
\end{equation}
given by
\[
(\ast \xra{c}\cC)\mapsto \Bigl(\cC_{/c} \xra{{\sf ev}_s} \cC\Bigr)~.
\]
-- here, we have used the suggestive notation:
\[
\cC_{/c}~:=~\Ar(\cC)_{|\cC^\sim}  \underset{\cC^\sim}\times \ast~.
\]

\begin{lemma}\label{slice.counts}
For each category $\cC$ internal to $\cX$, and each morphism $\ast \xra{c} \cC$, the simplicial object $\cC_{/c}$ in $\cX^{\ast/}$ is a category internal to $\cX^{\ast/}$.

\end{lemma}

\begin{proof}

Let $F\colon \cJ^{\tr} \to \bDelta$ be a functor from a right-cone on a category.
Restricting the composite functor~(\ref{4}) along $F$ gives a composite functor
\begin{equation}\label{6}
\Small
\Fun\bigl((\cJ^{\op})^{\tl},\cX^{\ast/}\bigr) \xra{(\ref{to.spans})} 
\Fun\bigl((\cJ^{\op})^{\tl},(\cX^{(\bullet \la \bullet \to \bullet)} )^{\ast/} \bigr)
\xra{\rm base~change} \Fun\bigl((\cJ^{\op})^{\tl},\Ar(\cX)^{\ast/}\bigr)
\xra{~{\sf ev}_s~}
\Fun \bigl( (\cJ^{\op})^{\tl} , \cX^{\ast/}  \bigr) ~.
\end{equation}
By direct inspection, the value of each of these functors on $(\cJ^{\op})^{\tl}$-points that are limit diagrams are again $(\cJ^{\op})^{\tl}$-points that are limit diagrams.  
Now, take $F$ to be a Segal diagram, or the univalence diagram.
Let $\cC$ be a category internal to $\cX$; let $\ast \xra{c} \cC$ be a morphism from the final category internal to $\cX$.
By definition of a category internal to $\cX$, the $(\cJ^{\op})^{\tl}$-point of $\cX$, which is the composite functor $(\cJ^{\op})^{\tl} \to \bDelta^{\op} \xra{(\ast \xra{c} \cC)} \cX^{\ast/}$, is a limit diagram.
This composite functor~(\ref{6}) thus carries this $(\cJ^{\op})^{\tl}$-point in $\cX^{\ast/}$ to a $(\cJ^{\op})^{\tl}$-point in $\cX$ that is again a limit diagram.
This is to say that $\cC_{/c}$ satisfies the Segal and univalence conditions.

\end{proof}

\begin{definition}\label{internal.final}
The full $\infty$-subcategory 
\[
\Cat^{\sf final}[\cX]~\subset~\Cat[\cX]^{\ast/}
\]
consists of those pointed categories $(\ast \xra{c} \cC)$ internal to $\cX$ for which the canonical morphism $\cC_{/c} \to \cC$ is an equivalence.  
We refer to an object in $\Cat^{\sf final}[\cX]$ as a \emph{category internal to $\cX$ equipped with a final object}.  

\end{definition}

\begin{example}\label{X=Spaces}
It follows from Rezk's work (\cite{rezk}) that there is a canonical identification between $\infty$-categories
\[
\Cat_\infty~\simeq~\Cat[\Spaces]
\]
between $\infty$-categories and categories internal to the $\infty$-category $\Spaces$.
Thereafter follows a canonical identification between $\infty$-categories
\[
\Cat_\infty^{\sf final}~\simeq~\Cat^{\sf final}[\Spaces]
\]
from $\infty$-categories equipped with a final object, and functors between such that preserve final objects.  

\end{example}

\begin{example}\label{X=ComSpaces}

Through Example~\ref{X=Spaces}, there is a canonical identification
\[
\CAlg(\Cat_\infty^\times)~\simeq~\Cat[\CAlg(\Spaces^\times)]
\]
between the $\infty$-category of symmetric monoidal $\infty$-categories and that of categories internal to commutative monoids in $\Spaces$.  
Likewise, there is a canonical identification
\[
\CAlg(\Cat_\infty^{\times})~\supset~\CAlg^{\uno = \ast}(\Cat_\infty^\times)~\simeq~\Cat^{\sf final}[\CAlg(\Spaces^\times)]
\]
from the full $\infty$-subcategory of symmetric monoidal $\infty$-categories for which the symmetric monoidal unit is final. 

\end{example}

In the next result, we regard $\cX^{\ast/}$ as the pointed $\infty$-category $\ast \xra{\{\ast \xra{=}\ast\}} \cX^{\ast/}$ selecting its final object, which exists if $\cX$ admits finite limits; we also regard $\bDelta_+^{\op}$ as the pointed $\infty$-category $\ast \xra{\{[0]\}} \bDelta_+^{\op}$.
\begin{lemma}\label{final.+}
Let $\cX$ be an $\infty$-category with finite limits.
There is a pullback diagram of $\infty$-categories:
\[
\xymatrix{
\Cat^{\sf final}[\cX]  \ar[rr]  \ar[d]
&&
\Cat[\cX^{\ast/}]     \ar@{^{(}->}[d]
\\
\Fun^{\ast/}(\bDelta_+^{\op},\cX^{\ast/})  \ar[rr]^-{\tr^\ast}
&&
\Fun(\bDelta^{\op},\cX^{\ast/}).
}
\]
In particular, a category $\cC$ internal to $\cX$ has a final object if and only if its associated simplicial object $\cC^{(\bullet)}\colon \bDelta^{\op}\to \cX$ admits an extension along $\bDelta \xra{\tr} \bDelta_+$ whose value on $[0]$ is final.

\end{lemma}

\begin{proof}
The base change functor in~(\ref{4}) restricts as the base change functor 
\[
{\sf b}\colon  \Fun^{\ast/}(\bDelta_+^{\op} , \cX^{\ast/})  \xra{~\rm base~change~}  \Fun^{\ast/}(\bDelta_+^{\op} , \cX^{\ast/})
~,\qquad
(\ast \to \w{\cC}) \mapsto \bigl( [p]\mapsto \ast \underset{ \w{\cC}^{(p)}} \times \w{\cC}^{(p)} \bigr)~.
\]
Consider the adjunction between $\infty$-categories:
\begin{equation}\label{30}
\tr^\ast \colon \Fun^{\ast/}(\bDelta_+^{\op} , \cX^{\ast/}) 
~\rightleftarrows~
\Fun(\bDelta^{\op} , \cX^{\ast/})\colon {\sf b}\circ {\sf inc}^\ast~.
\end{equation}
This left adjoint evaluates on a pointed functor $\w{\cC} \colon \bDelta_+^{\op} \to \cX^{\ast/}$ as the functor
\[
\tr^\ast (\w{\cC}) \colon [p] \mapsto \w{\cC}([p]^{\tr})~.
\]
This right adjoint evaluates on a simplicial object $(\ast \xra{c}\cC)\colon \bDelta^{\op} \to \cX^{\ast/}$ as the functor
\[
{\sf b}\circ {\sf inc}^\ast(\cC)\colon I^{\tr}   \mapsto   \cC(I^{\tr}) \underset{\cC(\{\infty\})} \times \ast  ~.
\]
which is indeed pointed.  
The unit for this adjunction evaluates on a pointed functor $\w{\cC}\colon \bDelta_+^{\op} \to \cX^{\ast/}$ as the natural transformation between pointed functors whose value on $I^{\tr}\in \bDelta_+$ is the morphism in $\cX^{\ast/}$,
\begin{equation}\label{32}
{\rm unit}\colon \w{\cC}(I^{\tr}) \longrightarrow    \w{\cC}\bigl( (I^{\tr})^{\tr'}\bigr) \underset{ \w{\cC}(\{\infty\}^{\tr'})} \times  \ast~,
\end{equation}
induced by the canonical morphism $(I^{\tr})^{\tr'} \to I^{\tr}$ in $\bDelta_+$ that identifies the two cone-points.  
The counit for this adjunction evaluates on a pointed simplicial object $(\ast \xra{c}\cC)\colon \bDelta^{\op} \to \cX^{\ast/}$ as the natural transformation between pointed simplicial objects whose value on $[p] \in \bDelta$ is the morphism in $\cX^{\ast/}$,
\begin{equation}\label{33}
{\rm counit}\colon \cC_{/c}([p]) := \cC([p]^{\tr}) \underset{\cC(\{\infty\})} \times \ast   \longrightarrow  \cC[p]
~,
\end{equation}
induced by the canonical morphism $[p] \to [p]^{\tr}$ in $\bDelta$ whose image is all but the adjoined maximum.

Now, consider the pullback $\infty$-category
\[
\xymatrix{
\Fun^{\ast/}( \bDelta_+^{\op} , \cX^{\ast/} )_{|\Cat[\cX^{\ast/}]}  \ar[rr]  \ar[d]
&&
\Cat[\cX^{\ast/}]    \ar[d]
\\
\Fun^{\ast/}( \bDelta_+^{\op} , \cX^{\ast/} )    \ar[rr]^-{\tr^\ast}
&&
\Fun( \bDelta^{\op} , \cX^{\ast/} )   .
}
\]
By direct inspection, should the simplicial object $(\ast \xra{c}\cC) \colon \bDelta^{\op} \to \cX^{\ast/}$ be a category internal to $\cX^{\ast/}$, then so is this value $(\ast \xra{c=c} \cC_{/c})$.  
We conclude that the adjunction~(\ref{30}) restricts as an adjunction:
\begin{equation}\label{31}
\tr^\ast \colon \Fun^{\ast/}( \bDelta_+^{\op} , \cX^{\ast/} )_{|\Cat[\cX^{\ast/}]} 
~\rightleftarrows~
\Cat[\cX^{\ast/}]    \colon {\sf b}\circ {\sf inc}^\ast~.
\end{equation}

Let $\w{\cC}$ be an object of $\Fun^{\ast/}( \bDelta_+^{\op} , \cX^{\ast/} )_{|\Cat[\cX^{\ast/}]} $.
Inspecting~(\ref{32}), the unit transformation of the adjunction~(\ref{31}) on $\w{\cC}$ evaluates on an object $I^{\tr} \in \bDelta_+$ in a way that canonically fits into a commutative diagram in $\cX^{\ast/}$:
\[
\xymatrix{
\w{\cC}(I^{\tr})  \ar[rr]^-{\rm unit~(\ref{32})}   \ar[drrrr]_-{=}
&&
\w{\cC}\bigl( (I^{\tr})^{\tr'}\bigr) \underset{ \w{\cC}(\{\infty\}^{\tr'})} \times  \ast    \ar[rr]
&&
\w{\cC} (  I^{\tr} )\underset{\w{\cC}(\{\infty\})} \times \w{\cC}(\{\infty\}^{\tr'}) \underset{ \w{\cC}(\{\infty\}^{\tr'})} \times  \ast    \ar[d]
\\
&&&&
\w{\cC} (  I^{\tr} ) .
}
\]
The downward morphism is an equivalence because of cancelation in pullbacks and because, by definition, the value of the functor $\w{\cC}$ is pointed: $\w{\cC}(\ast) \xra{\simeq} \ast$.
The right horizontal morphism is an equivalence precisly because the simplicial object $\tr^{\ast}(\w{\cC})$  in $\cX^{\ast/}$ is assumed a category object, and in particular it is Segal.  
We conclude from the 2-of-3 properties for equivalences in the $\infty$-category $\cX^{\ast/}$ that the unit transformation of the adjunction~(\ref{31}) is by equivalences.  
Therefore, the left adjoint in the adjunction~(\ref{31}) is fully-faithful.

The adjunction~(\ref{31}) determines, for each object $\w{\cC}  \in \Fun^{\ast/}( \bDelta_+^{\op} , \cX^{\ast/} )_{|\Cat[\cX^{\ast/}]}$, a commutative triangle among simplicial objects of $\cX^{\ast/}$:
\[
\xymatrix{
\tr^\ast(\w{\cC})     \ar[rr]^-=  \ar[dr]_-{\tr^\ast \circ {\rm unit}}
&&
\tr^\ast (\w{\cC})
\\
&
\tr^\ast \circ {\sf b}\circ {\sf inc}^\ast \circ \tr^\ast (\w{\cC})  \ar[ur]_-{{\rm counit} \circ \tr^\ast}
&
.
}
\]
Argued above is that this unit transformation is by equivalences.  
We conclude from the 2-of-3 properties for equivalences in $\infty$-categories that the uprightward arrow in the above triangle is an equivalence.  
Inspecting the Definition~\ref{internal.final} of the full $\infty$-category $\Cat^{\sf final}[\cX]\subset \Cat[\cX^{\ast/}]$, we conclude that the left adjoint in the adjunction~(\ref{31}) takes values in the full $\infty$-subcategory $\Cat^{\sf final}[\cX]\subset \Cat[\cX^{\ast/}]$.  
So the adjunction~(\ref{31}) restricts as an adjunction
\begin{equation}\label{34}
\tr^\ast \colon \Fun^{\ast/}( \bDelta_+^{\op} , \cX^{\ast/} )_{|\Cat[\cX^{\ast/}]} 
~\rightleftarrows~
\Cat^{\sf final}[\cX^{\ast/}]    \colon {\sf b}\circ {\sf inc}^\ast~.
\end{equation}
A further inspection of the Definition~\ref{internal.final} of the full $\infty$-subcategory $\Cat^{\sf final}[\cX]\subset \Cat[\cX^{\ast/}]$ reveals that the counit of this adjunction~(\ref{34}) is by equivalences.  
With both the unit and the counit of the adjunction~(\ref{34}) being by equivalences, we conclude that this adjunction~(\ref{34}) is an equivalence between $\infty$-categories, as desired.

\end{proof}

Presentability of $\cX$ accommodates the adjunction
\[
\tr_!\colon  \Fun(\bDelta^{\op},\cX)
~  \rightleftarrows ~
\Fun(\bDelta_+^{\op},\cX)  \colon \tr^\ast
\]
with right adjoint given by restriction along $\tr^{\op}$ and with left adjoint given by left Kan extension along $\tr^{\op}$.

\begin{lemma}\label{tr.!.calc}
Let $\cX$ be an $\infty$-category that admits finite coproducts, and let $\cM$ be a category internal to $\cX$.
The value of the endofunctor $\tr^\ast \tr_!$ on $\cM$ evaluates on objects as
\[
\tr^\ast \tr_!(\cM)\colon \bDelta^{\op}\ni [p]\mapsto |\cM| \amalg \underset{0\leq i \leq p}\coprod \cM[i]\in \cX
\]
where $|\cM|:=\colim(\bDelta^{\op} \xra{\cM}\cX)$ is the colimit;
and on a convex inclusion $\sigma\colon \{k<\dots<\ell\}\to [p]$ in $\bDelta$ as the canonical morphism in $\cX$
\[
\sigma^\ast\colon  |\cM| \amalg \underset{0\leq i \leq p }\coprod \cM[i] \simeq \bigl(|\cM| \amalg \underset{0\leq i<k {\rm ~or~}\ell<i\leq p}\coprod \cM[i]\bigr) \amalg \bigl(\underset{k\leq i \leq \ell}\coprod\cM[i]\bigr)  \longrightarrow  |\cM|\amalg  \underset{k\leq i \leq \ell} \coprod \cM[i]~.
\]

\end{lemma}

\begin{proof}
The value of the left Kan extension $\tr_!$ on $\cM$ evaluates as the colimit
\[
\tr_!(\cM)\colon \bDelta_+^{\op}\ni [p]^{\tr} \mapsto \colim\bigl({\bDelta^{\op}}_{/[p]^{\tr}} \to \bDelta^{\op}\xra{\cM} \cX\bigr)\in \cX~;
\]
here, $p\geq -1$ and it is understood that $[p]^{\tr} = [0]$ if $p=-1$.  
Let $p\geq 0$.  
Consider the full subcategory 
\begin{equation}\label{discrete.sub}
\sC_p\hookrightarrow {\bDelta^{\op}}_{/[p]^{\tr}}
\end{equation}
consisting of those objects $([p]^\tr \xra{c} [q]^\tr)$ for which the preimage $c^{-1}([q]) = \emptyset$ is empty or the restriction $c_|\colon c^{-1}([q]) \to [q]$ is an isomorphism.  
Notice the functor 
\[
\sC_p \to \{0,1,\dots,p,+\}~,\qquad ([p]^\tr \xra{c}[q]^{\tr})\mapsto {\sf Min}\{c^{-1}(+)\subset [p]^{\tr}\}
\]
to a finite set, regarded as a discrete category.
The fiber of this functor over each of $i=1,2,\dots,p,+$ is a terminal category, whereas the fiber of this functor over $0$ is identified as $\bDelta^{\op}$.  
In summary, there is an isomorphism between categories
\begin{equation}\label{C.discrete}
\sC_p~\cong~   \bDelta^{\op} \amalg \bigl\{([p]^\tr \xra{c} [i]^\tr) \mid 0\leq i \leq p\bigr\}~.  
\end{equation}
Now, for each object $([p]^\tr \xra{f} [q]^\tr)$ in ${\bDelta^{\op}}_{/[p]^{\tr}}$, the undercategory $\sC_p^{f/}$ has an initial object, as we name now.
Provided $f^{-1}([q])\neq\emptyset$ is non-empty, this initial object is $([p]^\tr \xra{c}(f^{-1}([q]))^{\tr} \xra{f_|} [q]^{\tr})$ where $c$ is the morphism in $\bDelta$ characterized by declaring the composite morphism $f^{-1}([q]) \to [p]^{\tr}  \xra{c} (f^{-1}([q]))^{\tr}$ to be the standard inclusion.
If $f^{-1}([q]) = \emptyset$, this initial object is $([p]^\tr \xra{+} [q]^\tr \xra{\id} [q]^\tr)$.  
It follows that the functor~(\ref{discrete.sub}) is a right adjoint in a localization.
In particular, the functor~(\ref{discrete.sub}) is final.

The established finality of~(\ref{discrete.sub}), together with the identification~(\ref{C.discrete}), identifies the values
\[
\tr^\ast\tr_!(\cM)\colon \bDelta^{\op}\ni [p]\mapsto \colim\bigl(\sC_p \hookrightarrow {\bDelta^{\op}}_{/[p]^{\tr}}\to \bDelta^{\op} \xra{\cM} \cX\bigr)\simeq |\cM|\amalg  \underset{0\leq i \leq p}\coprod \cM[i]\in \cX
\]
as coproducts, 
as desired.  
The asserted values of $\tr^\ast \tr_!(\cM)$ on convex inclusions follows directly by inspection.

\end{proof}

\subsection{Making units final}

We observe some facts about the $\infty$-category $\CAlg(\Spaces^\times)$ of commutative monoids in $\Spaces$.
\begin{observation}\label{com.facts}
\begin{enumerate}
\item[]

\item
The $\infty$-category $\CAlg(\Spaces^\times)$ is presentable (\S3.2.2 \& \S3.2.3 of~\cite{HA}).

\item 
The $\infty$-category $\CAlg(\Spaces^\times)$ has a zero-object, which is the unique commutative algebra structure on the terminal space $\ast$.

\item
Because $\CAlg(\Spaces^\times)$ has a zero-object, for each functor $I\to \CAlg(\Spaces^\times)$ from a finite set, there is a canonical morphism in $\CAlg(\Spaces^\times)$
\[
\underset{i\in I} \coprod X_i \longrightarrow \underset{i\in I} \prod X_i
\]
from the $I$-indexed coproduct to the $I$-indexed product.
Proposition~3.2.4.7 of~\cite{HA} implies that this canonical morphism is in fact an equivalence.

\item
The forgetful functor $\CAlg(\Spaces^\times) \to \Spaces$ preserves limits (\S3.2.2 of~\cite{HA}).

\end{enumerate}

\end{observation}

\begin{definition}\label{def.disjunctive}
A symmetric monoidal $\infty$-category $\cM$ is \emph{disjunctive} if the following two conditions are satisfied.
\begin{itemize}
\item
The unit $\uno = \emptyset$ is initial.

\item 
For each pair of objects $X,Y\in \cM$, the tensor product functor
\[
\otimes \colon \cM_{/X}\times \cM_{/Y} \xra{~\simeq~} \cM_{/X\otimes Y}
\]
is an equivalence between $\infty$-categories.  
\end{itemize}

\end{definition}

\begin{example}\label{disk.disjunctive}

For each right fibration $\cB\to \Bsc$, both of the symmetric monoidal $\infty$-categories $\Disk(\cB)$ and $\Mfld(\cB)$ are disjunctive.
In particular, both of the symmetric monoidal $\infty$-categories $\Disk_n$ and $\Mfld_n$ are disjunctive.

\end{example}

\begin{observation}\label{disjunctive.contractible}
Let $\cM$ be a disjunctive symmetric monoidal $\infty$-category.
The unit being initial is equivalent to the unit morphism $\ast \xra{\uno} \cM$ being a symmetric monoidal left adjoint.
It follows that the colimit $|\cM|:=\colim(\bDelta^{\op} \xra{\cM^{(\bullet)}} \CAlg(\Spaces^\times))\simeq \ast$ is equivalent to the zero-object in $\CAlg(\Spaces^\times)$, which is the unique commutative algebra structure on the terminal space.

\end{observation}

\begin{lemma}\label{M.+}
For each disjunctive symmetric monoidal $\infty$-category $\cM$, the simplicial commutative monoid
\[
\tr^\ast \tr_!(\cM)\colon \bDelta^{\op} \longrightarrow \CAlg(\Spaces^\times)
\]
is a symmetric monoidal $\infty$-category whose unit is final.

\end{lemma}

\begin{proof}
In this proof, we denote $\cM_+:=\tr^\ast \tr_!(\cM)$.  

In light of Lemma~\ref{final.+}, we need only verify that $\cM_+:=\tr^\ast\tr_!(\cM)$ satisfies the Segal and univalence conditions, and that the value $\tr_!(\cM) \simeq \ast$ is terminal.  
The calculation of Lemma~\ref{tr.!.calc} makes the univalence condition immediate, because $\cM$ satisfies the univalence condition.
We now prove the Segal condition.
Let $[p]\in \bDelta$ be an object.  
We must show that the diagram of commutative monoids in $\Spaces$
\[
\xymatrix{
\cM_+[p]  \ar[rr]  \ar[d]
&&
\cM_+\{p-1<\dots<p\}  \ar[d]
\\
\cM_+\{0<\dots<p-1\}  \ar[rr]
&&
\cM_+\{p-1\}
}
\]
is a pullback.
Importing the calculation of Lemma~\ref{tr.!.calc}, this diagram becomes
\begin{equation}\label{among.coprods}
\xymatrix{
\underset{0\leq i \leq p}\coprod \cM[i]  \ar[rr]  \ar[d]
&&
\cM\{p-1<p\} \amalg \cM\{p\}   \ar[d]
\\
\underset{0\leq i \leq p-1}\coprod \cM[i]   \ar[rr]
&&
\cM\{p-1\}
}
\end{equation}
-- here we have used that $|\cM| \simeq \ast$ is a zero-object in $\CAlg(\Spaces^\times)$ (Observation~\ref{disjunctive.contractible}).  
The facts reported in Observation~\ref{com.facts} imply that the diagram~(\ref{among.coprods}) is a pullback if and only if the solid diagram of underlying spaces
\begin{equation}\label{round.2}
\Small
\xymatrix{
\underset{0\leq i \leq p}\prod \cM[i]  \ar[rr]^-{\sf pr}  \ar[d]_-{(\id\times {\sf ev}_{[p-1]})\times \id}
&&
\cM[p] \times \cM[p-1]  \ar[rr]^-{{\sf ev}_{\{p-1<p\}}\times {\sf ev}_{p-1}}      \ar@{-->}[d]^-{{\sf ev}_{[p-1]}\times \id}
&&
\cM\{p-1<p\} \times \cM\{p-1\}   \ar[d]_-{{\sf ev}_{p-1}\times {\sf id}}
\\
\bigl( \cM[p-1]\times \cM[p-1] \bigr)  \times    \bigl(\underset{0\leq i <p-1}\prod \cM[i]\bigr)      \ar[d]_-{\otimes \times \id}
&&
\cM[p-1]\times \cM[p-1]        \ar@{-->}[d]_-{\otimes}    \ar@{-->}[rr]^-{{\sf ev}_{p-1}\times {\sf ev}_{p-1}}
&&
\cM\{p-1\}\times\cM\{p-1\}  \ar[d]^-{\otimes}
\\
\underset{0\leq i \leq p-1}\prod \cM[i]   \ar[rr]^-{\sf pr}
&&
\cM[p-1]  \ar[rr]^-{{\sf ev}_{p-1}}
&&
\cM\{p-1\}
}
\end{equation}
is a pullback -- here, we expanded each map as the composition from Lemma~\ref{tr.!.calc}.
Note the indicated fillers in this diagram of spaces.
By direct inspection, the left square in~(\ref{round.2}) is a pullback.
Because $\cM$ satisfies the Segal condition, the upper right square in~(\ref{round.2}) is a pullback as well.
Consequently, the outside solid diagram~(\ref{round.2}) is a pullback provided the lower left square in the diagram is a pullback.
This lower right square is a pullback if and only if, for each point
$(X,Y)\in  \cM\{p-1\}\times \cM\{p-1\}]$, 
the map between fibers
\[
\bigl(\cM[p-1]\times \cM[p-1]\bigr)_{|(X,Y)} 
\longrightarrow 
\bigl(\cM[p-1] \bigr)_{|X\otimes Y}
\]
is an equivalence between spaces. 
Recognize this map as that between spaces of functors from $[p-2]$,
\[
\Cat_\infty\bigl([p-2],\cM_{/X}\times \cM_{/Y}\bigr) \xra{~\otimes~} \Cat_\infty\bigl([p-2],\cM_{/X\otimes Y}\bigr)~,
\]
induced by the tensor product functor for $\cM$.  
This map between spaces is an equivalence precisely because $\cM$ is disjunctive.
This completes this proof.

\end{proof}

\begin{notation}\label{def.M.+}
For $\cM$ a disjunctive symmetric monoidal $\infty$-category, we denote by $\cM_+$ the symmetric monoidal $\infty$-category $\tr^\ast \tr_!(\cM)$ of Lemma~\ref{M.+}.  

\end{notation}

\begin{remark}\label{M+.explicit}
Let $\cM$ be a disjunctive symmetric monoidal $\infty$-category.
The symmetric monoidal $\infty$-category $\cM_+$ has the following explicit, though partial, description.
The maximal symmetric monoidal $\infty$-subgroupoid is $(\cM_+)^\sim \simeq \cM^\sim$ is that of $\cM$.
The symmetric monoidal $\infty$-groupoid of morphisms is 
\[
\cM_+^{(1)}~\simeq~ \cM^{(1)}\times \cM^{(\{0\})}~.
\]
The source map is
\[
\cM_+^{(1)} \simeq  \cM^{(1)}\times \cM^{(\{0\})} \xra{{\sf ev}_0\times \id} \cM^{(\{0\})}\times \cM^{(\{0\})}\xra{\otimes}  \cM^{(\{0\})}~,\qquad (X\to Y,Z)\mapsto X\otimes Z~;
\]
the target map is
\[
\cM_+^{(1)} \simeq  \cM^{(1)}\times \cM^{(\{0\})} \xra{\sf pr} \cM^{(1)}\xra{{\sf ev}_1}   \cM^{(\{1\})}~,\qquad (X\to Y,Z)\mapsto Y~.
\]
In other words, for $X_+,Y_+\in \cM_+$ two objects, the space of morphisms in $\cM_+$ from $X_+$ to $Y_+$ is
\[
\cM_+(X_+,Y_+)~\simeq~\underset{U\otimes V\simeq X}\coprod \cM(U,Y)~,
\]
a colimit indexed by the maximal $\infty$-subgroupoid of the fiber of the tensor product functor $\cM\times \cM \xra{\ot}\cM$ over $X$.  

\end{remark}

Note that, for $\cM$ a disjunctive symmetric monoidal $\infty$-category, the unit of the $(\tr_!,\tr^\ast)$-adjunction is a symmetric monoidal functor
\begin{equation}\label{M.to.M+}
\cM\longrightarrow \cM_+~.
\end{equation}
\begin{prop}\label{M.+.adjoint}
Let $\cV$ be a symmetric monoidal $\infty$-category.
For each disjunctive symmetric monoidal $\infty$-category $\cM$, restriction along the symmetric monoidal functor $\cM \xra{(\ref{M.to.M+})} \cM_+$ defines an equivalence between $\infty$-categories
\[
\Fun^{\ot}(\cM_+,\cV) \longrightarrow \Fun^{\ot, \sf aug}(\cM,\cV)
\]
of symmetric monoidal functors.

\end{prop}

\begin{proof}
We begin by explaining the diagram of $\infty$-categories of symmetric monoidal functors
\[
\xymatrix{
&&
\Fun^\otimes(\cM_+,\cV_{/\uno}) \ar[rr]  \ar[d]_-{\simeq}
&&
\Fun^{\ot}(\cM,\cV_{/\uno})  \ar[d]^-{\simeq}
\\
\Fun^{\ot}(\cM_+,\cV)  
&&
\Fun^{\ot, \sf aug}(\cM_+,\cV) \ar[rr]   \ar[ll]^-{\simeq}
&&
\Fun^{\ot, \sf aug}(\cM,\cV).
}
\]
The functor $(\Cat_\infty^{\ot})^{\op} \to \Spaces$, given by $\cK \mapsto \Fun^{\ot,\sf aug}(\cK,\cV)$, is represented by a symmetric monoidal $\infty$-category $\cV_{/\uno}$, which is a canonical symmetric monoidal structure on the $\infty$-overcategory $\cV_{/\uno}$. 
This explains the vertical equivalences of $\infty$-categories in the above diagram.
The leftward forgetful functor is an equivalence because the symmetric monoidal unit of $\cM_+$ is terminal.  
This explains the diagram.

With this diagram, to prove the result it is enough to show that the top horizontal functor is an equivalence.
By construction, the symmetric monoidal unit of the symmetric monoidal $\infty$-category $\cV_{/\uno}$ is final.
We are therefore reduced to proving that restriction along $\cM\to \cM_+$ defines an equivalence between $\infty$-categories of symmetric monoidal functors to a symmetric monoidal $\infty$-category $\cV$ whose unit is final; we proceed with this assumption on $\cV$.

The functor $(\Cat_\infty^{\ot})^{\op} \to \Spaces$, given by $\cK \mapsto \Ar\bigl(\Fun^{\ot}(\cK,\cV)\bigr)$, is represented by a symmetric monoidal $\infty$-category $\Ar(\cV)$, which is a canonical symmetric monoidal structure on the $\infty$-category $\Ar(\cV)$ of morphisms in $\cV$.
Because the unit of $\cV$ is final, so too is the unit of $\Ar(\cV)$.
In this way, we are reduced to proving the result just on the level of maximal $\infty$-subgroupoids: 
\begin{itemize}
\item[~]
Under the assumption that the symmetric monoidal unit of $\cV$ is final, restriction along $\cM \to \cM_+\simeq \tr^\ast \tr_!(\cM)$ defines an equivalence
\begin{equation}\label{+.not+}
\Map^{\ot}(\cM_+,\cV) \longrightarrow \Map^{\ot}(\cM,\cV)
\end{equation}
between \emph{spaces} of symmetric monoidal functors.  
\end{itemize}

Through Lemma~\ref{final.+}, there is a functor $\w{\cV}\colon \bDelta_+^{\op} \to \CAlg(\Spaces^\times)$ and an equivalence $\cV \simeq \tr^\ast(\w{\cV})$.
We therefore fit the map~(\ref{+.not+}) into a commutative diagram of spaces of morphisms
\[
\xymatrix{
\Map\bigl(\tr^\ast\tr_!(\cM),\tr^\ast(\w{\cV})\bigr)  \ar[rr]^-{(\ref{+.not+})}
&&
\Map\bigl(\cM,\tr^\ast(\w{\cV})\bigr)  
\\
&
\Map\bigl(\tr_!(\cM),\w{\cV}\bigr)  \ar[ul]^-{\tr^\ast}_-{\simeq}  \ar[ur]^-{\simeq}_-{(\tr_!,\tr^\ast)\text{-adj}}
&
.
}
\]
The diagonal rightward arrow is an equivalence because of the $(\tr_!,\tr^\ast)$-adjunction.
The diagonal leftward arrow is an equivalence because the functor $\tr^\ast$ is fully-faithful on $\Cat^{\sf final}[\CAlg(\Spaces^\times)]$, as observed at the end of Example~\ref{X=ComSpaces}.
This completes the proof.

\end{proof}

\begin{example}\label{main.example}
Consider the disjunctive symmetric monoidal $\infty$-category $\Mfld_n$ from Example~\ref{disk.disjunctive}.  
Recall from Definition~\ref{Mfld.+} the symmetric monoidal $\infty$-category $\Mfld_{n,+}$ under $\Mfld_n$.
Through Proposition~\ref{M.+.adjoint}, there is a unique symmetric monoidal functor
\[
(\Mfld_n)_+ \longrightarrow \Mfld_{n,+}
\]
under $\Mfld_n$.  
Through Remark~\ref{M+.explicit}, this symmetric monoidal functor is an equivalence.  
Likewise, there is a canonical identification $(\Disk_n)_+ \simeq \Disk_{n,+}$ between symmetric monoidal $\infty$-categories under $\Disk_n$.

\end{example}

\end{document}